\documentclass[10pt]{article}
\usepackage[all,pdf]{xy}
\usepackage{amsmath}
\usepackage{amsfonts}
\usepackage{amssymb}
\usepackage{amsthm}
\usepackage{latexsym}
\usepackage{indentfirst}
\usepackage{bm}
\usepackage{graphicx}
\usepackage{epsfig}
\usepackage{mathrsfs}
\usepackage{float}
\usepackage{xcolor}

\usepackage{extarrows}

\usepackage{quiver}

\usepackage{enumerate}

\usepackage{algorithm}
\usepackage{algorithmicx}
\usepackage{algpseudocode}
\usepackage{ulem}

\newtheorem{theorem}{Theorem}[section]
\newtheorem{example}[theorem]{Example}
\newtheorem{definition}[theorem]{Definition}
\newtheorem{proposition}[theorem]{Proposition}
\newtheorem{lemma}[theorem]{Lemma}
\newtheorem{corollary}[theorem]{Corollary}
\newtheorem{remark}[theorem]{Remark}

\usepackage{fancyhdr}

\usepackage{tikz-cd}

\usepackage{tikz}
\newcommand*{\circled}[1]{\lower.7ex\hbox{\tikz\draw (0pt, 0pt)%
		circle (.5em) node {\makebox[1em][c]{\small #1}};}}
\usepackage{geometry}
\begin{document}
	
	\title{\bf\Large Generalized parallel paths method for computing the first Hochschild cohomology group with applications to Brauer graph algebras}
	\author{Yuming Liu$^a$ and Bohan Xing$^{a,*}$}
	\maketitle
	
	\renewcommand{\thefootnote}{\alph{footnote}}
	\setcounter{footnote}{-1} \footnote{\it{Mathematics Subject
			Classification(2020)}: 16E40, 16Gxx.}
	\renewcommand{\thefootnote}{\alph{footnote}}
	\setcounter{footnote}{-1} \footnote{\it{Keywords}: Algebraic Morse theory; Brauer graph algebra; First Hochschild cohomology group; Generalized parallel paths method; Two-sided Anick resolution.}
	\setcounter{footnote}{-1} \footnote{$^a$School of Mathematical Sciences, Laboratory of Mathematics and Complex Systems, Beijing Normal University,
		Beijing 100875,  P. R. China. E-mail: ymliu@bnu.edu.cn (Y.M. Liu); bhxing@mail.bnu.edu.cn (B.H. Xing).}
	\setcounter{footnote}{-1} \footnote{$^*$Corresponding author.}
	
	{\noindent\small{\bf Abstract:} We use algebraic Morse theory to generalize the parallel paths method for computing the first Hochschild cohomology group. As an application, we describe and compare the Lie structure of the first Hochschild cohomology groups of Brauer graph algebras and their associated graded algebras. }
		
		\section{Introduction}
		
		It is well-known that the Hochschild cohomology groups are important invariants of associative algebras under derived equivalences. The computation of Hochschild cohomology groups is heavily based on a two-sided projective resolution of a given algebra. The smaller of the size of this projective resolution, the more efficient the computation. For monomial algebras, Bardzell \cite{Bardzell} constructed minimal two-sided projective resolutions. Based on the minimal two-sided projective resolution, Strametz \cite{Strametz} introduced the parallel paths method to compute the first Hochschild cohomology group of a monomial algebra.
		
		One aim of the  present paper is to generalize Strametz's parallel paths method for computing the first Hochschild cohomology group from monomial algebras to arbitrary quiver algebras of the form $kQ/I$, where $Q$ is a finite quiver and $I$ is an ideal of the path algebra $kQ$ contained in $kQ_{\geq 2}$. Our idea is to use the two-sided Anick resolution, which is based on algebraic Morse theory, to replace Bardzell's minimal two-sided projective resolution.
		
		About fifteen years ago, algebraic Morse theory was developed by Kozlov \cite{K}, by Sk\"{o}ldberg \cite{ES}, and by J\"{o}llenbeck and Welker \cite{JW}, independently. Since then, this theory has been widely used in algebra; for further references on this direction, see for example the introduction in \cite{CLZ}. In particular, Sk{\"o}ldberg \cite{ES} applied this theory to construct the so-called two-sided Anick resolution from the reduced bar resolution of a non-commutative polynomial algebra. Chen, Liu and Zhou \cite{CLZ} generalized the two-sided Anick resolution from non-commutative polynomial algebras to algebras given by quivers with relations.
		
		For a monomial algebra $kQ/I$, the ideal $I$ has a minimal generating set $Z$ given by paths in $Q$, which is one of ingredients in Strametz's construction. For an arbitrary quiver algebra $kQ/I$, we use the Gr\"{o}bner basis of $I$ to replace the above set $Z$. In order to generalize Strametz's construction, we use the two-sided Anick resolution (which is for an arbitrary quiver algebra) to replace Bardzell's minimal two-sided projective resolution (which is only for a monomial algebra). Similar to Strametz's method in \cite{Strametz}, we give a method to compute the zeroth and the first Hochschild cohomology groups by using parallel paths.
			\begin{proposition} {\rm(see Proposition \ref{gen-parallel paths})}
			Let ${A}=kQ/I$ be a quiver algebra such that $I$ has a finite reduced Gr{\"o}bner basis $\mathcal{G}$. Then the zeroth and the first Hochschild cohomology groups of $A$ can be computed by the following complex
			\[\begin{tikzcd}
				0 & {k(Q_0//\mathcal{B})} & {k(Q_1//\mathcal{B})} & {k(\mathrm{Tip}(\mathcal{G})//\mathcal{B})} & \cdots
				\arrow[from=1-1, to=1-2]
				\arrow["{\psi_0}", from=1-2, to=1-3]
				\arrow["{\psi_1}", from=1-3, to=1-4]
				\arrow[from=1-4, to=1-5]
			\end{tikzcd}\]
			where the maps are given by
			$$	\begin{array}{*{5}{lllll}}
				\psi_0& :&k(Q_0//\mathcal{B}) &\rightarrow&k(Q_1//\mathcal{B}),\\
				&& (e,\gamma)&\mapsto&\sum_{\alpha\in Q_1e}(\alpha,\pi(\alpha\gamma))-\sum_{\beta\in eQ_1}(\beta,\pi(\gamma\beta));\\\\
				\psi_1&: &k(Q_1//\mathcal{B})&\rightarrow&k(\mathrm{Tip}(\mathcal{G})//\mathcal{B}),\\
				&&(\alpha,\gamma)&\mapsto&\sum_{g\in\mathcal{G}}\sum_{p\in \mathrm{Supp}(g)}c_g(p)\cdot(\mathrm{Tip}(g),\pi(p^{(\alpha,\gamma)})).
			\end{array}$$
			with $g=\sum_{p\in \mathrm{Supp}(g)}c_g(p)p$, $c_g(p)\in k$.
			
			In particular, we have $\mathrm{HH}^0({A})\cong \mathrm{Ker}\psi_0$, $\mathrm{HH}^1({A})\cong \mathrm{Ker}\psi_1/\mathrm{Im}\psi_0$. 
		\end{proposition}
		Moreover, we also describe the Lie algebra structure on the first Hochschild cohomology group as follows.
		\begin{theorem} {\rm(see Theorem \ref{gen-lie bracket})}
			The bracket $$[(\alpha,\gamma),(\beta,\varepsilon)]=(\beta,\pi(\varepsilon^{(\alpha,\gamma)}))-(\alpha,\pi(\gamma^{(\beta,\varepsilon)}))$$
			for all $(\alpha,\gamma),(\beta,\varepsilon)\in Q_1//\mathcal{B}$ induces a Lie algebra structure on $\mathrm{Ker}\psi_1/\mathrm{Im}\psi_0$, such that $\mathrm{HH}^1({A})$ and $\mathrm{Ker}\psi_1/\mathrm{Im}\psi_0$ are isomorphic as Lie algebras.
		\end{theorem}
		It should be noted that, Artenstein, Lanzilotta and Solotar \cite{ALS} recently studied the Hochschild cohomology of toupie algebras. The cochain complex  obtained in \cite[Section 3]{ALS} to compute $\mathrm{HH}^1$ for toupie algebras coincides with the cochain complex in our Proposition \ref{gen-parallel paths}.

		In Section 4, we apply our method to study the first Hochschild cohomology groups of Brauer graph algebras (abbr. BGAs). These algebras coincide with finite dimensional symmetric special biserial algebras (see for example in \cite{Sch}). When the ground field $k$ satisfies condition $(*)$ in Definition \ref{condition*}, we can describe explicitly a $k$-basis of the first Hochschild cohomology group $\mathrm{HH}^1(A)$ of a BGA  (Theorem \ref{gen-set of A}), and compare it with a $k$-basis of the first Hochschild cohomology group $\mathrm{HH}^1(gr(A))$ of the associated graded algebra (Theorem \ref{gen-set of grA}). Based on these results, we describe when $\mathrm{HH}^1(A)$ and $\mathrm{HH}^1(gr(A))$ are solvable.
		\begin{corollary} {\rm(see Corollary \ref{A-solvable})}
			Let $A$ be a BGA over a field of characteristic zero such that the corresponding Brauer graph $G$ is different from	
			$(\begin{tiny}	
				\begin{tikzcd}
					\bullet & \bullet
					\arrow[shift left=1, no head, from=1-1, to=1-2]
					\arrow[shift right=1, no head, from=1-1, to=1-2]
				\end{tikzcd}
			\end{tiny})$ (in this case, both vertices have multiplicity $1$). Then $\mathrm{HH}^1(A)$ is solvable.
		\end{corollary}
			\begin{theorem} {\rm(see Theorem \ref{grA-solvable})}
			Let $A$ be a BGA over a field of characteristic zero such that the corresponding Brauer graph $G$ is different from	
			$(\begin{tiny}	
				\begin{tikzcd}
					\bullet & \bullet
					\arrow[shift left=1, no head, from=1-1, to=1-2]
					\arrow[shift right=1, no head, from=1-1, to=1-2]
				\end{tikzcd}
			\end{tiny})$ (in this case, both vertices have multiplicity $1$) and $gr(A)$ the associated graded algebra of $A$. Then $\mathrm{HH}^1(gr(A))$ is solvable.
		\end{theorem}
		In particular, we construct an injection $i$ from $\mathrm{HH}^1(A)$ to $\mathrm{HH}^1(gr(A))$ as follows.
		\begin{theorem} {\rm(see Theorem \ref{inj-map})}
			Let $A$ be a BGA associated with a Brauer graph $G$, and $gr(A)$ the associated graded algebra of $A$. If $G\neq (v_S{-}v_L)$ with $m(v_L)>m(v_S)\geq 2$,  then there is a monomorphism $i$ from $\mathrm{HH}^1(A)$ to $\mathrm{HH}^1(gr(A))$ as Lie algebras.
		\end{theorem}
		 In Example \ref{not injection}, we show that if the assumptions of the above theorem are not satisfied, then $i$ is not a Lie algebra monomorphism in general. 
		
		 Denote the identity component of the outer automorphism group of $A$ by $\mathrm{Out}(A)^{\circ}$. The injection $i$ also gives the following result.
		 	\begin{corollary} {\rm(see Corollary \ref{diff2})}
		 	The difference between the dimension of $\mathrm{HH}^1(A)$ and of $\mathrm{HH}^1(gr(A))$ is equal to the difference between the rank of $\mathrm{Out}(A)^{\circ}$ and of $\mathrm{Out}(gr(A))^{\circ}$. In particular, it is also equal to the difference between their dimensions of the corresponding maximal dual fundamental groups.
		 \end{corollary}
		
		\medskip
		After we submitted this paper on arXiv, we noticed that Rubio y Degrassi, Schroll and Solotar have recently obtained similar results in \cite{RSS}. However, our generalized parallel paths method for computing $\mathrm{HH}^1$ is deduced from two-sided Anick resolutions using algebraic Morse theory, rather than directly using the Chouhy-Solotar projective resolution which is constructed in \cite{CS}. Moreover, we described explicitly a $k$-basis of $\mathrm{HH}^1(A)$ when $A$ is a BGA, and our comparison study on the first Hochschild cohomology groups between BGAs and their associated graded algebras is also new. Finally, we provide a counterexample (Example \ref{counter-example}) to Theorem 4.2 in \cite{RSS} in positive characteristic.
		
		\medskip
		\textbf{Outline.} \;In Section 2, we introduce some preliminaries and give some notations which we need throughout this paper. In particular, we will review the Gr{\"o}bner basis theory for path algebras and the two-sided Anick resolutions for quiver algebras based on algebraic Morse theory. In Section 3, we generalize the parallel paths method for computing $\mathrm{HH}^1$ from monomial algebras to general quiver algebras. In Sections 4 and 5, we apply the generalized parallel paths method to BGAs and their associated graded algebras. Moreover, we give a simple formula (Corollary \ref{dim(A-grA)}) for the difference between the dimensions of the first Hochschild cohomology groups of a BGA and its associated graded algebra.

		\section{Preliminaries}
		
		Throughout this paper we will concentrate on quiver algebras of the form $kQ/I$, where $k$ is a field, $Q$ is a finite quiver, $I$ is a two-sided ideal in the path algebra $kQ$. For each integer $n\geq 0$, we denote by $Q_n$ the set of all paths of length $n$ and by $Q_{\geq n}$ the set of all paths with length at
		least $n$. We shall assume that the ideal $I$ is contained in $kQ_{\geq 2}$ so that $kQ/I$ is not necessarily finite dimensional. We denote by $s(p)$ the source vertex of a path $p$ and by $t(p)$ its terminus vertex. We will write paths from right to left, for example, $p=\alpha_{n}\alpha_{n-1}\cdots\alpha_{1}$ is a path with starting arrow $\alpha_{1}$ and ending arrow $\alpha_{n}$.  
		The length of a path $p$ will be denoted by $l(p)$. Two paths $\varepsilon,\gamma$ of $Q$ are called parallel if $s(\varepsilon)=s(\gamma)$ and $t(\varepsilon)=t(\gamma)$. An element in $kQ$ is called uniform if it is a linear combination of parallel paths. If $X$ and $Y$ are sets of paths of $Q$, the set $X//Y$ of parallel paths is formed by the couples $(\varepsilon,\gamma)\in X\times Y$ such that $\varepsilon$ and $\gamma$ are parallel paths. For instance, $Q_0//Q_n$ is the set of oriented cycles of $Q$ of length $n$. We denote by $k(X//Y)$ the $k$-vector space generated by the set $X//Y$. For a subset $S$ of $k(X//Y)$, we denote by $\langle S\rangle$ the subspace of $k(X//Y)$ generated by $S$. By abuse of notation, for a subset $S$ of the algebra $kQ/I$, we also use $\langle S\rangle$ to denote the ideal generated by $S$.
		
		\subsection{Gr{\"o}bner bases of quiver algebras}
		
		Let $A=kQ/I$ be a quiver algebra where	$I$ is generated by a set of relations. In this subsection we recall from \cite{Green} the Gr{\"o}bner basis theory for the ideal $I$. Let us first introduce a special kind of well-order (that is, a total order such that every nonempty subset has a minimal element) on the basis $Q_{\geq 0}$ of the path algebra $kQ$. By \cite[Section 2.2.2]{Green}, a well-order $>$ on $Q_{\geq 0}$ is called admissible if it satisfies the following conditions where $p,q,r,s\in Q_{\geq 0}$:
		\begin{itemize}
			\item if $p<q$, then $pr<qr$ where both $pr\neq 0$ and $qr\neq 0$.
			
			\item  if $p<q$, then $sp<sq$ where both $sp\neq 0$ and $sq\neq 0$.
			
			\item  if $p=qr$, then $p\geq q$ and $p\geq r$.
		\end{itemize}
		
		Given a quiver $Q$ as above, there are ``natural'' admissible orders on $Q_{\geq 0}$. Here is one example:
		
		\begin{example}
			The (left) length-lexicographic order on $Q_{\geq 0}$:
			
			Order the vertices $Q_0=\{v_1,\cdots,v_n\}$ and arrows $Q_1=\{a_1,\cdots,a_m\}$ arbitrarily and set the vertices smaller than the arrows. Thus for example, let $$v_1<\cdots<v_n<a_1<\cdots<a_m.$$
			If $p$ and $q$ are paths of length at least $1$, set $p<q$ if $l(p)<l(q)$ or if $p=b_1\cdots b_r$ and $q=b_1'\cdots b_r'$ with $b_1,\cdots,b_r,b_1',\cdots,b_r'\in Q_1$ and for some $1\leq i\leq r$, $b_j=b_j'$ for $j<i$ and $b_i<b_i'$.
		\end{example}
		We now fix an admissible well-order $>$ on $Q_{\geq 0}$. For each $a\in kQ$, we have $a=\sum_{p\in Q_{\geq 0},\;\lambda_p\in k}\lambda_p p$ and write $\mathrm{Supp}(a)=\{p\mid \lambda_p\neq 0\}$. 
		
		We call $\mathrm{Tip}(a)=p$,  if $p\in \mathrm{Supp}(a)$ and $p'\leq p$ for all $p'\in \mathrm{Supp}(a)$. Then we denote the tip of a set $W\subseteq kQ$ by $\mathrm{Tip}(W)=\{\mathrm{Tip}(w)\;|\;w\in W\}$ and write $\mathrm{NonTip}(W):=Q_{\geq 0}\backslash \mathrm{Tip}(W)$. We also denote the coefficient of the tip of $a$ by $\mathrm{CTip}(a)$. In particular, we will use $\mathrm{Tip}(I)$ and $\mathrm{NonTip}(I)$ for the ideal $I$ of $kQ$.
		By \cite{Green}, there is a decomposition of vector spaces
		$$kQ=I\oplus \mathrm{Span}_{k}(\mathrm{NonTip}(I)).$$
		So $\mathrm{NonTip}(I)$ (modulo $I$) gives a ``monomial" $k$-basis of the quotient algebra $A=kQ/I$.
		
		\begin{definition} $(\cite[~Definition~2.4]{Green})$
			With the notations as above, a subset $\mathcal{G}$ of uniform elements in $I$ is a Gr{\"o}bner basis for the ideal $I$ with respect to the order $>$ if $$\langle \mathrm{Tip}(I)\rangle=\langle \mathrm{Tip}(\mathcal{G})\rangle,$$
			that is, $\mathrm{Tip}(I)$ and $\mathrm{Tip}(\mathcal{G})$ generate the same ideal in $kQ$.
		\end{definition}
		
		Indeed, note that $I=\langle \mathcal{G}\rangle$. By the discussion in \cite{Green}, we have a complete method to judge whether a set of generators of an ideal $I$ in $kQ$ is a Gr{\"o}bner basis, which is called the Termination Theorem. The idea is checking whether some special elements of the ideal $I$ are divisible by this basis, instead of to check all the elements in $I$.
		
		\begin{definition} $(\cite[~Definition~2.7]{Green})$
			Let $kQ$ be a path algebra, $>$ an admissible order on $Q_{\geq 0}$ and $f,g\in kQ$. Suppose $b,c\in Q_{\geq 0}$, such that
			\begin{itemize}
				\item $\mathrm{Tip}(f)c=b\mathrm{Tip}(g)$,
				\item $\mathrm{Tip}(f)\nmid b$, $\mathrm{Tip}(g)\nmid c$.
			\end{itemize}
			Then the overlap relation of $f$ and $g$ by $b,c$ is
			$$o(f,g,b,c)=(\mathrm{CTip}(f))^{-1}\cdot fc-(\mathrm{CTip}(g))^{-1}\cdot bg.$$
			Note that $\mathrm{Tip}(o(f,g,b,c))<\mathrm{Tip}(f)c=b\mathrm{Tip}(g)$.
		\end{definition}
		
		\begin{theorem} $(\cite[~Theorem~2.3]{Green})$\label{Termin.Thm}
			Let $kQ$ be a path algebra, $>$ an admissible order on $Q_{\geq 0}$, $\mathcal{G}$ a set of uniform elements of $kQ$. Suppose for every overlap relation, we have $$o(g_1,g_2,p,q)\Rightarrow_{\mathcal{G}} 0,$$
			which means that $o(g_1,g_2,p,q)$ can be divided by $\mathcal{G}$ (see for example in \cite[Section 2.3]{Green}), with $g_1,g_2\in\mathcal{G}$ and $p,q\in Q_{\geq 0}$. Then $\mathcal{G}$ is a Gr{\"o}bner basis of the ideal generated by $\mathcal{G}$, that is, $\langle \mathcal{G}\rangle$.
		\end{theorem}
		
		\begin{definition} $(\mathrm{c.f.}~\cite[~Section~4]{CLZ})$ A Gr\"{o}bner basis $\mathcal{G}$ for the ideal $I$ is reduced if the following three conditions are satisfied: \begin{itemize}
				\item $\mathcal{G}$ is tip-reduced: for $g,h\in \mathcal{G}$ with $g\neq h$, $\mathrm{Tip}(g)\nmid \mathrm{Tip}(h)$;
				\item $\mathcal{G}$ is monic: for every element $g\in \mathcal{G}$, $\mathrm{CTip}(g)=1$;
				\item For any $g\in \mathcal{G}$, $g-\mathrm{Tip}(g)\in \mathrm{Span}_{k}(\mathrm{NonTip}(I))$.
			\end{itemize}
		\end{definition}
		
		It is obvious that under a given admissible order, $I$ has a unique reduced Gr\"{o}bner basis $\mathcal{G}$, and in this case, $\mathrm{Tip}(\mathcal{G})$ is a minimal generating set of $\langle\mathrm{Tip}(I)\rangle$. Moreover, $b\in Q_{\geq 0}$ lies in $\mathrm{NonTip}(I)$ if and only if $b$ cannot be divided by any element of $\mathrm{Tip}(\mathcal{G})$. In the following, we always assume that $\mathcal{G}$ is a reduced Gr\"{o}bner basis of $I$. 
		
		Note that when $\mathcal{G}$ is reduced, there is a one-to-one correspondence between $\mathcal{G}$ and $\mathrm{Tip}(\mathcal{G})$:  for $ g\in\mathcal{G}$, $\mathrm{Tip}(g)\in \mathrm{Tip}(\mathcal{G})$; conversely, for $w\in\mathrm{Tip}(\mathcal{G})$, there is a unique $g\in\mathcal{G}$ such that $w=\mathrm{Tip}(g)$. We shall denote the correspondence from $\mathcal{G}$ to $\mathrm{Tip}(\mathcal{G})$ by $\mathrm{Tip}$ and its inverse by $\mathrm{Tip^{-1}}$.

		\subsection{The reduced bar resolution of quiver algebras}
		
		Let $E:\simeq\oplus_{e\in Q_0}ke$ be the separable subalgebra of $A$ generated by the classes modulo $I$ of the vertices of $Q$, such that $A=E\oplus A_+$ as $E^e$-modules, where $E^e=E\otimes_k E$ and $A_+:=\mathrm{Span}_k\{\mathrm{NonTip}(I)\backslash Q_0\}$. There is an $E^e$-projection from $A$ to $A_+$, denoted by $p_A$. The reduced bar resolution of the quiver algebra $A$ can be written by the following theorem in the sense of Cibils \cite{C}.
		
		\begin{theorem}
			For the algebra $A=kQ/I$, the reduced bar resolution $B(A)$ is a two-sided projective resolution of $A$ with $B_0(A)=A\otimes_EA$, $B_n(A)=A\otimes_E(A_+)^{\otimes_En}\otimes_EA$ and the differential $d=(d_n)$ is
			
			$$d_n([a_1|\cdots|a_n])=a_1[a_2|\cdots|a_n]+\sum_{i=1}^{n-1}(-1)^i[a_1|\cdots|a_ia_{i+1}|\cdots|a_n]+(-1)^n[a_1|\cdots|a_{n-1}]a_n$$	
			with $[a_1|\cdots|a_n]:=1\otimes a_1\otimes\cdots\otimes a_n\otimes 1$. By convention $B_{-1}(A)=A$, and $d_0:A\otimes_EA\longrightarrow A$ is given by the multiplication $\mu_A$ in $A$.
		\end{theorem}
		
		\begin{remark}
			By the definition of $A_+$, $B_n(A)$ can be decomposed as
			$$B_n(A)=\bigoplus A[w_1|\cdots|w_n]A=\bigoplus A^e[w_1|\cdots|w_n],$$
			where $A^e=A\otimes_k A^{op}$ is the enveloping algebra of $A$ and the direct sum is taken over all $[w_1|\cdots|w_n]$ such that all $w_i\in \mathrm{NonTip}(I)\backslash Q_0$  and $w_1\cdots w_n$ is a path in $Q$.
		\end{remark}
		
		Since the reduced bar resolution is a two-sided projective resolution of $A$, we can use it to compute the Hochschild cohomology groups $\mathrm{HH}^n(A)$ of $A$. More concretely, applying the functor $\mathrm{Hom}_{{A}^e}(-,{A})$ to $B(A)$ we get a cochain complex $(C^*(A),\delta^*)$, where $$C^0(A)\cong\mathrm{Hom}_{E^e}(E,A)\cong A^E=\{a\in A\mid sa=as\text{ for all }s\in E\},$$ $$C^n(A)=\mathrm{Hom}_{A^e}(B_n(A),A)\cong \mathrm{Hom}_{E^e}((A_+)^{\otimes_En},A)$$ for $n\geq 1$ (c.f. Lemma \ref{identification-1}), $(\delta^0a)(x)=ax-xa$ for $a\in A^E$ and $x\in A_+$, and $$(\delta^nf)(x_1\otimes\cdots\otimes x_{n+1})=x_1f(x_2\otimes\cdots\otimes x_{n+1})$$ $$+\sum_{i=1}^{n}(-1)^if(x_1\otimes\cdots\otimes x_ix_{i+1}\otimes\cdots\otimes x_{n+1})+(-1)^{n+1}f(x_1\otimes\cdots\otimes x_n)x_{n+1}.$$ 
		Then we have $$\mathrm{HH}^n(A)=\mathrm{Ker}\delta^n/\mathrm{Im}\delta^{n-1}.$$
		In particular, we have $\mathrm{HH}^{1}(A)= \mathrm{Ker}\delta^{1}/\mathrm{Im}\delta^{0}$ as $k$-spaces, where $\mathrm{Ker}\delta^{1}$ is the set of $E^e$-derivations of $A_+$ into $A$ and the elements in $\mathrm{Im}\delta^{0}$ are inner $E^e$-derivations of $A_+$ into $A$. Note that we can identify $\mathrm{Der}_{E^e}(A_+,A)$ with $\mathrm{Der}_{E^e}(A,A)$. Moreover, $\mathrm{Ker}\delta^{1}=\mathrm{Der}_{E^e}(A_+,A)$ has a Lie algebra structure under the Lie bracket
		$$[f,g]_{HH}:=f\circ p_A\circ g-g\circ p_A\circ f$$
		for $f,g\in \mathrm{Der}_{E^e}(A_+,A)$, where $p_A$ denotes the $E^e$-projection from $A$ to $A_+$. Moreover, $\mathrm{Im}\delta^{0}$ is a Lie ideal of $\mathrm{Ker}\delta^{1}$, so that $\mathrm{HH}^{1}(A)$ is a Lie algebra. This structure was first defined by Gerstenhaber \cite{Ger} using the standard bar resolution of $A$.
		
In next two subsections, we will explain how to use the algebraic Morse theory to shrink the above reduced bar resolution of $A$ to a ``smaller'' one, such that the homology of the two complexes coincides.

		\subsection{Algebraic Morse theory} 
		
		The most general version of algebraic Morse theory was presented in Chen, Liu and Zhou \cite{CLZ}. For our purpose, we will adopt a Morse matching condition defined in \cite[Proposition 3.2]{CLZ}.
		
		Let $R$ be an associative ring and $C_{*}=(C_n, \partial_n)_{n\in\mathbb{Z}}$ be a chain complex of left $R$-modules. We assume that each $R$-module $C_n$ has a decomposition $C_n\simeq\oplus_{i\in I_n}C_{n,i}$ of $R$-modules, so we can regard the differentials $\partial_n$ as a matrix $\partial_n=(\partial_{n,ji})$ with $i\in I_n$ and $j\in I_{n-1}$ where $\partial_{n,ji}:C_{n,i}\rightarrow C_{n-1,j}$ is a homomorphism of $R$-modules.
		
		Given the complex $C_*$ as above, we construct a weighted quiver $G(C_*):=(V,E)$. The set $V$ of vertices of $G(C_*)$ consists of the pairs $(n,i)$ with $n\in\mathbb{Z},i\in I_n$ and the set $E$ of weighted arrows is given by the following rule:
		if the map $\partial_{n,ji}$ does not vanish, draw an arrow in E from $(n,i)$ to $(n-1,j)$ 
		and denote the weight of this arrow by the map $\partial_{n,ji}$.
		
		A full subquiver $\mathcal{M}$ of the weighted quiver $G(C_*)$ is called a partial matching if it satisfies the following two conditions:
		\begin{itemize}
			\item $(Matching)$ Each vertex in $V$ belongs to at most one arrow of $\mathcal{M}$.
			
			\item $(Invertibility)$ Each arrow in $\mathcal{M}$ has its weight invertible as a $R$-homomorphism.
			
		\end{itemize}
		With respect to a partial matching $\mathcal{M}$, we can define a new weighted quiver $G_{\mathcal{M}}(C_*)=(V,E_{\mathcal{M}})$, where $E_{\mathcal{M}}$ is given by 
		\begin{itemize}
			\item Keep everything for all arrows which are not in $\mathcal{M}$ and call them thick arrows.
			
			\item For an arrow in $\mathcal{M}$, replace it by a new dotted arrow in the reverse direction and the weight of this new arrow is the negative inverse of the weight of original arrow.
		\end{itemize}
		A path in  $G_{\mathcal{M}}(C_*)$ is called a zigzag path if dotted arrows and thick arrows appear alternately.
		
		Next, for convenience, we will introduce from J\"{o}llenbeck and Welker \cite{JW} the notations related to the weighted quiver $G(C_*)=(V,E)$ with a partial matching $\mathcal{M}$ on it.
		
		\begin{definition} 
			\begin{enumerate}[(1)]
				\item A vertex $(n,i)\in V$ is critical with respect to $\mathcal{M}$ if $(n,i)$ does not lie in any arrow in $\mathcal{M}$. Let $V_n$ denote all the vertices with the first number equal to $n$, we write
				$$V_n^{\mathcal{M}}:=\{(n,i)\in V_n\;|\;(n,i)\;is \; critical\}$$
				for the set of all critical vertices of homological degree $n$.
				
				\item Write $(m,j)\leq (n,i)$ if there exists an arrow from $(n,i)$ to $(m,j)$ in $G(C_*)$.
				
				\item Denote by $P((n,i),(m,j))$ the set of all zigzag paths from $(n,i)$ to $(m,j)$ in  $G_{\mathcal{M}}(C_*)$.
				
				\item The weight $w(p)$ of a path $$p=((n_1,i_1)\rightarrow (n_2,i_2)\rightarrow \cdots\rightarrow (n_r,i_r))\in P((n_1,i_1),(n_r,i_r))$$ in  $G_{\mathcal{M}}(C_*)$ is given by
				
				$$w(p):=w((n_{r-1},i_{r-1})\rightarrow (n_{r},i_{r}))\circ\cdots\circ w((n_1,i_1)\rightarrow (n_2,i_2)),$$
				
				$$
				w((n,i)\rightarrow (m,j)):=\left\{
				\begin{array}{*{3}{lll}}
					-\partial_{m,ij}^{-1}&,& (n,i)\leq (m,j),\\
					\partial_{n,ji}&,& (m,j)\leq (n,i).
				\end{array}
				\right.
				$$
				Then we write $\Gamma((n,i), (m,j))=\sum_{p\in P((n,i),(m,j))}w(p)$ for the sum of weights of all zigzag paths from $(n,i)$ to $(m,j)$.
				
			\end{enumerate}
		\end{definition}
		
		Following \cite[Proposition 3.2]{CLZ}, we call a partial matching $\mathcal{M}$ as above a Morse matching if any zigzag path starting from $(n,i)$ is of finite length for each vertex $(n,i)$ in $G_{\mathcal{M}}(C_*)$. 
		
		Now we can define a new complex $C_*^{\mathcal{M}}$, which we call the Morse complex of $C_*$ with respect to $\mathcal{M}$. The complex $C_*^{\mathcal{M}}=(C_n^{\mathcal{M}},\partial_n^{\mathcal{M}})_{n\in\mathbb{Z}}$ is defined by
		
		$$C_n^{\mathcal{M}}:=\oplus_{(n,i)\in V_n^{\mathcal{M}}}C_{n,i},$$
		
		$$
		\partial_n^{\mathcal{M}}:\left\{
		\begin{array}{*{3}{lll}}
			C_n^{\mathcal{M}}& \rightarrow& C_{n-1}^{\mathcal{M}}\\
			x\in C_{n,i}& \mapsto& \sum_{(n-1,j)\in V_{n-1}^{\mathcal{M}}}\Gamma((n,i),(n-1,j))(x).
		\end{array}
		\right.
		$$
		
		The main theorem of algebraic Morse theory can be stated as follows.
		
		\begin{theorem} \label{morse}
			The complex $C_*^{\mathcal{M}}$ is a complex of left $R$-modules which is homotopy equivalent to the original complex $C_*$. Moreover, the maps defined below are chain homotopies between $C_*$ and $C_*^{\mathcal{M}}$:			
			$$
			f:\left\{
			\begin{array}{*{3}{lll}}
				C_n& \rightarrow& C_n^{\mathcal{M}}\\
				x\in C_{n,i}& \mapsto& \sum_{(n,j)\in V_{n}^{\mathcal{M}}}\Gamma((n,i),(n,j))(x),
			\end{array}
			\right.
			$$			
			$$
			g:\left\{
			\begin{array}{*{3}{lll}}
				C_n^{\mathcal{M}}& \rightarrow& C_n\\
				x\in C_{n,i}& \mapsto& \sum_{(n,j)\in V_{n}}\Gamma((n,i),(n,j))(x).
			\end{array}
			\right.
			$$			
		\end{theorem}

		\subsection{Two-sided Anick resolution}
		
		Starting from the reduced bar resolution of a one-vertex quiver algebra $A$ which is viewed as a chain complex of projective $A^e$-modules, Sk\"{o}ldberg \cite{ES} constructed a ``smaller'' $A^e$-projective resolution of $A$ using algebraic Morse theory, which is called the two-sided Anick resolution of $A$. It was pointed out in \cite{CLZ} that Sk{\"o}ldberg's construction generalizes to general quiver algebras.
		
		Let $A=kQ/I$ be a quiver algebra, let $\mathcal{G}$ be a reduced Gr{\"o}bner basis of the ideal $I$, and denote $W:=\mathrm{Tip}(\mathcal{G})$. Denote by $B(A)=(B_*(A),d_*)$ the reduced bar resolution of $A$ (c.f. Section 2.2). Similar as in \cite{CLZ}, we define a new quiver $Q_W=(V,E)$ with respect to $W$, which is called the Ufnarovski\u\i~graph (or just Uf-graph).
		
		\begin{definition} \label{Uf-graph}
			A Uf-graph $Q_W=(V,E)$ with respect to $W$ of the algebra $A=kQ/I$ is given by
			
			$$V:=Q_0\cup Q_1\cup \{u\in Q_{\geq 0}\ |\ \text{$u$ is a proper right factor of some $v\in W$}\}$$
			
			$$E:=\{e\rightarrow x\ |\ e\in Q_0,\;x=ex\in Q_1\}~~\cup$$$$\{u\rightarrow v\ |\ uv\in\langle \mathrm{Tip}(\mathcal{G})\rangle, \text{ but } w\notin\langle \mathrm{Tip}(\mathcal{G})\rangle\;for\;uv=wp,\;l(p)\geq 1\}$$
		\end{definition}
		
		Using the Uf-graph $Q_W$, one can define (for each $i\geq -1$) the $i$-chains, which form a subset of generators of $B_i(A)=\bigoplus A^e[w_1|\cdots|w_i]$ for $i\geq 0$, with $w_1,\cdots,w_i\in \mathrm{NonTip}(I)\backslash Q_0$, $w_1\cdots w_i\in Q_{\geq 0}$.
		
		\begin{definition} \label{i-chains}
			
			\begin{itemize}
				\item The set $W^{(i)}$ of $i$-chains consists of all sequences $[w_1|\cdots|w_{i+1}]$ with each $w_k\in \mathrm{NonTip}(I)\backslash Q_0$, such that
				$$e\rightarrow w_1\rightarrow w_2\rightarrow \cdots\rightarrow w_{i+1}$$
				is a path in $Q_W$. Define $W^{(-1)}:=Q_0$.
				
				\item For all $ p\in Q_{\geq 0}$, define $$V_{p,i}^{(n)}=\{[w_1|\cdots|w_n]\ |\ p=w_1\cdots w_n,\; [w_1|\cdots|w_{i+1}]\in W^{(i)},\;[w_1|\cdots|w_{i+2}]\notin W^{(i+1)}\}$$
			\end{itemize}
		\end{definition}

		By the definition above, we can define a partial matching $\mathcal{M}$ to be the set of arrows of the following form in the weighted quiver $G(B_*)$, where $B_*=B(A)$ is the reduced bar resolution of the algebra $A=kQ/I$:
		
		$$[w_1|\cdots|w_{i+1}|w_{i+2}'|w_{i+2}''|w_{i+3}|\cdots|w_n]\stackrel{(-1)^{i+2}}{\longrightarrow}[w_1|\cdots|w_{i+2}|\cdots|w_n]$$
		where $$w=w_1\cdots w_n=w_1\cdots w_{i+2}'w_{i+2}''\cdots w_n,~~w_{i+2}=w_{i+2}'w_{i+2}'',$$ 
		$$[w_1|\cdots|w_{i+1}|w_{i+2}'|w_{i+2}''|w_{i+3}|\cdots|w_n]\in V_{w,i+1}^{(n+1)},~~[w_1|\cdots|w_{i+2}|\cdots|w_n]\in V_{w,i}^{(n)}.$$
		
		\begin{theorem}$(\cite[~Theorem~4.3]{CLZ})$\label{morse matching}
			The partial matching $\mathcal{M}$ is a Morse matching of $G(B_*)$. In this case, the set of critical vertices in $n$-th component is given by $W^{(n-1)}$.
		\end{theorem}
		Therefore, by Theorem \ref{morse matching} and Theorem \ref{morse}, $(B^{\mathcal{M}}(A), d^{\mathcal{M}})$ is also a $A^e$-projective resolution of $A$, but the projective $A^e$-modules in $B^{\mathcal{M}}(A)$ are usually much smaller than $B(A)$. The resolution $(B^{\mathcal{M}}(A), d^{\mathcal{M}})$ is called the two-sided Anick resolution of $A$ in \cite{ES} and in \cite{CLZ}.
		
		Note that for $n\geq 0$, the $n$-th component of $B^{\mathcal{M}}(A)$ is $A\otimes_E kW^{(n-1)}\otimes_EA$. In particular, we have the following identifications:
		$$W^{(-1)}=Q_0,~~ W^{(0)}=\{[w_1]\ |\ w_1\in Q_1\}\cong Q_1,$$
		$$W^{(1)}=\{[w_1|w_2]\ |\ \;w_1\in Q_1,\; w_1w_2\in \mathrm{Tip}(\mathcal{G})\}\cong \mathrm{Tip}(\mathcal{G})=W .$$

		\section{Generalized parallel paths method for computing the first Hochschild cohomology group}
		
		\subsection{Parallel paths method of Strametz}
		
		Let ${A}=kQ/I$ be a quiver algebra and $E\simeq kQ_0$ be its separable subalgebra as in Section 2.2. Recall that we assume $I\subseteq kQ_{\geq 2}$. The following notations (most of them are taken from \cite{Strametz}) will be useful.

		\begin{definition}
			\begin{itemize}
				\item Let $\varepsilon$ be a path in $Q$ and $(\alpha,\gamma)\in Q_1//Q_{\geq 0}$.  Denote by $\varepsilon^{(\alpha,\gamma)}$ the sum of all nonzero paths obtained by replacing one appearance of the arrow $\alpha$ in $\varepsilon$ by the path $\gamma$. If the path $\varepsilon$ does not contain the arrow $\alpha$, we set $\varepsilon^{(\alpha,\gamma)}=0$.
				\item If we fix a Gr{\"o}bner basis $\mathcal{G}$ of $I$ in $kQ$, then there is a $k$-linear basis $\mathcal{B}$ of the algebra ${A}$ with respect to $\mathcal{G}$. Indeed $\mathcal{B}$ is given by $\mathrm{NonTip}(I)$ modulo $I$ and there is a bijection between the elements of $\mathcal{B}$  and the elements of $\mathrm{NonTip}(I)$ (c.f. Section 2.1). For this reason, we often identify $\mathcal{B}$ with $\mathrm{NonTip}(I)$.
				
				\item Let $\pi:kQ\rightarrow {A}$ be the canonical projection. For a path $p\in Q$, $\pi(p)$ can be uniquely written as a linear combination of the elements in the basis $\mathcal{B}$.
				
				\item If $X$ is a set of paths of $Q$ and $e$ is a vertex of $Q$, the set $Xe$ is formed by the paths of $X$ with source vertex $e$. In the same way, $eX$ denotes the set of all paths of $X$ with terminus vertex $e$.
			\end{itemize}
		\end{definition}
		
		We now give a brief review about Strametz's method in \cite{Strametz} for computing the first Hochschild cohomology group of monomial algebras. Recall that an algebra ${A}=kQ/I$ is called a monomial algebra, if the ideal $I$ is generated by a set $Z$ of paths in $Q$. We shall assume that $Z$ is minimal, so that $Z$ is a reduced Gr\"{o}bner basis of $I$ with $Z=\mathrm{Tip}(Z)$. Note that $\mathrm{NonTip}(\langle Z\rangle)$ modulo $I$ gives a $k$-basis of $A$, which we denote by $\mathcal{B}$.
		
		\begin{proposition}\label{strametz} $(\cite[~Proposition~2.6]{Strametz})$
			Let $A$ be a finite dimensional monomial algebra. Then the beginning of the cochain complex of the minimal $A^e$-projective resolution of $A$ can be described by:
			
			\[\begin{tikzcd}
				0 & {k(Q_0//\mathcal{B})} & {k(Q_1//\mathcal{B})} & {k(Z//\mathcal{B})} & \cdots
				\arrow[from=1-1, to=1-2]
				\arrow["{\psi_0}", from=1-2, to=1-3]
				\arrow["{\psi_1}", from=1-3, to=1-4]
				\arrow[from=1-4, to=1-5]
			\end{tikzcd}\]
			where the differentials are given by
			$$	\begin{array}{*{5}{lllll}}
				\psi_0& :&k(Q_0//\mathcal{B}) &\rightarrow&k(Q_1//\mathcal{B}),\\
				&& (e,\gamma)&\mapsto&\sum_{\alpha\in Q_1e}(\alpha,\pi(\alpha\gamma))-\sum_{\beta\in eQ_1}(\beta,\pi(\gamma\beta));\\\\
				\psi_1&: &k(Q_1//\mathcal{B})&\rightarrow&k(Z//\mathcal{B}),\\
				&&(\alpha,\gamma)&\mapsto&\sum_{p\in Z}(p,\pi(p^{(\alpha,\gamma)})).
			\end{array}$$
			In particular, we have $\mathrm{HH}^0({A})\cong \mathrm{Ker}\psi_0$, $\mathrm{HH}^1({A})\cong \mathrm{Ker}\psi_1/\mathrm{Im}\psi_0$. 
		\end{proposition}
		
		The proof of Proposition \ref{strametz} uses the minimal $A^e$-projective resolution of the monomial algebra $A$ given by Bardzell \cite{Bardzell} and the following two general lemmas.
		
		\begin{lemma} \label{identification-1} $(\cite[~Lemma~2.2]{Strametz})$
			Let ${A}=kQ/I$ be a quiver algebra, $E\simeq kQ_0$ be the separable subalgebra of $A$, $M$ be an $E$-bimodule and $T$ be a ${A}$-bimodule. Then $\mathrm{Hom}_{{A}^e}({A}\otimes_E M\otimes_E{A}, T)$ and $\mathrm{Hom}_{E^e}(M,T)$ are isomorphic as vector spaces.
		\end{lemma}
		
		\begin{proof}
			It is easy to check that the $k$-linear maps given by 	
			$$f\mapsto{(m\mapsto f(1\otimes m\otimes1))}$$ and
			$$g\mapsto{(1\otimes m\otimes1\mapsto g(m))}$$ are well-defined and inverse to each other.		
		\end{proof}

		\begin{lemma} \label{identification-2} $(\cite[~Lemma~2.3]{Strametz})$\label{Lem2}
			Let ${A}=kQ/I$ be a quiver algebra and $E\simeq kQ_0$ be its separable subalgebra. Let $X$ and $Y$ be the  sets of paths of $Q$ and let $kX$ and $kY$ be the corresponding $E$-bimodules. If $X$ is a finite set, then the vector spaces $k(X//Y)$ and $\mathrm{Hom}_{E^e}(kX,kY)$ are isomorphic.
		\end{lemma}
		
		\begin{proof}
			It is easy to check that the $k$-linear maps given by 	
			$$(x,y)\mapsto{(x\mapsto y,\; x'\mapsto 0,\text{ for }x'\neq x)}$$ and
			$$f\mapsto\sum_{x\in X}\sum_{i}\lambda_i(x,p_i)$$ with $f(x)=\sum_{i}\lambda_ip_i$, $\lambda_i\in k$, $p_i\in Y$, are well-defined and inverse to each other.	
		\end{proof}
		
		\begin{remark}
			Lemma \ref{Lem2} is the intrinsic reason why the parallel paths method can not always be used for the quiver algebras with infinite dimension. Indeed, we need to restrict $X$ to be finite to make the above maps well-defined.
			In particular, for $X=\mathrm{Tip}\mathcal{G}$, in order to use above lemma, we need $X$ to be a finite set. When $kQ/I$ is finite dimensional, $\mathrm{Tip}\mathcal{G}$ is always finite by \cite[~Proposition~2.10]{Green}. When $kQ/I$ is infinite dimensional but $I$ has a finite Gr\"{o}bner basis in $kQ$, we can still use Lemma \ref{Lem2} for $X=\mathrm{Tip}\mathcal{G}$. For example, we can use Lemma \ref{Lem2} to $X=\mathrm{Tip}\mathcal{G}$ for the algebra $k\langle x,y\rangle/\langle x^2\rangle$, but not for the algebra $k\langle x,y\rangle/\langle xyx,xyyx,xyyyx,\cdots\rangle$. 		
		\end{remark}

		\subsection{Generalized parallel paths method} 
		
		In this subsection, we will extend the parallel paths method for computing the first Hochschild cohomology group from monomial algebras to general quiver algebras. We will refer to this method as the generalized parallel paths method.
		
		Let ${A}=kQ/I$ be a quiver algebra such that $I$ has a finite reduced Gr{\"o}bner basis $\mathcal{G}$. 	
		The following lemma can be seen as a generalization of the beginning of the two-sided minimal projective resolution of a monomial algebra given by Bardzell \cite{Bardzell}. Our proof is a careful analysis of the beginning of the two-sided Anick resolution (c.f. Section 2.4) of $A$. It is worth noting that Bardzell has proved the same result in the case of finite dimensional algebras in \cite{Bardzell}. Our proof uses algebraic Morse theory without the restriction that $A$ is finite dimensional.
		
		\begin{lemma} $(\cite[~Proposition~2.1]{Bardzell})$
			The beginning of the two-sided Anick resolution $(B^{\mathcal{M}}({A}),d^{\mathcal{M}})$ of ${A}$ can be described by (for simplicity we just denote $d^{\mathcal{M}}$ by $d$):
			\[\begin{tikzcd}
				\cdots & {{A}\otimes_E k\mathrm{Tip}(\mathcal{G})\otimes_E {A}} & {{A}\otimes_E kQ_1\otimes_E {A}} & {{A}\otimes_E kQ_0\otimes_E {A}} & {A} & 0,
				\arrow["{d_1}", from=1-3, to=1-4]
				\arrow["{d_0}", from=1-4, to=1-5]
				\arrow["{d_2}", from=1-2, to=1-3]
				\arrow[from=1-5, to=1-6]
				\arrow[from=1-1, to=1-2]
			\end{tikzcd}\]
			where the differentials can be described as follows: 
			\begin{itemize}
				\item for all $ 1\otimes e_i\otimes 1\in{A}\otimes_E kQ_0\otimes_E{A}$, $d_0(1\otimes e_i\otimes 1)=e_i$;
				\item for all $ 1\otimes\alpha\otimes 1\in{A}\otimes_E kQ_1\otimes_E{A}$, $d_1(1\otimes\alpha\otimes 1)=\alpha\otimes s(\alpha)\otimes1-1\otimes t(\alpha)\otimes\alpha$;
				\item  for all $ 1\otimes w\otimes 1\in{A}\otimes_E k\mathrm{Tip}(\mathcal{G})\otimes_E{A}$, $$d_2(1\otimes w\otimes1)=\sum_{\alpha_{m}\cdots\alpha_{1}\in \mathrm{Supp}(g)}\sum_{i=1}^{m}c(\alpha_{m}\cdots\alpha_{1})\alpha_{m}\cdots\alpha_{i+1}\otimes\alpha_{i}\otimes \alpha_{i-1}\cdots\alpha_{1},$$
				where $ g\in\mathcal{G}$ such that $g=w+\sum_{p\in Q_{\geq 0};\ p\neq w}c(p)p$ with $\mathrm{Tip}(g)=w$ and $0\neq c(p)\in k$. (Note that $c(w)=1$ and $l(w)\geq 2$.)
			\end{itemize}
		\end{lemma}
		
		\begin{proof}
			By the identifications at the end of Section 2.4, we just need to check the differentials in this lemma. Obviously, $d_0=\mu_A$ which comes from the reduced bar resolution of $A$, is given by the multiplication of $A$. By the definition of the Morse matching $\mathcal{M}$ in the weighted quiver $G(B_*)$ (c.f. Section 2.4), there are no dotted arrows starting from $W^{(-1)}=Q_0$ in the new weighted quiver $G_{\mathcal{M}}(B_*)$, where $B_*:=B(A)$ is the reduced bar resolution of $A$. Thus for $[\alpha]\in W^{(0)}$ with $\alpha\in Q_1$, the zigzag paths from $W^{(0)}=Q_1$ to $W^{(-1)}=Q_0$ can be given by
			$$
			\begin{tikzcd}
				&& \lbrack\alpha\rbrack \\
				\lbrack s(\alpha)\rbrack &&&& \lbrack t(\alpha)\rbrack
				\arrow["(\alpha\otimes 1)"', from=1-3, to=2-1]
				\arrow["(-1)\cdot(1\otimes\alpha^{op})", from=1-3, to=2-5]
			\end{tikzcd}$$
			
			For $ 1\otimes w\otimes 1\in{A}\otimes_{E} k\mathrm{Tip}(\mathcal{G})\otimes_{E}{A}$, let $w=w_1w_2$ with $w_1\in Q_1$. We are going to find all zigzag paths from $[w_1|w_2]$ to some $[\alpha]$ in $G_{\mathcal{M}}(B_*)$ with $\alpha\in Q_1$. First of all, in the original reduced bar resolution $B_*$, the differential of $[w_1|w_2]$ is $$w_1\lbrack w_2\rbrack+\sum_{p\in \mathrm{Supp}(g);\ p\neq w}c(p)\lbrack p\rbrack+\lbrack w_1\rbrack w_2,$$
			where $w=w_1w_2=-\sum_{p\in \mathrm{Supp}(g);\ p\neq w}c(p)p$ in $A$ (modulo $I$). For some $\alpha_{m}\cdots\alpha_{1}\in \mathrm{Supp}(g)$, there are two cases to be considered. 
			
			\textit{Case 1}. If $\alpha_{m}\cdots\alpha_{1}=w$, the only zigzag path from $[w_1|w_2]$ to $[\alpha_k]$ ($1\leq k\leq m$) is given by
			$$\begin{tikzcd}
				\lbrack\alpha_{m}|\alpha_{m-1}\cdots\alpha_{1}\rbrack & \lbrack\alpha_{m-1}|\alpha_{m-2}\cdots\alpha_{1}\rbrack & \lbrack\alpha_{k}|\alpha_{k-1}\cdots\alpha_{1}\rbrack \\
				\lbrack\alpha_{m-1}\cdots\alpha_{1}\rbrack & \cdots & \lbrack\alpha_k\rbrack
				\arrow["(\alpha_{m}\otimes 1)"', from=1-1, to=2-1]
				\arrow["1", dashed, from=2-1, to=1-2]
				\arrow[from=1-2, to=2-2]
				\arrow[dashed, from=2-2, to=1-3]
				\arrow["(-1)^2\cdot(1\otimes(\alpha_{k-1}\cdots\alpha_{1})^{op})", from=1-3, to=2-3]
			\end{tikzcd}$$
			Thus the coefficient of $[\alpha_{k}]$ is $(\alpha_{m}\cdots\alpha_{k+1})\otimes(\alpha_{k-1}\cdots\alpha_{1})^{op}$.
			
			\textit{Case 2}. If $\alpha_{m}\cdots\alpha_{1}=p$ with $p\in \mathrm{Supp}(g)$ and $p\neq w$, the only zigzag path from $[w_1|w_2]$ to $[\alpha_k]$ ($1\leq k\leq m$) is given by
			$${\footnotesize\begin{tikzcd}
					\lbrack w_1|w_2\rbrack & \lbrack\alpha_{m}|\alpha_{m-1}\cdots\alpha_{1}\rbrack & \lbrack\alpha_{m-1}|\alpha_{m-2}\cdots\alpha_{1}\rbrack & \lbrack\alpha_{k}|\alpha_{k-1}\cdots\alpha_{1}\rbrack \\
					\lbrack\alpha_{m}\cdots\alpha_{1}\rbrack & \lbrack\alpha_{m-1}\cdots\alpha_{1}\rbrack & \cdots & \lbrack\alpha_k\rbrack
					\arrow["(-1)\cdot(-c(\alpha_{m}\cdots\alpha_{1}))"', from=1-1, to=2-1]
					\arrow["1", dashed, from=2-1, to=1-2]
					\arrow["(\alpha_{m}\otimes 1)"', from=1-2, to=2-2]
					\arrow["1", dashed, from=2-2, to=1-3]
					\arrow[from=1-3, to=2-3]
					\arrow[dashed, from=2-3, to=1-4]
					\arrow["(1\otimes(\alpha_{k-1}\cdots\alpha_{1})^{op})", from=1-4, to=2-4]
			\end{tikzcd}}$$
			Thus the coefficient of $[\alpha_{k}]$ is $c(\alpha_{m}\cdots\alpha_{1})(\alpha_{m}\cdots\alpha_{k+1})\otimes(\alpha_{k-1}\cdots\alpha_{1})^{op}$.
			\end{proof}
		
		Applying the functor $\mathrm{Hom}_{{A}^e}(-,{A})$ to $B^{\mathcal{M}}({A})$ and using the identification in Lemma \ref{identification-1} yields the following cochain complex which we denote by $C_\mathcal{M}({A})=(C_\mathcal{M}^*,\delta^*)$:
		\[\begin{tikzcd}
			0 & {\mathrm{Hom}_{E^e}(kQ_0,{A})} & {\mathrm{Hom}_{E^e}(kQ_1,{A})} & {\mathrm{Hom}_{E^e}(k\mathrm{Tip}(\mathcal{G}),{A})} & \cdots
			\arrow[from=1-1, to=1-2]
			\arrow["{\delta^0}", from=1-2, to=1-3]
			\arrow["{\delta^1}", from=1-3, to=1-4]
			\arrow[from=1-4, to=1-5]
		\end{tikzcd}\]
		where the coboundaries $\delta^0$ and $\delta^1$ are given by
		$$\delta^0(f)(\alpha)=\alpha\cdot f(s(\alpha))-f(t(\alpha))\cdot\alpha$$
		$$\delta^1(g)(w)=\sum_{\alpha_{m}\cdots\alpha_{1}\in \mathrm{Supp}(\mathrm{Tip^{-1}}(w))}\sum_{i=1}^{m}c(\alpha_{m}\cdots\alpha_{1})\cdot\alpha_{m}\cdots\alpha_{i+1}g(\alpha_{i})\alpha_{i-1}\cdots\alpha_{1}$$
		where $f\in \mathrm{Hom}_{E^e}(kQ_0,{A})$, $\alpha\in Q_1$, $g\in \mathrm{Hom}_{E^e}(kQ_1,{A})$, $w\in \mathrm{Tip}(\mathcal{G})$.
		
		If we carry out the identification suggested in Lemma \ref{identification-2}, we can rewrite the coboundaries and obtain:
		
		\begin{proposition}\label{gen-parallel paths}
			The beginning of the cochain complex $C_{\mathcal{M}}({A})$ can be characterized in the following way
			\[\begin{tikzcd}
				0 & {k(Q_0//\mathcal{B})} & {k(Q_1//\mathcal{B})} & {k(\mathrm{Tip}(\mathcal{G})//\mathcal{B})} & \cdots
				\arrow[from=1-1, to=1-2]
				\arrow["{\psi_0}", from=1-2, to=1-3]
				\arrow["{\psi_1}", from=1-3, to=1-4]
				\arrow[from=1-4, to=1-5]
			\end{tikzcd}\]
			where the maps are given by
			$$	\begin{array}{*{5}{lllll}}
				\psi_0& :&k(Q_0//\mathcal{B}) &\rightarrow&k(Q_1//\mathcal{B}),\\
				&& (e,\gamma)&\mapsto&\sum_{\alpha\in Q_1e}(\alpha,\pi(\alpha\gamma))-\sum_{\beta\in eQ_1}(\beta,\pi(\gamma\beta));\\\\
				\psi_1&: &k(Q_1//\mathcal{B})&\rightarrow&k(\mathrm{Tip}(\mathcal{G})//\mathcal{B}),\\
				&&(\alpha,\gamma)&\mapsto&\sum_{g\in\mathcal{G}}\sum_{p\in \mathrm{Supp}(g)}c_g(p)\cdot(\mathrm{Tip}(g),\pi(p^{(\alpha,\gamma)})).
			\end{array}$$
			with $g=\sum_{p\in \mathrm{Supp}(g)}c_g(p)p$, $c_g(p)\in k$.
			
			In particular, we have $\mathrm{HH}^0({A})\cong \mathrm{Ker}\psi_0$, $\mathrm{HH}^1({A})\cong \mathrm{Ker}\psi_1/\mathrm{Im}\psi_0$. 
		\end{proposition}
		
		\begin{proof} The verifications are straightforward.
		\end{proof}	
		
		We should mention that there is an analogous result in \cite[~Section~2.2]{RSS} given by Rubio y Degrassi, Schroll and Solotar, whose method is based on the Chouhy-Solotar projective  resolution.
		
		\begin{theorem}\label{gen-lie bracket} $(\mathrm{c.f.}~\cite[~Theorem~2.7]{Strametz}$$)$
			The bracket $$[(\alpha,\gamma),(\beta,\varepsilon)]=(\beta,\pi(\varepsilon^{(\alpha,\gamma)}))-(\alpha,\pi(\gamma^{(\beta,\varepsilon)}))$$
			for all $(\alpha,\gamma),(\beta,\varepsilon)\in Q_1//\mathcal{B}$ induces a Lie algebra structure on $\mathrm{Ker}\psi_1/\mathrm{Im}\psi_0$, such that $\mathrm{HH}^1({A})$ and $\mathrm{Ker}\psi_1/\mathrm{Im}\psi_0$ are isomorphic as Lie algebras.
		\end{theorem}
		
		\begin{proof} As $B({A})$ and $B^{\mathcal{M}}({A})$ are projective resolutions of ${A}^e$-modules ${A}$, there exist, thanks to the Comparison Theorem, chain maps $w:B({A})\rightarrow B^{\mathcal{M}}({A})$ and $\xi:B^{\mathcal{M}}({A})\rightarrow  B({A})$ such that they give inverse homotopy equivalences.
			\[\begin{tikzcd}
				{B({A}):} & \cdots & {{A}\otimes_E {A}^+\otimes_E{A}} & {A}\otimes_E{A} & {A} & 0 \\
				{B^{\mathcal{M}}({A}):} & \cdots & {{A}\otimes_E kQ_1\otimes_E{A}} & {{A}\otimes_E kQ_0\otimes_E{A}} & {A} & 0
				\arrow[from=1-2, to=1-3]
				\arrow[from=2-2, to=2-3]
				\arrow[from=1-3, to=1-4]
				\arrow[from=2-3, to=2-4]
				\arrow[from=1-4, to=1-5]
				\arrow[from=2-4, to=2-5]
				\arrow[shift right=1, no head, from=1-5, to=2-5]
				\arrow[shift left=1, from=1-5, to=1-6]
				\arrow[from=2-5, to=2-6]
				\arrow[shift left=1, no head, from=1-5, to=2-5]
				\arrow["{w_0}"', shift right=1, harpoon', from=1-4, to=2-4]
				\arrow["{\xi_0}"', shift right=1, harpoon', from=2-4, to=1-4]
				\arrow["{w_1}"', shift right=1, harpoon', from=1-3, to=2-3]
				\arrow["{\xi_1}"', shift right=1, harpoon', from=2-3, to=1-3]
			\end{tikzcd}\]
			Indeed, according to the algebraic Morse theory (Theorem \ref{morse}), we can make the above chain maps explicitly as follows:
			\begin{itemize}
				\item $w_0:{A}\otimes_E{A}\rightarrow{A}\otimes_E kQ_0\otimes_E{A}$,\quad $\lambda\otimes\mu\mapsto\lambda\otimes e\otimes\mu$;
				
				\item $\xi_0:{A}\otimes_E kQ_0\otimes_E{A}\rightarrow{A}\otimes_E{A}$, \quad $\lambda\otimes e\otimes\mu\mapsto\lambda\otimes\mu$;
				
				\item $w_1:{A}\otimes_E{A}^+\otimes_E{A}\rightarrow{A}\otimes_E kQ_1\otimes_E{A}$,\quad
				$$\lambda\otimes\alpha_{n}\cdots\alpha_{1}\otimes\mu\mapsto\sum_{i=1}^{n}\lambda\alpha_{n}\cdots\alpha_{i+1}\otimes\alpha_{i}\otimes\alpha_{i-1}\cdots\alpha_{1}\mu;$$
				
				\item $\xi_1:{A}\otimes_E kQ_1\otimes_E{A}\rightarrow{A}\otimes_E{A}^+\otimes_E{A}$,\quad $\lambda\otimes\alpha\otimes\mu\mapsto\lambda\otimes\alpha\otimes\mu$,
			\end{itemize}
			where $\lambda,\mu\in{A}$, $\alpha,\alpha_{1},\cdots,\alpha_n\in Q_1$ and $\alpha_{n}\cdots\alpha_1\in \mathrm{NonTip}(I)\backslash Q_0$. 
			
			By the identifications in Lemma \ref{identification-1} and Lemma \ref{identification-2}, there is an isomorphism of vector spaces from $k(Q_1//\mathcal{B})$ to $\mathrm{Hom}_{{A}^e}(A\otimes_E kQ_1\otimes_EA,{A})$ which we denote by $i$.
			
			Using the above comparison morphisms, if we have defined some bilinear operation on $\mathrm{HH}^1({A})$, then we can define an induced operation on $\mathrm{Ker}\psi_1/\mathrm{Im}\psi_0$. This method is widely used, see for example an introduction in Volkov \cite[Section 3]{Volkov}. For simplicity, we often omit the sign $\circ$ and write the composition of morphisms directly. In this way, the Lie bracket on $\mathrm{Ker}\psi_1/\mathrm{Im}\psi_0$ can be defined by
			$$[(\alpha,\gamma),(\beta,\varepsilon)]:=i^{-1}([i((\alpha,\gamma))w_1,i((\beta,\varepsilon))w_1]_{\mathrm{HH}}\cdot\xi_1),$$
			for all $(\alpha,\gamma),(\beta,\varepsilon)\in Q_1//\mathcal{B}$.
			$$
			\begin{array}{*{3}{lll}}
				[(\alpha,\gamma),(\beta,\varepsilon)] & =&i^{-1}(i((\alpha,\gamma))w_1p_Ai((\beta,\varepsilon))w_1\xi_1-i((\beta,\varepsilon))w_1p_Ai((\alpha,\gamma))w_1\xi_1)\\
				& =& i^{-1}(i((\alpha,\gamma))w_1p_Ai((\beta,\varepsilon))-i((\beta,\varepsilon))w_1p_Ai((\alpha,\gamma)))\\
				&=& (\beta,\pi(\varepsilon^{(\alpha,\gamma)}))-(\alpha,\pi(\gamma^{(\beta,\varepsilon)})).
			\end{array}
			$$
			The Lie algebra isomorphism from $\mathrm{Ker}\psi_1/\mathrm{Im}\psi_0$ to $\mathrm{HH}^1({A})$ is given by $i(-)\circ w_1$, since the equation  
			$$i([(\alpha,\gamma),(\beta,\varepsilon)])w_1=[i((\alpha,\gamma))w_1,i((\beta,\varepsilon))w_1]_{\mathrm{HH}}$$
			naturally holds at the cohomology level.		
		\end{proof}

		\subsection{Graded structure of $\mathrm{HH}^1({A})$ when $I$ is a homogeneous ideal}
		
		In this subsection, we discuss the graded structure of the Lie algebra $\mathrm{HH}^1({A})$ when $I$ is a homogeneous ideal. Note that there is a similar discussion for this graded structure in \cite[~Section~2.3]{RSS}. However, our Lemma \ref{loop-power} and Proposition \ref{l_{-1}} below are new.
		
		By Theorem \ref{gen-lie bracket}, we can regard the Lie algebra structures on $\mathrm{HH}^1({A})$ and on $\mathrm{Ker}\psi_1/\mathrm{Im}\psi_0$ as the same one.
		Let ${A}=kQ/I$ be a quiver algebra. Let $\mathcal{G}$ be a reduced Gr{\"o}bner basis of $I$, $\mathcal{B}\subseteq Q_{\geq 0}$ be the $k$-linear basis of ${A}$ with respect to $\mathcal{G}$, and denote $\mathcal{B}_n:=\mathcal{B}\cap Q_{n}$. Assume that the ideal $I$ is homogeneous, that is, under the length-lexicographic order, for all $ g\in\mathcal{G}$, the summands of $g$ are in same length. In this case, the canonical projection $\pi:kQ\rightarrow {A}$ does not change the length of any path. Therefore, if $(\alpha,\gamma)\in Q_1//\mathcal{B}_n$, $(\beta,\varepsilon)\in Q_1//\mathcal{B}_m$, then
		$$[(\alpha,\gamma),(\beta,\varepsilon)]\in k(Q_1//\mathcal{B}_{n+m-1}).$$
		We have $k(Q_1//\mathcal{B})=\bigoplus_{i\in\mathbb{N}}k(Q_1//\mathcal{B}_i)$, $\mathrm{HH}^1({A})=\mathrm{Ker}\psi_1/\mathrm{Im}\psi_0$. Then, if we set 
		\begin{itemize}
			\item $L_{-1}:=k(Q_1//Q_0)\cap \mathrm{Ker}\psi_1$
			
			\item $L_{0}:=k(Q_1//Q_1)\cap \mathrm{Ker}\psi_1\bigg/ \langle \psi_0(e,e)\;|\;e\in Q_0\rangle$
			
			\item $L_{i}:=k(Q_1//\mathcal{B}_{i+1})\cap \mathrm{Ker}\psi_1\bigg/ \langle \psi_0(e,\gamma)\;|\;(e,\gamma)\in Q_0//\mathcal{B}_i\rangle$
		\end{itemize}
		for all $i\in\mathbb{N},i\geq 1$. Note that $$\psi_0(e,e)=\sum_{a\in Q_1e}(a,a)-\sum_{b\in eQ_1}(b,b),\;\;\psi_0(e,\gamma)=\sum_{a\in Q_1e}(a,\pi(a\gamma))-\sum_{b\in eQ_1}(b,\pi(\gamma b)).$$ Then we obtain $\mathrm{HH}^1({A})=\bigoplus_{i\geq -1}L_i$, and $[L_i,L_j]\subseteq L_{i+j}$ for all $i,j\geq -1$, where $L_{-2}=0$.
		The above discussion (following the idea from \cite[Section 4]{Strametz}) shows that we have a gradation on the Lie algebra $\mathrm{HH}^1({A})$ if $I$ is a homogeneous ideal of $kQ$. 
		
		\begin{remark} Note that the above notions of $L_{-1}$ and of $L_0$ and the following lemma make sense even if the ideal $I$ is not homogeneous and they will be used in later sections.
		\end{remark}
		
		If we restrict our discussion to the case where $A=kQ/I$ is finite dimensional, we can give some more precise description about the Gr\"{o}bner basis of $I$ and the $L_{-1}$.
		
		\begin{lemma} \label{loop-power}
			Let ${A}=kQ/I$ be a finite dimensional quiver algebra (with $I\subseteq kQ_{\geq 2}$). Let $a$ be a loop of the quiver $Q$. Then there exists an integer $m\geq 2$ such that $a^m\in \mathrm{Tip}(\mathcal{G})$ and $a^{m-1}\in\mathcal{B}$.
		\end{lemma}
		
		\begin{proof} Clearly there exists an integer $m\geq 2$ such that $a^m\in \mathrm{Tip}(\mathcal{G})$, since otherwise, the elements $a^2,a^3,\cdots$ ($\subseteq \mathrm{NonTip}(I)$) are $k$-linearly independent, contradicting the finite-dimensionality of $A$. Moreover, $a^{m-1}\in\mathcal{B}$ since $\mathcal{G}$ is tip-reduced.
		\end{proof}
		
		Since the following condition of the characteristic of the field $k$ is useful, we give a precise definition as follows.
		\begin{definition}\label{condition*}
			Let $kQ/I$ be a finite dimensional quiver algebra (where $I\subseteq kQ_{\geq 2}$ but not necessarily homogeneous). We say that the field $k$ satisfies condition $(*)$ for $kQ/I$, if for a given Gr\"{o}bner basis $\mathcal{G}$ of $I$ and every $(a,e)\in Q_1//Q_0$ of $Q$, the characteristic of the field $k$ does not divide the integer $m\geq 2$ for which $a^m\in \mathrm{Tip}(\mathcal{G})$, $a^{m-1}\in\mathcal{B}$.
		\end{definition}
		
		The next result is a proper generalization of some results for monomial algebras in \cite[Proposition 4.2]{Strametz} and it may be useful when one deals with algebras defined by homogeneous ideals.
		
		\begin{proposition} \label{l_{-1}}
			Let ${A}=kQ/I$ be a finite dimensional quiver algebra where $I$ is a homogeneous ideal. Each of the following conditions implies $L_{-1}=0$:
			\begin{enumerate}[(1)]
				
				\item The quiver does not have a loop.
				
				\item  The field $k$ satisfies condition $(*)$ for $A$.
				
				\item  The characteristic of $k$ is equal to $0$.
			\end{enumerate}
		\end{proposition}
		
		\begin{proof}
			\begin{enumerate}[(1)]
				\item: Clear.
				
				\item: Let $(a,e)\in Q_1//Q_0$ where $a$ is a loop in $Q$. We begin to show that $(a,e)\notin \mathrm{Ker}\psi_1$. 
				
				\qquad By Lemma \ref{loop-power}, there exists an integer $m\geq 2$ such that $p:=a^m\in \mathrm{Tip}(\mathcal{G})$ and $a^{m-1}\in\mathcal{B}$. If the characteristic of $k$ does not divide $m$, then $\pi(p^{(a,e)})=ma^{m-1}\neq 0$. Consider $\mathrm{Tip^{-1}}(p)\in\mathcal{G}$, if there is another $p'\in \mathrm{Supp}(\mathrm{Tip^{-1}}(p))$ with $p'\neq p$, such that $a^{m-1}$ is a summand of $\pi(p'^{(a,e)})$, so there exists a decomposition of $p'$ such that $p'=p_1ap_2$, $p_1p_2$ is a summand of $p'^{(a,e)}$ and $a^{m-1}$ is a summand of $\pi(p_1p_2)$. However, since $\mathrm{Tip}(\mathrm{Tip^{-1}}(p))=p\in \mathrm{Tip}(\mathcal{G})$, $p=a^m>p'=p_1ap_2$ under the length-lexicographic order. This will lead to $p_1p_2<a^{m-1}$. This is a contradiction, since $a^{m-1}$ is a summand of $\pi(p_1p_2)$. Therefore, for $g:=\mathrm{Tip^{-1}}(p)\in\mathcal{G}$, we have $\sum_{p\in \mathrm{Supp}(g)}c_g(p)\pi(p^{(a,e)})\neq 0$.
				
				\qquad Then we show that above $(a,e)$ is not a summand of any element in $\mathrm{Ker}\psi_1$.
				
				\qquad If the $(a,e)$ above is a summand of some nonzero element $t\in L_{-1}$, then there exists some $(b_1,e_1)\in k(Q_1//Q_0)$ and is also a summand of $t$, such that for some $q_1\in \mathrm{Supp}(g)$, $l(q_1)=m$ and $a^{m-1}$ is a summand in $\pi(q_1^{(b_1,e_1)})$.
				
				\begin{itemize}
					\item If $a^{m-1}$ is not a summand of $q_1^{(b_1,e_1)}$, since $a^m>q_1$, the only case is that $q_1=a^{m_1}b_1q_1'$ with $l(q_1')=m-m_1-1$, and so we get $q_1'>a^{m-m_1-1}$ and $a>b_1$.
					
					\item If $a^{m-1}$ is a summand of $q_1^{(b_1,e_1)}$, then $q_1=a^{k_1}ba^{k_2}$, with $k_1+k_2=m-1$. By $a^m>a^{k_1}ba^{k_2}$, we also have $a>b_1$.
				\end{itemize}
				
				\qquad Note that $b_1$ is also a loop with $b_1^{n_1}\in \mathrm{Tip}\mathcal{G}$, $b_1^{n_1-1}\in\mathcal{B}$. Since $(b_1,e_1)$ is also a summand of $t$, there exists some $(b_2,e_2)\in k(Q_1//Q_0)$ and is also a summand of $t$, such that for some $q_2\in \mathrm{Supp}(\mathrm{Tip}^{-1}(b_1^{n_1}))$, $l(q_2)=n_1$ and $b_1^{n_1-1}$ is a summand in $\pi(q_2^{(b_2,e_2)})$. Same as before, we have $b_1>b_2$.
				
				\qquad Repeat this process, we get an infinite descending sequence of loops:
				$$a>b_1>b_2>\cdots$$
				This contradicts the fact that $Q$ is a finite quiver. 
				
				\qquad Therefore, there is no nonzero element $t$ in $L_{-1}$ that has a summand $(a,e)$, which implies $L_{-1}=0$.
				
				\item: Clear, because (3) implies (2).
			\end{enumerate}		
		\end{proof}
		
		\begin{remark}
			\begin{enumerate}[(i)]	
				\item The results (1) and (3) in Proposition \ref{l_{-1}} are known even when the ideal $I$ is not homogeneous. See for example \cite[Lemma 2.6]{LR} and \cite[Theorem 4.2]{Hoc}, respectively.
				
				\item It is obvious that if $L_{-1}$ equals $0$, then for every loop $(a,e)\in Q_1//Q_0$, there exists an element $g\in\mathcal{G}$, such that $\sum_{p\in \mathrm{Supp}(g)}c_g(p)\pi(p^{(a,e)})\neq 0$.
				However, unlike the monomial case \cite[Lemma 4.1]{Strametz}, the converse is not true. For example, let $A=k\langle x,y\rangle/\langle x^3+yx^2, xy+yx, y^2\rangle$ with $\mathrm{char}(k)=2$, which is a finite dimensional symmetric algebra with $\langle x^3+yx^2, xy+yx, y^2\rangle$ a homogeneous ideal of $k\langle x,y\rangle$. It can be checked that $\mathcal{G}=\{x^3+yx^2, xy+yx, y^2\}$ is a Gr\"{o}bner basis of the ideal $\langle x^3+yx^2, xy+yx, y^2\rangle$ under the length-lexicographic order with respect to $x>y$ and $\mathrm{Tip}\mathcal{G}=\{x^3,xy,y^2\}$. We have 
				$$\psi_1((x,1))=3(x^3,x^2)+2(x^3,yx)+2(xy,y)=(x^3,x^2)\neq 0,$$
				$$\psi_1((y,1))=(x^3,x^2)+2(xy,x)+2(y^2,y)=(x^3,x^2)\neq 0.$$
				But $(x,1)+(y,1)\in L_{-1}$ and so $L_{-1}\neq 0$.
				
				\item For BGAs (whenever the ideal $I$ is homogeneous or not), we can deduce a stronger conclusion from the condition $(2)$ in Proposition \ref{l_{-1}}, see Proposition \ref{BGA L-1}. 
				
			\end{enumerate}	
		\end{remark}

		\section{Brauer graph algebras}
		
		Brauer graph algebras form an important class of finite dimensional tame algebras (see for example, Schroll \cite{SS}). In recent years, there has been a renewed interest in these algebras. Since Brauer graph algebras are in general not monomial algebras, the parallel paths method developed by Strametz in \cite{Strametz} cannot be directly applied to study the first Hochschild cohomology of these algebras. In this section, we will use our generalized parallel paths method to study the first Hochschild cohomology groups of Brauer graph algebras.
		
		\subsection{Brauer graph algebras}

		\begin{definition}$(\cite[~Definition~2.1]{SS})$
			A Brauer graph $G$ is a tuple $G=(V,E, m, o)$ where
			\begin{itemize}
				\item $(V,E)$ is a finite (unoriented) connected graph with vertex set $V$ and edge set $E$.
				
				\item $m:V\rightarrow\mathbb{Z}_{>0}$ is a function, called the multiplicity or $multiplicity \; function$ of $G$.
				
				\item $o$ is called the orientation of $G$ which is given, for every vertex $v\in V$, by a cyclic ordering of the edges incident with $v$ such that if $v$ is a vertex incident to a single edge $i$ then if $m(v)=1$, the cyclic ordering at $v$ is given by $i$ and if $m(v)>1$ the cyclic ordering at $v$ is given by $i<i$.
			\end{itemize}
		\end{definition}

		In a concrete Brauer graph, we often use the sign $[m(v)]$ to denote the multiplicity of $v$. We note that the Brauer graph $G=(V,E)$ may contain loops and multiple edges. The valency of a vertex $v\in V$, denote by $val(v)$, is the number of edges in $G$ incident to $v$ with the convention that a loop is counted twice. We call the edge $i$ with vertex $v$ truncated at $v$ if $m(v)val(v)=1$.

		Given a Brauer graph $G=(V,E,m,o)$, we can define a quiver $Q_G=(Q_0,Q_1)$ as follows:
		$$Q_0:=E,$$	
		$$Q_1:=\{i\rightarrow j\;|\;i,j\in E,\;\text{there exists $ v\in V$, such that $i<j$ belong to $o(v)$}\}.$$

		For $i\in E$, if $v\in V$ is a vertex of $i$ and $i$ is not truncated at $v$, then there is a special $i$-cycle $C_v(\alpha)$ at $v$ which is an oriented cycle given by $o(v)$ in $Q_G$ with the starting arrow $\alpha$ (where the starting vertex of $\alpha$ in $Q_G$ is $i$). Note that if $i$ is a loop at $v$, there are exactly two special $i$-cycles at $v$; for an example, see \cite[Example 2.3 (1)]{GL}. However, if we do not care about the starting arrows of the special cycles, then there is a unique special cycle (up to cyclic permutation) at each non-truncated vertex $v$ and we will just denote it by $C_v$. The corresponding path algebra of the above quiver $Q_G$ is denoted by $kQ_G$.
		
		We now define an ideal $I_G$ in $kQ_G$ generated by three types of relations. For this recall that we identify the set of edges $E$ of a Brauer graph $G$ with the set of vertices $Q_0$ of the corresponding quiver $Q_G$ and that we denote the set of vertices of the Brauer graph by $V$.
		
		\begin{itemize}
			\item Relations of type \uppercase\expandafter{\romannumeral1}
			$$C_v(\alpha)^{m(v)}-C_{v'}(\alpha')^{m(v')}$$
			for any $i\in Q_0$ and for any special $i$-cycles $C_v(\alpha)$ and $C_{v'}(\alpha')$ at $v$ and $v'$, respectively, such that both $v$ and $v'$ are not truncated.
			
			\item Relations of type \uppercase\expandafter{\romannumeral2}
			$$\alpha C_v(\alpha)^{m(v)}$$
			for any $i\in Q_0$, any $v\in V$ and where $C_v(\alpha)$ is a special $i$-cycle at $v$ with starting arrow $\alpha$.
			
			\item Relations of type \uppercase\expandafter{\romannumeral3}
			$$\beta\alpha$$
			for any $i\in Q_1$ such that $\beta\alpha$ is not a subpath of any special cycle except if $\beta=\alpha$ is a loop associated with a vertex $v$ of valency one and multiplicity $m(v)>1$.
		\end{itemize}
		The quotient algebra $A=kQ_G/I_G$ is called the Brauer graph algebra (abbr. BGA) of the Brauer graph $G$. In fact, $A$ is a finite dimensional symmetric algebra, that is, $A$ is a finite dimensional algebra such that $A\cong \mathrm{Hom}_k(A,k)$ as $A$-$A$-bimodules. Moreover, $A$ is also a special biserial algebra, which means that $A$ satisfies the following conditions:
		
		\begin{enumerate}[(1)]
			\item At every vertex $i$ in $Q_G$, there are at most two arrows starting at $i$ and there are at most two arrows ending at $i$.
			
			\item For every arrow $\alpha$ in $Q_G$, there exists at most one arrow $\beta$ such that $\beta\alpha\notin I_G$ and there exists at most one arrow $\gamma$ such that $\alpha\gamma\notin I_G$.
		\end{enumerate}
		
		We will also use the following notion on a Brauer graph $G$ from Guo and Liu \cite{GL}.	
		
		\begin{definition}$(\cite[~Definition~2.4]{GL})$
			For each vertex $v$ in a Brauer graph $G$, we define the graded degree grd$(v)$ as follows. If $val(v)=1$, we denote by $v'$ the unique vertex adjacent to $v$. If $G$ is given by a single edge with both vertices $v$ and $v'$ of multiplicity $1$, then $grd(v)=grd(v')=1$; Otherwise
			
			$$grd(v)= \left\{
			\begin{array}{*{3}{lll}}
				m(v)val(v), & \text{if $m(v)val(v)>1$;}\\
				grd(v'), & \text{if $m(v)val(v)=1$.}
			\end{array}
			\right.$$
		\end{definition}
		
		\subsection{The Gr{\"o}bner basis of the ideal $I_G$}

		Let $A=kQ_G/I_G$ be a BGA. In order to use the generalized parallel paths method on $A$, we need to find a Gr{\"o}bner basis of the ideal $I_G$ in $kQ_G$. Throughout, we will use the left length-lexicographic order on the set of paths of $Q_G$. We first show that the three types of generating relations in $I_G$ already give a Gr{\"o}bner basis.
		
		For convenience, for a quiver $Q$, $p,q\in Q_{\geq 0}$, we write $p\mid q$  if $p$ is a subpath of $q$.
		
		\begin{proposition}\label{Grobner basis of BGA}
			Let $G$ be a Brauer graph and $Q_G$ be the quiver of $G$, $$R_i=\{\text{the $i$-th type of relations in $kQ_G$}\}.$$
			Then the ideal $I_G=\langle R_1\cup R_2\cup R_3\rangle$ of $kQ_G$ has a Gr{\"o}bner basis $R_1\cup R_2\cup R_3$.
		\end{proposition}
		
		\begin{proof}
			By Theorem \ref{Termin.Thm}, it suffices to check that all overlap relations of elements in $R_1\cup R_2\cup R_3$ can be divided by the set $R_1\cup R_2\cup R_3$. Without loss of generality, for each $f$ in $R_1$, we assume that under the fixed left length-lexicographic order, the coefficient of $\mathrm{Tip}(f)$ is equal to $1$.
			
			\begin{itemize}
				\item Let $ f,g\in R_2\cup R_3$. Then $f,g$ are monomial relations. Hence $o(f,g,*,*)=0$.
				
				\item Let $ f=C_v(\alpha)^{m(v)}-C_{v'}(\alpha')^{m(v')}\in R_1$, $ g\in R_2\cup R_3$.
				If there exists $ b,c\in Q_{G\geq 0}$, such that $o(f,g,b,c)\neq 0$, then
				$$o(f,g,b,c)=-C_{v'}(\alpha')^{m(v')}\cdot c.$$
				
				\item Let $ f\in R_2\cup R_3$ and $ g=C_v(\alpha)^{m(v)}-C_{v'}(\alpha')^{m(v')}\in R_1$.
				If there exists $ b,c\in Q_{G\geq 0}$, such that $o(f,g,b,c)\neq 0$, then
				$$o(f,g,b,c)=b\cdot C_{v'}(\alpha')^{m(v')}.$$
				
				\item Let $ f=C_v(\alpha)^{m(v)}-C_{v_1}(\alpha')^{m(v_1)}, g=C_{v_2}(\alpha)^{m(v_2)}-C_{v_3}(\alpha')^{m(v_3)}\in R_1$.
				If there exists $ b,c\in Q_{G\geq 0}$, such that $o(f,g,b,c)\neq 0$, then
				$$o(f,g,b,c)=b\cdot C_{v_3}(\alpha')^{m(v_3)}-C_{v_1}(\alpha')^{m(v_1)}\cdot c.$$
			\end{itemize}
			By the construction of $I_G$, for a path $p\in Q_{\geq 0}$ that satisfies $p\mid C_v^{m(v)}$, if $p\cdot C_v(\alpha)^{m(v)}\neq 0$ in $kQ_G$ (respectively, $C_v(\alpha)^{m(v)}\cdot p\neq 0$ in $kQ_G$), then $p\cdot C_v(\alpha)^{m(v)}\in\langle R_2\rangle$ (respectively, $C_v(\alpha)^{m(v)}\cdot p\in\langle R_2\rangle$). Similarly, for a path $p\in Q_{\geq 0}$ that satisfies $p\nmid C_v^{m(v)}$, if $p\cdot C_v(\alpha)^{m(v)}\neq 0$ in $kQ_G$ (respectively, $C_v(\alpha)^{m(v)}\cdot p\neq 0$ in $kQ_G$), then $p\cdot C_v(\alpha)^{m(v)}\in\langle R_3\rangle$ (respectively, $C_v(\alpha)^{m(v)}\cdot p\in\langle R_3\rangle$).
			
			Therefore, all the nonzero overlap relations with respect to $R_1\cup R_2\cup R_3$ can be divided by some monomial relations in $R_2\cup R_3$. Hence $R_1\cup R_2\cup R_3$ is a Gr{\"o}bner basis of $I_G$ by Theorem \ref{Termin.Thm}.		
		\end{proof}
		
		\begin{remark}
		The above Gr{\"o}bner basis of $I_G$ may not be tip-reduced in general. Some relations of type \uppercase\expandafter{\romannumeral2} can  be reduced by the relations of type \uppercase\expandafter{\romannumeral1} and type \uppercase\expandafter{\romannumeral3}. However, with respect to an admissible order on $kQ_G$, we can reduce $R_1\cup R_2\cup R_3$ to a reduced Gr{\"o}bner basis (which we denote by $\mathcal{G}$) consisting of $R_1$, $R_3$ and part of $R_2$.
		\end{remark}
		
		The following proposition is a strengthened and generalized version of Proposition \ref{l_{-1}} for BGAs.

		\begin{proposition}\label{BGA L-1}
			Let ${A}$ be a BGA associated with a Brauer graph $G$. Suppose the field $k$ satisfies condition $(*)$ for $A$. Then every nonzero element in $\mathrm{Ker}\psi_1$ will not have summands in $k(Q_1//Q_0)$.
		\end{proposition}
		
		\begin{proof}
			First we recall from Proposition \ref{gen-parallel paths} that the map $\psi_1: k(Q_1//\mathcal{B})\rightarrow k(\mathrm{Tip}(\mathcal{G})//\mathcal{B})$ is defined by $$(\alpha,\gamma)\mapsto \sum_{g\in\mathcal{G}}\sum_{p\in \mathrm{Supp}(g)}c_g(p)\cdot(\mathrm{Tip}(g),\pi(p^{(\alpha,\gamma)}))$$
			where $g=\sum_{p\in \mathrm{Supp}(g)}c_g(p)p$, $c_g(p)\in k$, and $\mathcal{B}$ is identified with $\mathrm{NonTip}(I)$ as before.
			
			Now let $(a,e)\in Q_1//Q_0$ with $a$ a loop of $Q$. Since ${A}$ is finite dimensional,  by Lemma \ref{loop-power}, there exists an integer $m\geq 2$ such that $p:=a^m\in \mathrm{Tip}(\mathcal{G})$ and $a^{m-1}\in\mathcal{B}$. If the characteristic of $k$ does not divide $m$, then $\pi(p^{(a,e)})=ma^{m-1}$ is different from $0$. Now consider a special cycle $C_v$ with $a\mid C_v$.
			
			\begin{itemize}
				\item If there is another arrow in $C_v$ different from $a$, then according to the relations of the third type, $a^2\in\mathcal{G}$, and $2(a^2,a)$ will not appear in any summand of other elements in $\mathrm{Im}\psi_1$.
				
				\item If $a$ is the unique arrow in $C_v$, then by the definition of BGA and the property of the reduced Gr\"{o}bner basis, exactly one of the following two cases holds: $a^{m(v)+1}\in\mathcal{G}$ or $a^{m(v)}\in\mathrm{Tip}\mathcal{G}$. 
				In the first case, $(m(v)+1)(a^{m(v)+1},a^{m(v)})$ will not appear in any summand of other elements in $\mathrm{Im}\psi_1$. 
				In the second case, $\mathrm{Tip}^{-1}(a^{m(v)})=a^{m(v)}-q$ where $q$ is also a special cycle in $A$. Therefore, it does not exist any $(b,\varepsilon)\in k(Q_1//\mathcal{B})$, such that $m(v)(a^{m(v)},a^{m(v)-1})$ becomes a summand of $\psi_1((b,\varepsilon))$.
			\end{itemize}	
		\end{proof}
		
		\begin{remark}
			When $I_G$ is a homogeneous ideal of $kQ_G$, the conclusion of Proposition \ref{BGA L-1} is equal to $L_{-1}=0$, which can also be deduced from Proposition \ref{l_{-1}}.
		\end{remark}

		\subsection{Lie algebra structure on the first Hochschild cohomology group of a BGA}
		
		In this subsection, we assume that the characteristic of the ground field $k$ is $0$, unless otherwise stated. We often write simply $Q$ for $Q_G$.	
		
		Let $A=kQ/I$ be a BGA. Consider (for $\psi_1$, see Proposition \ref{gen-parallel paths})
		$$L_{0}:=k(Q_1//Q_1)\cap \mathrm{Ker}\psi_1\bigg/ \langle\sum_{a\in Q_1e}(a,a)-\sum_{b\in eQ_1}(b,b)|e\in Q_0\rangle,$$ which is a Lie subalgebra of $\mathrm{HH}^1(A)$.
		Furthermore, we can consider $L_{00}$ which is given by
		$$L_{00}:=\langle(\alpha,\alpha)|\alpha\in Q_1\rangle\cap \mathrm{Ker}\psi_1\bigg/ \langle\sum_{a\in Q_1e}(a,a)-\sum_{b\in eQ_1}(b,b)|e\in Q_0\rangle.$$
		Note that $L_{00}$ is an abelian Lie subalgebra of $\mathrm{HH}^1(A)$ and $L_{00}\subseteq L_0$.
		
		\begin{lemma}\label{dimL00 of A}
			Let $G=(V,E)$ be a Brauer graph and $A=kQ/I$ be the corresponding BGA. Then $$\mathrm{dim}_kL_{00}=|E|-|V|+2.$$
		\end{lemma}
		
		\begin{proof}
			Let $g\in\mathcal{G}$ be an element in the reduced Gr{\"o}bner basis of $I_G$. Then $g$ is a relation of type \uppercase\expandafter{\romannumeral1} or $g$ is a monomial relation. 
			For each $\alpha\in Q_1$, fix $(\alpha,\alpha)\in Q_1//Q_1$. We study $\psi_1(\alpha,\alpha)$ through its summands  $\sum_{p\in \mathrm{Supp}(g)}c_g(p)\cdot(\mathrm{Tip}(g),\pi(p^{(\alpha,\alpha)}))$ for each $g=\sum_{p\in Q_{\geq 0}}c_g(p)p$.
			
			{\it Case 1.} If $g$ is a monomial relation, then $(g,\pi(g^{(\alpha,\alpha)}))=0$. 
			
			{\it Case 2.} If $g$ is a relation of type \uppercase\expandafter{\romannumeral1}, then $g=C_v(\beta)^{m(v)}-C_w(\gamma)^{m(w)}$ where $C_v(\beta),C_w(\gamma)$ are special cycles in the Brauer graph $G$ such that both $v$ and $w$ are not truncated. Without loss of generality, let $\mathrm{Tip}(g)=C_v(\beta)^{m(v)}$.
			
			\begin{itemize}
				\item If $\alpha\mid C_v(\beta)$ and $\alpha\nmid C_w(\gamma)$,
				$$\sum_{p\in \mathrm{Supp}(g)}c_g(p)\cdot(\mathrm{Tip}(g),\pi(p^{(\alpha,\alpha)}))=m(v)(C_v(\beta)^{m(v)},C_w(\gamma)^{m(w)});$$
				
				\item if $\alpha\nmid C_v(\beta)$ and $\alpha\mid C_w(\gamma)$,
				$$\sum_{p\in \mathrm{Supp}(g)}c_g(p)\cdot(\mathrm{Tip}(g),\pi(p^{(\alpha,\alpha)}))=-m(w)(C_v(\beta)^{m(v)},C_w(\gamma)^{m(w)});$$
				
				\item if $\alpha\mid C_v(\beta)$ and $\alpha\mid C_w(\gamma)$, then $v=w$, and
				$$\sum_{p\in \mathrm{Supp}(g)}c_g(p)\cdot(\mathrm{Tip}(g),\pi(p^{(\alpha,\alpha)}))=(m(v)-m(w))(C_v(\beta)^{m(v)},C_w(\gamma)^{m(w)})=0.$$
			\end{itemize}
			Therefore, each element in $\langle (\alpha,\alpha)|\alpha\in Q_1\rangle\cap \mathrm{Ker}\psi_1$ is a linear combination of the following two types of elements:
			\begin{itemize}
				\item $\sum_{v\in V^*}k_v(\alpha_{v},\alpha_{v})$, where $\alpha_{v}$ is an arbitrary arrow in the special cycle at $v$,
				$$V^*=\{v\in V\;|\; v \text{ is not truncated}\}\text{, and }k_v=\prod_{w\in V^*, w\neq v}m(w).$$
				
				\item $(\alpha_{j},\alpha_{j})-(\alpha_{0},\alpha_{0})$ ($j=1,\cdots,val(v)-1$), where the special cycle $C_v=\alpha_{val(v)-1}\cdots\alpha_{0}$ is fixed for each vertex $v\in V^*$;
			\end{itemize}
			For each $v\in V^*$, we have $(val(v)-1)$ elements of the second type. Although we can choose different $\alpha_v$ for each element of the first type, after reducing by the elements of the second type above, we get exactly one representative element of the first type. It is straightforward to show that these $(1+\sum_{v\in V^*}(val(v)-1))$ elements are $k$-linearly independent by simple linear algebra argument. Thus these elements form a $k$-basis of $\langle (\alpha,\alpha)|\alpha\in Q_1\rangle\cap \mathrm{Ker}\psi_1$. Furthermore, we have
			
			$$
			\begin{array}{*{3}{lll}}
				\mathrm{dim}_k(\langle (\alpha,\alpha)|\alpha\in Q_1\rangle\cap \mathrm{Ker}\psi_1) & =& 1+\sum_{v\in V^*}(val(v)-1)\\
				&=& 1+2|E|-(|V|-|V^*|)-|V^*|\\
				&=&2|E|-|V|+1.
			\end{array}
			$$
			
			Since the quiver $Q_G$ is connected, we can check that $\sum_{e\in Q_0}\psi_0(e,e)=0$ and for a proper subset $S_0$ of $Q_0$, the elements $\psi_0(e',e')$ ($e'\in S_0$) are linearly independent. (This is the content of Lemma 3.4 in \cite{LRW}.) Therefore, $$\mathrm{dim}_k(\langle \sum_{a\in Q_1e}(a,a)-\sum_{b\in eQ_1}(b,b)|e\in Q_0\rangle)=|Q_0|-1=|E|-1.$$	
			As a result, we get
			$$
			\begin{array}{*{3}{lll}}
				\mathrm{dim}_kL_{00} & =&\mathrm{dim}_k(\langle (\alpha,\alpha)|\alpha\in Q_1\rangle\cap \mathrm{Ker}\psi_1)- \mathrm{dim}_k(\langle \sum_{a\in Q_1e}(a,a)-\sum_{b\in eQ_1}(b,b)|e\in Q_0\rangle)\\
				&=& (2|E|-|V|+1)-(|E|-1)\\
				&=&|E|-|V|+2
			\end{array}
			$$\end{proof}
		
		We now describe a $k$-basis of $\mathrm{HH}^1(A)$ for any BGA, say $A$ as follows.
		
		\begin{theorem}\label{gen-set of A}
			Let $A$ be a BGA. There is a $k$-basis $\mathcal{B}_{L,Ker}$  of $\mathrm{Ker}\psi_1$ which consists of the following five subsets of $\mathrm{Ker}\psi_1$:
			\begin{itemize}
				\item $S_1$: the basis of $\langle (\alpha,\alpha)|\alpha\in Q_1\rangle\cap \mathrm{Ker}\psi_1$ described in Lemma \ref{dimL00 of A}.
				
				\item $S_2$: elements of the form $(\beta_0,\hat{\alpha_k})-(\alpha_{k},\hat{\beta_0})$ with $\alpha_{k}\neq\beta_0$, $\hat{\alpha_k}:=\alpha_{k-1}\cdots\alpha_{0}\cdot C_v(\alpha_{0})^{m(v)-1}$, $\hat{\beta_0}:=\beta_m\cdots\beta_{1}\cdot C_w(\beta_{1})^{m(w)-1}$,  $\alpha_{k}\beta_0,\beta_0\alpha_{k}\in R_3$. The corresponding Brauer graph is given by: 
				\begin{figure}[H]
					
					\centering
					\tikzset{every picture/.style={line width=0.75pt}} 
					
					\begin{tikzpicture}[x=0.75pt,y=0.75pt,yscale=-1,xscale=1]
						
						\draw  [fill={rgb, 255:red, 210; green, 210; blue, 210 }  ,fill opacity=1 ] (80,95) .. controls (80,86.72) and (86.72,80) .. (95,80) .. controls (103.28,80) and (110,86.72) .. (110,95) .. controls (110,103.28) and (103.28,110) .. (95,110) .. controls (86.72,110) and (80,103.28) .. (80,95) -- cycle ;
						\draw  [fill={rgb, 255:red, 210; green, 210; blue, 210 }  ,fill opacity=1 ] (230,95) .. controls (230,86.72) and (236.72,80) .. (245,80) .. controls (253.28,80) and (260,86.72) .. (260,95) .. controls (260,103.28) and (253.28,110) .. (245,110) .. controls (236.72,110) and (230,103.28) .. (230,95) -- cycle ;
						\draw    (107,87.5) .. controls (146,63) and (190,61) .. (233,86.5) ;
						\draw    (108,103.5) .. controls (148,124) and (189,126.5) .. (231,104) ;
						\draw    (128,103.5) .. controls (133.64,95.04) and (130.44,95.42) .. (126.72,86.36) ;
						\draw [shift={(126,84.5)}, rotate = 70.02] [color={rgb, 255:red, 0; green, 0; blue, 0 }  ][line width=0.75]    (8.74,-2.63) .. controls (5.56,-1.12) and (2.65,-0.24) .. (0,0) .. controls (2.65,0.24) and (5.56,1.12) .. (8.74,2.63)   ;
						\draw    (97,130.5) .. controls (113.24,130.5) and (114.88,128.68) .. (124.57,119.8) ;
						\draw [shift={(126,118.5)}, rotate = 137.73] [color={rgb, 255:red, 0; green, 0; blue, 0 }  ][line width=0.75]    (8.74,-2.63) .. controls (5.56,-1.12) and (2.65,-0.24) .. (0,0) .. controls (2.65,0.24) and (5.56,1.12) .. (8.74,2.63)   ;
						\draw    (122,71.5) .. controls (114.24,60.83) and (117.77,64.27) .. (98.82,57.19) ;
						\draw [shift={(97,56.5)}, rotate = 20.85] [color={rgb, 255:red, 0; green, 0; blue, 0 }  ][line width=0.75]    (10.93,-3.29) .. controls (6.95,-1.4) and (3.31,-0.3) .. (0,0) .. controls (3.31,0.3) and (6.95,1.4) .. (10.93,3.29)   ;
						\draw    (209,82.5) .. controls (205.2,94.85) and (205.91,92.76) .. (211.99,102.81) ;
						\draw [shift={(213,104.5)}, rotate = 239.74] [color={rgb, 255:red, 0; green, 0; blue, 0 }  ][line width=0.75]    (10.93,-3.29) .. controls (6.95,-1.4) and (3.31,-0.3) .. (0,0) .. controls (3.31,0.3) and (6.95,1.4) .. (10.93,3.29)   ;
						\draw    (218,115.5) .. controls (223.79,127.08) and (229.58,128.42) .. (244.36,135.68) ;
						\draw [shift={(246,136.5)}, rotate = 206.57] [color={rgb, 255:red, 0; green, 0; blue, 0 }  ][line width=0.75]    (10.93,-3.29) .. controls (6.95,-1.4) and (3.31,-0.3) .. (0,0) .. controls (3.31,0.3) and (6.95,1.4) .. (10.93,3.29)   ;
						\draw    (245,50.5) .. controls (227.54,52.44) and (227.96,53.44) .. (215.22,66.27) ;
						\draw [shift={(214,67.5)}, rotate = 315] [color={rgb, 255:red, 0; green, 0; blue, 0 }  ][line width=0.75]    (10.93,-3.29) .. controls (6.95,-1.4) and (3.31,-0.3) .. (0,0) .. controls (3.31,0.3) and (6.95,1.4) .. (10.93,3.29)   ;
						
						\draw (90,90) node [anchor=north west][inner sep=0.75pt]   [align=left] {v};
						\draw (238,90) node [anchor=north west][inner sep=0.75pt]   [align=left] {w};
						\draw (35,90) node [anchor=north west][inner sep=0.75pt]  [font=\Large] [align=left] {......};
						\draw (267,90) node [anchor=north west][inner sep=0.75pt]  [font=\Large] [align=left] {......};
						\draw (113,129) node [anchor=north west][inner sep=0.75pt]   [align=left] {$\alpha_{k-1}$};
						\draw (137,87) node [anchor=north west][inner sep=0.75pt]   [align=left] {$\alpha_k$};
						\draw (113,47) node [anchor=north west][inner sep=0.75pt]   [align=left] {$\alpha_{0}$};
						\draw (207,129) node [anchor=north west][inner sep=0.75pt]   [align=left] {$\beta_{1}$};
						\draw (189,87) node [anchor=north west][inner sep=0.75pt]   [align=left] {$\beta_0$};
						\draw (203,45) node [anchor=north west][inner sep=0.75pt]   [align=left] {$\beta_m$};
						\draw (160,57) node [anchor=north west][inner sep=0.75pt]   [align=left] {$e_1$};
						\draw (160,123) node [anchor=north west][inner sep=0.75pt]   [align=left] {$e_2$};

					\end{tikzpicture}
					\caption{The subgraph corresponding to some element in $S_2$}
					\label{S2-subgraph}
				\end{figure}
				
				\item $S_3$: elements of the form $(\beta,C_v(\alpha)^{m(v)})$ with $C_w(\beta)^{m(w)}=\beta^{m(w)}>C_v(\alpha)^{m(v)}$, where the corresponding Brauer graph is given by:  
				
				\begin{figure}[H]
					\centering
					\tikzset{every picture/.style={line width=0.75pt}} 
					\begin{tikzpicture}[x=0.75pt,y=0.75pt,yscale=-1,xscale=1]
						
						\draw    (100,108) -- (200,108) ;
						\draw  [fill={rgb, 255:red, 210; green, 210; blue, 210 }  ,fill opacity=1 ] (75.5,108) .. controls (75.5,101.23) and (80.98,95.75) .. (87.75,95.75) .. controls (94.52,95.75) and (100,101.23) .. (100,108) .. controls (100,114.77) and (94.52,120.25) .. (87.75,120.25) .. controls (80.98,120.25) and (75.5,114.77) .. (75.5,108) -- cycle ;
						\draw  [fill={rgb, 255:red, 210; green, 210; blue, 210 }  ,fill opacity=1 ] (200,108) .. controls (200,101.23) and (205.48,95.75) .. (212.25,95.75) .. controls (219.02,95.75) and (224.5,101.23) .. (224.5,108) .. controls (224.5,114.77) and (219.02,120.25) .. (212.25,120.25) .. controls (205.48,120.25) and (200,114.77) .. (200,108) -- cycle ;
						\draw    (120,99.5) .. controls (86,50.5) and (53,84.5) .. (55,111.5)(55,108) .. controls (56.98,131.27) and (85.42,159.92) .. (117.04,120.71) ;
						\draw [shift={(118,119.5)}, rotate = 127.97] [color={rgb, 255:red, 0; green, 0; blue, 0 }  ][line width=0.75]    (10.93,-3.29) .. controls (6.95,-1.4) and (3.31,-0.3) .. (0,0) .. controls (3.31,0.3) and (6.95,1.4) .. (10.93,3.29)   ;
						\draw    (212,72.5) .. controls (199.46,71.54) and (193.43,79.88) .. (189.43,95.74) ;
						\draw [shift={(189,97.5)}, rotate = 283.24] [color={rgb, 255:red, 0; green, 0; blue, 0 }  ][line width=0.75]    (10.93,-3.29) .. controls (6.95,-1.4) and (3.31,-0.3) .. (0,0) .. controls (3.31,0.3) and (6.95,1.4) .. (10.93,3.29)   ;
						\draw    (188,121.5) .. controls (193.85,137.1) and (193.05,134.64) .. (213.39,140.99) ;
						\draw [shift={(215,141.5)}, rotate = 197.65] [color={rgb, 255:red, 0; green, 0; blue, 0 }  ][line width=0.75]    (10.93,-3.29) .. controls (6.95,-1.4) and (3.31,-0.3) .. (0,0) .. controls (3.31,0.3) and (6.95,1.4) .. (10.93,3.29)   ;
						
						\draw (241,108) node [anchor=north west][inner sep=0.75pt]   [align=left] {......};
						\draw (81,105) node [anchor=north west][inner sep=0.75pt]   [align=left] {w};
						\draw (208,105) node [anchor=north west][inner sep=0.75pt]   [align=left] {v};
						\draw (38,102) node [anchor=north west][inner sep=0.75pt]   [align=left] {$\beta$};
						\draw (180,140) node [anchor=north west][inner sep=0.75pt]   [align=left] {$\alpha$};
					\end{tikzpicture}
						\caption{The subgraph corresponding to some element in $S_3$}
				\end{figure}
				
				\item $S_4$: elements of the form ($\alpha$,$p$) satisfying
				\begin{itemize}
					\item $l(p)>1$;
					\item $\psi_1(\alpha,p)=0$;
					\item there exists a special cycle $C_v$, such that $\alpha\mid C_v$, $p\mid C_v^{m(v)}$.
				\end{itemize}
				
				\item $S_5$: the basis of the subspace of $\mathrm{Im}\psi_0$ generated by all the elements $(\alpha_1,\alpha_1 p)-(\alpha_0,p\alpha_0)$ where $p$ is a cycle in $Q$ and $\psi_1((\alpha_1,\alpha_1 p)) =\psi_1((\alpha_0,p\alpha_0))\neq 0$.
			\end{itemize}
			Furthermore, $\mathrm{Im}\psi_0$ is contained in $\langle S_1\cup S_4\cup S_5\rangle$. The $k$-basis of $\mathrm{HH}^1(A)$ induced from $\mathcal{B}_{L,Ker}$ will be denoted by $\mathcal{B}_L$.
		\end{theorem}
		
		\begin{proof}
			By the definition, it is straightforward that the elements in the set $S_1\cup S_2\cup S_3\cup S_4\cup S_5$ are linearly independent and contained in $\mathrm{Ker}\psi_1$. We just verify that the elements in $S_2$ are contained in $\mathrm{Ker}\psi_1$. Without loss of generality, assume that $C_v(\alpha_0)^{m(v)}>C_w(\beta_1)^{m(w)}$ and $C_v(\alpha_k)^{m(v)}>C_w(\beta_0)^{m(w)}$. Then we have both $\alpha_{k}\beta_0$ and $\beta_0\alpha_{k}$ are in $R_3$, and
			$$\psi_1(\beta_0,\hat{\alpha_k})=(\alpha_k\beta_0,\pi(C_v(\alpha_0)^{m(v)}))+(\beta_0\alpha_k,\pi(C_v(\alpha_k)^{m(v)}))$$
			$$=(\alpha_k\beta_0,C_w(\beta_0)^{m(w)})+(\beta_0\alpha_k,C_w(\beta_1)^{m(w)}),$$
			$$\psi_1(\alpha_{k},\hat{\beta_0})=(\beta_0\alpha_k,C_w(\beta_1)^{m(w)})+(\alpha_k\beta_0,C_w(\beta_0)^{m(w)}).$$
			This shows that $(\beta_0,\alpha_{k-1}\cdots\alpha_{0}\cdot C_v(\alpha_{0})^{m(v)-1})-(\alpha_{k},\beta_m\cdots\beta_{1}\cdot C_w(\beta_{1})^{m(w)-1})\in \mathrm{Ker}\psi_1$.
			
			Now consider $t=\sum_{i\in I}c_i\cdot(\alpha_{i},p_i)\in \mathrm{Ker}\psi_1$ with $c_i\in k$, and let $(\alpha,p)=(\alpha_{i},p_i)$ be a summand of $t$. By Proposition \ref{BGA L-1}, we can assume that $(\alpha,p)\in Q_1//(\mathcal{B}\backslash Q_0)$.
			
			First we suppose that $\psi_1(\alpha,p)\neq 0$. This implies that there exists some $g\in\mathcal{G}$, $$\sum_{p'\in \mathrm{Supp}(g)}c_g(p')\cdot(\mathrm{Tip}(g),\pi(p'^{(\alpha,p)}))\neq 0.$$
			Note that $\psi_1(g,\pi(g^{(\alpha,p)}))=0$ when $g\in R_2$.
			
			If there exist $ g\in R_1\subseteq\mathcal{G}$ and $ p_0\in \mathrm{Supp}(g)$, such that $(\mathrm{Tip}(g),\pi(p_0^{(\alpha,p)}))\neq 0$, then $l(p)=1$. Since $A$ is special biserial, $\pi(p_0^{(\alpha,p)})\neq 0$ implies $\alpha=p$. By definition of $\psi_1(\alpha,\alpha)$ (see the proof of Lemma \ref{dimL00 of A} for further details), there must exist a summand of $t$, which is given by $$t'\in kS_1=\langle (\alpha,\alpha)|\alpha\in Q_1\rangle\cap \mathrm{Ker}\psi_1,$$ such that $(\alpha,\alpha)$ is a summand of $t'$.
			
			If  there exists $ g\in R_3\subseteq\mathcal{G}$, such that  $(g,\pi(g^{(\alpha,p)}))\neq 0$, then there are two cases to be considered. If $\pi(g^{(\alpha,p)})\neq g^{(\alpha,p)}$, or if $\pi(g^{(\alpha,p)})= g^{(\alpha,p)}$ and $(\alpha, p)$ is a summand of the elements in $S_2$, there will be an element in $S_2$ having a summand which is equal to $(\alpha,p)$, therefore $(\alpha,p)$ must appear in a summand $t''$ of $t$ with $t''\in kS_2$.  If $\pi(g^{(\alpha,p)})= g^{(\alpha,p)}$ and $(\alpha, p)$ is not a summand of the elements in $S_2$, without loss of generality, we can just let $g=\beta\alpha$, so $(\beta\alpha,\pi(g^{(\alpha,p)})=(\beta\alpha,\beta p)$. Then $t$ must contain a summand of the form $c\cdot((\beta,q)-(\alpha,p))$ with $\beta p=q\alpha$ and $c\in k$. This leads to $p=p_1\alpha,q=\beta p_1$ with $p_1\in Q_{\geq 1}$. Since $A$ is a special biserial algebra, we have $(\beta,q)-(\alpha,p)=\psi_0(e,p_1)$ for some $e\in Q_0$. So $((\beta,q)-(\alpha,p))\in S_5$.
			
			Next we suppose that $\psi_1(\alpha,p)=0$. We will show that in this case $(\alpha,p)\in S_3\cup S_4$. Suppose $(\alpha,p)\notin S_4$, and assume that $l(p)=1$ with $p\mid C_v$. This means that $p$ is a parallel arrow of $\alpha$. Since $A$ is special biserial, let $p'$ be the only arrow with $p'p\in\mathcal{B}$. Then we have $p'\alpha\in R_3\subseteq\mathcal{G}$. However, we have $(p'\alpha,\pi((p'\alpha)^{(\alpha,p)}))=(p'\alpha,p'p)\neq 0$, which contradicts the fact that $\psi_1(\alpha,p)=0$. Therefore, $l(p)>1$. Since $(\alpha,p)\notin S_4$, then $p\nmid C_v^{m(v)}$. But in this case, if $(\alpha,p)\notin S_3$, there will exist a relation $g\in R_3\subseteq\mathcal{G}$, such that $(g,\pi(g^{(\alpha,p)}))\neq 0$, which contradicts the fact that $\psi_1(\alpha,p)=0$.
			
			Summarizing the above discussion, we know that $t\in k(S_1\cup S_2\cup S_3\cup S_4\cup S_5)$.
		\end{proof}
		
		The basis above can help us to compute the dimension of the first Hochschild cohomology group. In a special case, we have the following formula for this dimension.
		
		\begin{corollary}\label{dim-noloop}
			Let $A$ be a BGA whose corresponding Brauer graph $G=(V,E)$ does not have loops. Then $\mathrm{dim}_k\mathrm{HH}^1(A)=|E|-2|V|+\sum_{v\in V}m(v)+|S_2|+2$.
		\end{corollary}
		
		\begin{proof}
			We first identify the element in $S_4$. For $(\alpha,p)\in S_4$ with $v\in V$, since the Brauer graph $G$ does not have loops, the starting arrow and the ending arrow of $p$ are all $\alpha$, that means $p=\alpha C_v(\alpha)^n$, with $s(C_v(\alpha))=t(C_v(\alpha))=e$ and $1\leq n<m(v)$. On the other hand, when $val(v)=1$, we have $C_v(\alpha)=\alpha$, so $(\alpha,\pi(\alpha C_v(\alpha)^n))=(\alpha,\pi(\alpha^{n+1}))$ will also contain the elements in $S_3$. That means $$S_3\cup S_4=\{(\alpha_i,\pi(\alpha_i C_v(\alpha_i)^n))\ |\ 1\leq i\leq val(v), 1\leq n<m(v) \}.$$ Since these elements also appears in $\psi_0(e,C_v(\alpha)^k)$, all the elements of the form $(\alpha_i,\pi(\alpha_i C_v(\alpha_i)^n))$ (where $n$ is same and all $\alpha_i$ ($1\leq i\leq val(v$)) are in same special cycle) will become a unique element in $\mathrm{HH}^1(A)$ reduced by the elements in $\mathrm{Im}\psi_0$. That means the number of elements in $S_3\cup S_4$ is determined by the multiplicities of vertices. Since the number of elements in $S_1$ is given by Lemma \ref{dimL00 of A}, to sum it up, we have 
			$$
			\begin{array}{*{3}{lll}}
				\mathrm{dim}_k\mathrm{HH}^1(A)& =&|S_1|+|S_2|+|S_3|+|S_4|+|S_5|-|\mathrm{Im}\psi_0|\\
				&=& \mathrm{dim}_kL_{00}+|S_2|+\sum_{v\in V}(m(v)-1)\\
				&=&|E|-2|V|+\sum_{v\in V}m(v)+|S_2|+2.
			\end{array}
			$$
		\end{proof}
		
		Another consequence of Theorem \ref{gen-set of A} is the following corollary, which has recently also obtained by Rubio y Degrassi, Schroll and Solotar in \cite[~Theorem~4.2]{RSS} using different method (see also remarks before Example \ref{counter-example}).	
		
		\begin{corollary}\label{A-solvable}
			Let $A$ be a BGA over a field of characteristic zero such that the corresponding Brauer graph $G$ is different from	
			$(\begin{tiny}	
				\begin{tikzcd}
					\bullet & \bullet
					\arrow[shift left=1, no head, from=1-1, to=1-2]
					\arrow[shift right=1, no head, from=1-1, to=1-2]
				\end{tikzcd}
			\end{tiny})$ (in this case, both vertices have multiplicity $1$). Then $\mathrm{HH}^1(A)$ is solvable.
		\end{corollary}
		
		\begin{proof}
			Let $L:=\mathrm{HH}^1(A)$. Since there is a $k$-basis of $L$ induced from $\mathcal{B}_{L,Ker}$ by Theorem \ref{gen-set of A}. Since $S_5\subseteq \mathrm{Im}\psi_0$, we have $L\subseteq \langle S_1\cup S_2\cup S_3\cup S_4\rangle$.
			\begin{itemize}
				\item Firstly, we show the derived Lie subalgebra $L^{(1)}=[L,L]\subseteq \langle S_2\cup S_3\cup S_4\rangle$. Since $L_{00}$ is abelian and $(\alpha_i,\beta_j),(\beta_l,\alpha_k)$ will only appear in $S_2$ in the form of $(\alpha_i,\beta_j)-(\beta_l,\alpha_k)$, $\alpha_i,\alpha_k,\beta_j,\beta_l\in Q_1$, that means $grd(v)=grd(w)=2$, in other words,  $G=	
				(\begin{tiny}	
					\begin{tikzcd}
						\bullet & \bullet
						\arrow[shift left=1, no head, from=1-1, to=1-2]
						\arrow[shift right=1, no head, from=1-1, to=1-2]
					\end{tikzcd}
				\end{tiny})$ or
				$(\begin{tiny}		
					\begin{tikzcd}
						{\bullet[2]} & {\bullet[2]}
						\arrow[no head, from=1-1, to=1-2]
				\end{tikzcd}\end{tiny})$ or $(\begin{tiny}		
					\begin{tikzcd}
						\bullet \arrow[loop, distance=2em, in=35,no head, out=325]
				\end{tikzcd}\end{tiny})$. Then if  $G\neq	
				(\begin{tiny}	
					\begin{tikzcd}
						\bullet & \bullet
						\arrow[shift left=1, no head, from=1-1, to=1-2]
						\arrow[shift right=1, no head, from=1-1, to=1-2]
					\end{tikzcd}
				\end{tiny})$, all the elements in $S_1$ will not appear in $L^{(1)}$.

				\item Secondly, we show the derived Lie subalgebra $L^{(2)}=[L^{(1)},L^{(1)}]\subseteq \langle S_3\cup S_4\rangle$. For each $$x=(\beta_0,\hat{\alpha_k})-(\alpha_{k},\hat{\beta_0}),\;y=(\beta_0',\hat{\alpha_k}')-(\alpha_{k}',\hat{\beta_0}')\in S_2,$$ suppose the corresponding special cycles are $C_v,C_w,C_{v'},C_{w'}$, with $\alpha_{k},\hat{\alpha_k}\mid C_v$, $\beta_0,\hat{\beta_0}\mid C_w$, $\alpha_{k}',\hat{\alpha_k}'\mid C_{v'}$, $\beta_0',\hat{\beta_0}'\mid C_{w'}$. Without loss of generality, assume $x\neq y$. If there exists, for example, $\pi(\hat{\beta_0}^{(\beta_0',\hat{\alpha_k}')})\neq 0$, then $C_w=C_{w'}$ and, $C_v=C_w$ or $l(\hat{\beta_0})=1$.
				
				\begin{itemize}
					\item If $l(\hat{\beta_0})=1$, then $grd(w)=2$. This leads the Brauer graph $G$ to be  $(\begin{tiny}	
						\begin{tikzcd}
							\bullet & \bullet
							\arrow[shift left=1, no head, from=1-1, to=1-2]
							\arrow[shift right=1, no head, from=1-1, to=1-2]
						\end{tikzcd}
					\end{tiny})$, which contradicts the hypothesis in this corollary.
					
					\item If $C_v=C_w$, then $l(\hat{\beta_0})=l(\hat{\alpha_k})=grd(w)-1$. When $l(\hat{\beta_0})=1$, the argument follows identically to the previous case. When $l(\hat{\beta_0})>1$, since $\pi(\hat{\beta_0}^{(\beta_0',\hat{\alpha_k})})\neq 0$, we have $$l(\hat{\beta_0}^{(\beta_0',\hat{\alpha_k})})=l(\hat{\beta_0})-1+l(\hat{\alpha_k})=2grd(w)-3\leq grd(w).$$ Thus $grd(w)\leq 3$. Since $w$ is not truncated, if $grd(w)=2$, then $G=(\begin{tiny}		
						\begin{tikzcd}
							\bullet \arrow[loop, distance=2em, in=35,no head, out=325]
					\end{tikzcd}\end{tiny})$, which has been discussed in the first case. If $grd(w)= 3$, however, $\beta_0, \alpha_{k},\beta_0'$ are different arrows in special cycle $C_w$, with $s(\beta_0)=e(\alpha_k),s(\alpha_k)=e(\beta_0)$. This leads $val(w)$ bigger than  $3$. A contradiction.
				\end{itemize}
				
				Therefore, $[S_2,S_2]=0$. 
				
				\qquad For $x=(\beta_0,\hat{\alpha_k})-(\alpha_{k},\hat{\beta_0})\in S_2$, $y\in S_3$, if $\pi(\hat{\beta_0}^y)\neq 0$, then $l(\hat{\beta_0})=1$. That also means that $\beta_0,\alpha_{k},\hat{\beta_0}$ are loops. This implies that $G$ is a Brauer graph consisting of precisely one edge. If $G=(\begin{tiny}		
					\begin{tikzcd}
						\bullet \arrow[loop, distance=2em, in=35,no head, out=325]
				\end{tikzcd}\end{tiny})$, the corresponding BGA is $A=k\langle x,y\rangle/\langle xy-yx,x^2,y^2\rangle$, and the basis of its first Hochschild cohomology group is $\{(x,x)+(y,y)\}$, hence $\mathrm{HH}^1(A)$ is solvable. If $G=(\begin{tiny}		
					\begin{tikzcd}
						{\bullet[m]} & {\bullet[2]}
						\arrow[no head, from=1-1, to=1-2]
				\end{tikzcd}\end{tiny})$ with $m>1$, then the corresponding BGA is given by	$A=k\langle x,y\rangle/\langle x^m-y^2,x^{m+1},y^3,xy,yx\rangle$. Thus $S_2=\{(x,y)-(y,x^{m-1})\}$ and $S_3=\{(x,y^2)\}$. When $m=2$, we have that $[S_2,S_3]\subseteq \langle S_3\rangle$. Otherwise, $[S_2,S_3]=0$.
				
				\qquad For $x=(\beta_0,\hat{\alpha_k})-(\alpha_{k},\hat{\beta_0})\in S_2$, $y=(\alpha,p)\in S_4$. By the definition of $S_4$, we have $l(p)>1$. This implies that either $\pi(\hat{\alpha_k}^y)=0$ or $\pi(\hat{\beta_0}^y)=0$. Thus, $[x,y]$ admits no summands in $S_2$, and therefore $[S_2,S_4]\subseteq \langle S_3\cup S_4\rangle$. 
				
				\qquad By the definition of $S_3$ and $S_4$, for all $(\beta,C_v^{m(v)}(\alpha))$, $l(C_v^{m(v)}(\alpha))=grd(v)$. Indeed, for all $(\alpha,p)\in S_4$, $l(p)>1$. Thus, $[S_3,S_3]=[S_3,S_4]=0$. Since the length of every summand of the non monomial relations is equal to graded degree of some vertex, we have $[S_4,S_4]\subseteq \langle S_3\cup S_4\rangle$. 
				
				\qquad Therefore, by the discussion above, we have $L^{(2)}\subseteq \langle S_3\cup S_4\rangle$. 
				
				\item Finally, we construct some $N\in\mathbb{N}$, such that the derived Lie subalgebra $\langle S_3\cup S_4\rangle ^{(N)}=0$. Let $S$ be a subspace of $k(Q_1//\mathcal{B})$. Denote by $$l(S):=min\{l(p)|\text{$(\alpha,p)$ is a summand of some elements in $S$}\}.$$ Denote by $B:=\langle S_3\cup S_4\rangle\subseteq k(Q_1//\mathcal{B})$, then we have $l(B)>1$. If $B^{(1)}\neq 0$, we can check that $l(B^{(1)})>l(B)$. By induction, if for all $n\in\mathbb{N}$, $B^{(n)}\neq 0$, there exists $N\in \mathbb{N}$, such that $l(B^{(N)})>M:=max\{grd(v)+1|v\in V\}$. However, $Q_{\geq M}\subseteq I_G$. That leads $B^{(N)}=0$.
			\end{itemize}	
			Therefore, $L^{(N+2)}=0$, which means that $L$ is solvable. 
		\end{proof}
		
		\begin{remark}
			Although the discussion in Lemma \ref{dimL00 of A} needs the condition that $\mathrm{char}(k)=0$, the conclusions of Theorem \ref{gen-set of A} and Corollary \ref{A-solvable} only need that the field $k$ satisfies condition $(*)$ for $A$. Under this assumption, if $A$ is a BGA such that the corresponding Brauer graph $G$ is different from	
				$(\begin{tiny}	
					\begin{tikzcd}
						\bullet & \bullet
						\arrow[shift left=1, no head, from=1-1, to=1-2]
						\arrow[shift right=1, no head, from=1-1, to=1-2]
					\end{tikzcd}
				\end{tiny})$ (in this case, both vertices have multiplicity $1$), then $\mathrm{HH}^1(A)$ is solvable. In particular, when the Brauer graph $G$ has trivial multiplicity, Corollary \ref{A-solvable} is reduced to \cite[Theorem 5.4]{CSS}.
		\end{remark}

		\subsection{Some examples of non-solvable Lie algebras as the first Hochschild cohomology groups of BGAs}
		
		The case excluded in Corollary \ref{A-solvable} provides an example of a non-solvable Lie algebra $\mathrm{HH}^1(A)$ in characteristic zero. This special case is also discussed in \cite[Theorem 5.4]{CSS}, using a different method.
		
		\begin{example} \label{exception}
			The BGA which is isomorphic to the trivial extension of the Kronecker algebra is given by
			
			$$
			\begin{tikzcd}
				{G:} & {\bullet[1]} & {\bullet[1]} && {Q_G:} & 1 && 2.
				\arrow[shift left=1, no head, from=1-2, to=1-3]
				\arrow[shift right=1, no head, from=1-2, to=1-3]
				\arrow["{\beta_2}"', curve={height=6pt}, from=1-6, to=1-8]
				\arrow["{\alpha_2}", curve={height=-6pt}, from=1-6, to=1-8]
				\arrow["{\beta_1}", curve={height=-18pt}, from=1-8, to=1-6]
				\arrow["{\alpha_1}"', curve={height=18pt}, from=1-8, to=1-6]
			\end{tikzcd}$$
			The corresponding BGA of the Brauer graph $G$ is $A=kQ_G/I_G$, where $I_G$ is generated by
			\begin{itemize}
				\item $R_1=\{\alpha_1\alpha_2-\beta_1\beta_2,\; \alpha_2\alpha_1-\beta_2\beta_1\}$,
				
				\item $R_2=\{\alpha_1\alpha_2\alpha_1,\;\alpha_2\alpha_1\alpha_2,\;\beta_1\beta_2\beta_1,\;\beta_2\beta_1\beta_2\}$,
				
				\item $R_3=\{\alpha_1\beta_2,\;\beta_2\alpha_1,\;\alpha_2\beta_1,\;\beta_1\alpha_2\}$.
			\end{itemize}
			The length-lexicographic order is induced by $\alpha_1>\alpha_2>\beta_1>\beta_2>e_1>e_2$, then
			
			$$
			\begin{array}{*{9}{lllllllll}}
				\psi_0& :&(e_1,e_1)&\mapsto&(\alpha_2,\alpha_2)-(\alpha_1,\alpha_1)+(\beta_2,\beta_2)-(\beta_1,\beta_1)&, &(e_1,\alpha_1\alpha_2)&\mapsto&0\\
				&&(e_2,e_2)&\mapsto&(\alpha_1,\alpha_1)-(\alpha_2,\alpha_2)+(\beta_1,\beta_1)-(\beta_2,\beta_2)&, &(e_2,\alpha_2\alpha_1)&\mapsto&0\\
				\psi_1& :&(\alpha_1,\alpha_1),(\alpha_2,\alpha_2)&\mapsto&(\alpha_1\alpha_2,\beta_1\beta_2)+(\alpha_2\alpha_1,\beta_2\beta_1)&&&&\\ &&(\beta_1,\beta_1),(\beta_2,\beta_2)&\mapsto&-(\alpha_1\alpha_2,\beta_1\beta_2)-(\alpha_2\alpha_1,\beta_2\beta_1)&&&&\\
				&&(\alpha_1,\beta_1),(\beta_2,\alpha_2)&\mapsto&(\alpha_1\beta_2,\beta_1\beta_2)+(\beta_2\alpha_1,\beta_2\beta_1)&&&&\\&&(\beta_1,\alpha_1),(\alpha_2,\beta_2)&\mapsto&(\beta_1\alpha_2,\beta_1\beta_2)+(\alpha_2\beta_1,\beta_2\beta_1)
			\end{array}
			$$
			Then we get $$\mathrm{HH}^1(A)=k\{(\alpha_1,\alpha_1)-(\alpha_2,\alpha_2),\; (\alpha_1,\alpha_1)+(\beta_1,\beta_1),\; (\alpha_1,\beta_1)-(\beta_2,\alpha_2),\; (\alpha_2,\beta_2)-(\beta_1,\alpha_1)\}.$$ 
			The Lie bracket of any two elements is always zero except in the following cases:
			$$
			\begin{array}{*{3}{lllll}}
				\lbrack(\alpha_1,\beta_1)-(\beta_2,\alpha_2),\;(\alpha_1,\alpha_1)-(\alpha_2,\alpha_2)\rbrack&= &(\alpha_1,\beta_1)-(\beta_2,\alpha_2)\\
				\lbrack(\alpha_1,\alpha_1)-(\alpha_2,\alpha_2),\;(\alpha_2,\beta_2)-(\beta_1,\alpha_1)\rbrack&= &(\alpha_2,\beta_2)-(\beta_1,\alpha_1) \\
				\lbrack (\alpha_1,\beta_1)-(\beta_2,\alpha_2),\; (\alpha_2,\beta_2)-(\beta_1,\alpha_1)\rbrack&=&-(\beta_1,\beta_1)-(\alpha_2,\alpha_2)+(\beta_2,\beta_2)+(\alpha_1,\alpha_1)\\&=&2((\alpha_1,\alpha_1)-(\alpha_2,\alpha_2))
			\end{array}
			$$
			Therefore, $\mathrm{HH}^1(A)\cong k\oplus sl_2(k)$ which is not solvable if and only if $\mathrm{char}(k)\neq 2$. 
			
			Moreover, since there is a multiple edge in $G$ and there are two elements $(\alpha_1,\beta_1)-(\beta_2,\alpha_2)$ and $(\alpha_2,\beta_2)-(\beta_1,\alpha_1)$ in $S_2$,  we have $|S_2|=2$. By Proposition \ref{dim-noloop}, $\mathrm{dim}_k\mathrm{HH}^1(A)=2-2\cdot 2+(1+1)+2+2=4$.
		\end{example}

		In general, when the characteristic of the field is positive, there are many other examples that are not solvable. Note that the following example is a counter-example of \cite[~Theorem~4.2]{RSS} in positive characteristic. In fact, \cite[~Theorem~4.2]{RSS} misses the following not solvable case:  $G=(\begin{tiny}		
			\begin{tikzcd}
				{\bullet[m]} & {\bullet[1]}
				\arrow[no head, from=1-1, to=1-2]
		\end{tikzcd}\end{tiny})$ with $m>1$ and $\mathrm{char}(k)\;|\;(m+1)$. This mistake is caused by missing to consider the relations of the form $r=\alpha^{m+1}$ for $m\geq 2$.
		
		\begin{example} \label{counter-example}
			Let $\mathrm{char}(k)=3$.
			$$
			\begin{tikzcd}
				G: & \bullet[2] \arrow[r, no head] & \bullet[1] & &Q_G: & \bullet \arrow["x"', loop, distance=2em, in=35, out=325]
			\end{tikzcd}$$
			The corresponding BGA is $A=k\langle x\rangle/\langle x^3\rangle$. 
			$$
			\begin{array}{*{13}{lllllllllllll}}
				\psi_0& :&(1,1)&\mapsto&0&, &(1,x)&\mapsto&0&, &(1,x^2)&\mapsto&0;\\
				\psi_1& :&(x,1)&\mapsto&3(x^3,x^2)=0& ,&(x,x)&\mapsto&0&, &(x,x^2)&\mapsto& 0.
			\end{array}
			$$
			Then we get $\mathrm{HH}^1(A)=k\{(x,1),(x,x),(x,x^2)\}$. The Lie bracket on $\mathrm{HH}^1(A)$ is
			$$
			\begin{array}{*{3}{lllll}}
				\lbrack(x,1),(x,x^2)\rbrack&= &(x,(x^2)^{(x,1)})=2(x,x)\\
				\lbrack (x,1),(x,x)\rbrack&= &(x,1) \\
				\lbrack(x,x^2),(x,x)\rbrack&=&(x,x^2)-(x,(x^2)^{(x,x)})=-(x,x^2)
			\end{array}
			$$
			Therefore, $\mathrm{HH}^1(A)\cong sl_2(k)$ is a simple Lie algebra.
		\end{example}

		\section{The associated graded algebras of BGAs}
		
		For a finite dimensional algebra $A$, one can construct a graded algebra $gr(A)$ associated with the radical filtration of $A$. In \cite{GL}, Guo and Liu compared the representation types between a BGA and its associated graded algebra. In this section, we compare the first Hochschild
		cohomology groups between them.
		
		\subsection{The Lie algebra structure on $\mathrm{HH}^1(gr(A))$}
		
		\begin{definition}
			Let $A$ be a finite dimensional algebra. Denote by $\mathfrak{r}$ the (Jacobson) radical of $A$. Then the graded algebra $gr(A)$ of $A$ associated with the radical filtration is defined as follows. As a graded vector space,
			$$gr(A)=A/\mathfrak{r}\oplus\mathfrak{r}/\mathfrak{r}^2\oplus\cdots\oplus\mathfrak{r}^t/\mathfrak{r}^{t+1}\oplus\cdots.$$
			The multiplication of $gr(A)$ is given as follows. For any two homogeneous elements:
			$$x+\mathfrak{r}^{m+1}\in\mathfrak{r}^m/\mathfrak{r}^{m+1},\quad y+\mathfrak{r}^{n+1}\in\mathfrak{r}^n/\mathfrak{r}^{n+1},$$
			we have $$(x+\mathfrak{r}^{m+1})\cdot(y+\mathfrak{r}^{n+1})=xy+\mathfrak{r}^{m+n+1}.$$
		\end{definition}
		
		We now specialize $A$ to be a BGA associated with a Brauer graph $G=(V,E)$. Recall that the associated $gr(A)$ can be described as follows.
		
		\begin{lemma}$(\cite[~Lemma~2.9]{GL})$\label{relations of grA}
			Let $A=kQ/I$ be a BGA. The generating relations of the second and the third types in $I$ are given by paths, relation of the first type is of the form $\rho=p-q$, where $p$ and $q$ are two paths with $s(p)=t(p)=s(q)=t(q)$. For any relation $\rho=p-q$ of first type (suppose that the length of $p$ is $m$ and the length of $q$ is $n$), we replace it by
			$$\rho'= \left\{
			\begin{array}{*{3}{lll}}
				\rho& , & m=n,\\
				q& , & m>n,\\
				p&,&m<n.
			\end{array}
			\right.$$
			Then, the associated graded algebra $gr(A)$ is isomorphic to $kQ/I'$, where $Q$ is the same quiver as above and $I'$ is an admissible ideal whose generating relations are obtained from that of $I$ by replacing each $\rho$ by $\rho'$.
		\end{lemma}
		
		\begin{remark}
			It can be checked that the generating set of $I'$ described in Lemma \ref{relations of grA} is a Gr{\"o}bner basis of $I'$ in $kQ$, which also may not be tip-reduced.
		\end{remark}
		
		The following notion of unbalanced edges in a Brauer graph is useful in our later discussions.
		
		\begin{definition}$(\cite[~Definition~2.12]{GL})$
			Let $G=(V,E)$ be a Brauer graph with graded degree function $grd$ and $A=kQ/I$ the corresponding BGA. We identify $Q_0$ with $E$ by the natural bijection between them.
			\begin{itemize}
				\item We call an edge $v_1\stackrel{i}{-}v_2$ in $G$ with $grd(v_1)\neq grd(v_2)$ an unbalanced edge, and denote the end points of $i$ by $v_L^{(i)},v_S^{(i)}$,  with $grd(v_L^{(i)})>grd(v_S^{(i)})$. Whenever the context is clear, we will omit the superscript $(i)$. Other edges which are not unbalanced will be called the balanced edges.				
				\item For any unbalanced edge $v_S\stackrel{i}{-}v_L$  in $G$, there is a relation of the first type $\rho_i=p_i-q_i$ in $I$, where $p_i=C_{v_S}^{m(v_S)},q_i=C_{v_L}^{m(v_L)}$ are two paths with lengths $grd(v_S),grd(v_L)$, respectively.
                \item We denote by $\mathbb{W}$ the set of unbalanced edges in $G$.      
			\end{itemize}
		\end{definition}
			
		In order to deal with the Lie algebra structure on the first Hochschild cohomology group of $gr(A)$, we introduce the notion of the balanced components of a Brauer graph.
		
		\begin{definition}
			Let $G$ be a Brauer graph. We denote by $G'$ the graph obtained from $G$ by the following rule: split each unbalanced edge in $G$ into two edges by attaching two new truncated vertices. We define the balanced components of $G$ to be the connected components of $G'$. Denote the set of the balanced components of $G$ by $\Gamma_G$.
		\end{definition}
		An example of a Brauer graph which has two balanced components is given by Figure \ref{Ba.Comp} (where the edge $e$ splits into two edges $e'$ and $e''$). Note that in general, there is no direct relation between the number of elements in $\Gamma_G$ and the number of unbalanced edges in $G$. For example, a Brauer graph which consists of two vertices and $n$ parallel unbalanced edges always has two balanced components.
		\begin{figure}[H]
		$$
		\begin{tikzcd}
			&                                               & {[2]} \arrow[r, no head] & {[1]} \arrow[no head, loop, distance=2em, in=125, out=55] \arrow[r, no head] & {[1]} &                                                                               & {[1]} \arrow[r,"e'", no head]                     & {[2]} \arrow[r, no head] & {[1]} \arrow[r, no head] \arrow[no head, loop, distance=2em, in=125, out=55] & {[1]} \\
			{[1]} \arrow[r, no head] \arrow[no head, loop, distance=2em, in=215, out=145] & {[1]} \arrow[rd, no head] \arrow[ru,"e", no head] &                          & {} \arrow[r, Rightarrow]                                                     & {}    & {[1]} \arrow[r, no head] \arrow[no head, loop, distance=2em, in=215, out=145] & {[1]} \arrow[r,"e''", no head] \arrow[rd, no head] & {[1]}                    &                                                                              &       \\
			&                                               & {[3]}                    &                                                                              &       &                                                                               &                                              & {[3]}                    &                                                                              &
		\end{tikzcd}$$
		\caption{A Brauer graph with two balanced components}
		\label{Ba.Comp}
	\end{figure}
		
		Note that when considered a balanced component of $G$ as a Brauer graph, the corresponding BGA and its associated graded algebra are the same (c.f. \cite[Propostion~2.13]{GL}). 

We are ready to study the Lie structure on $\mathrm{HH}^1(gr(A))$ under the assumption that the characteristic of the field $k$ is $0$. Recall the definition of the maps $\psi_0$, $\psi_1$ from Proposition \ref{gen-parallel paths} and denote these maps for $gr(A)$ by $\psi_0^{gr(A)}$, $\psi_1^{gr(A)}$ respectively.
		Let $$L_{0}^{gr(A)}:=k(Q_1//Q_1)\cap \mathrm{Ker}\psi_1^{gr(A)}\bigg/ \langle\sum_{a\in Q_1e}(a,a)-\sum_{b\in eQ_1}(b,b)|e\in Q_0\rangle,$$ which is a Lie subalgebra of $\mathrm{HH}^1(gr(A))$.
		Furthermore, we can consider $L_{00}^{gr(A)}$ which is given by
		$$L_{00}^{gr(A)}:=\langle(\alpha,\alpha)|\alpha\in Q_1\rangle\cap \mathrm{Ker}\psi_1^{gr(A)}\bigg/ \langle\sum_{a\in Q_1e}(a,a)-\sum_{b\in eQ_1}(b,b)|e\in Q_0\rangle.$$
		Note that $L_{00}^{gr(A)}$ is an abelian Lie subalgebra of $\mathrm{HH}^1(gr(A))$ and $L_{00}^{gr(A)}\subseteq L_0^{gr(A)}$.
		
		\begin{lemma}\label{dimL00 of grA}
			Let $A$ be a BGA associated with a Brauer graph $G=(V,E)$ and $gr(A)$ the associated graded algebra of $A$. Then
			$$\mathrm{dim}_k L_{00}^{gr(A)}=|E|-|V|+1+|\Gamma_G|.$$
		\end{lemma}
		
		\begin{proof}
			In $L_{00}^{gr(A)}$, we only need to consider the linear combination of $(\alpha,\alpha)$. For each $\alpha\in Q_1$, if $\alpha$ is not involved in some homogeneous relations of the first type, then $\psi_1^{gr(A)}((\alpha,\alpha))=0$. That means the monomial relations in the Gr{\"o}bner basis $\mathcal{G}$ of $I$ will not influence the basis of $L_{00}^{gr(A)}$. Therefore, it is enough to consider the balanced components of $G$ as Brauer graphs. By Lemma \ref{dimL00 of A}, for $C\in\Gamma_G$,  $\mathrm{dim}_k(\langle (\alpha,\alpha)|\alpha\in Q_1\rangle\cap \mathrm{Ker}\psi_1^{C})=\sum_{v\in V_C}(val(v)-1)+1=2|E_C|-|V_C|+1$. In addition, $\mathrm{dim}_k(\langle \psi_0^{gr(A)}(e,e)|e\in Q_0\rangle )=|E|-1$ in the same way. Then
			$$
			\begin{array}{*{3}{lll}}
				\mathrm{dim}_kL_{00}^{gr(A)} & =&\sum_{C\in\Gamma_G}(2|E_C|-|V_C|+1)-|E|+1\\
				& =&2(|E|+|\mathbb{W}|)-(|V|+2|\mathbb{W}|)+|\Gamma_G|-|E|+1\\
				&=& |E|-|V|+1+|\Gamma_G|.
			\end{array}
			$$
		\end{proof}
		
		We now check the solvability of $\mathrm{HH}^1(gr(A))$ which can be verified by using the same method as in Section 4. Besides, since $I'$ is an homogeneous ideal of $kQ_G$, it is easier to verify the solvability of $\mathrm{HH}^1(gr(A))$ by using the graded structure of $\mathrm{HH}^1(gr(A))$ as we defined in Section 3.3.
		
		\begin{lemma}\label{grL0=L00}
			Let $G$ be a Brauer graph that does not contain a subgraph $	
			(v{\xlongequal{\;\;}}w)$ with $grd(v)=2$, or
			$G\neq (v{-}w)$ with $m(v)=2$, $m(w)\geq 2$. Then $L_0^{gr(A)}=L_{00}^{gr(A)}$.
		\end{lemma}
		
		\begin{proof}
			If $(k(Q_1//Q_1)\cap \mathrm{Ker}\psi_1^{gr(A)})\backslash(\langle (\alpha,\alpha)|\alpha\in Q_1\rangle\cap \mathrm{Ker}\psi_1^{gr(A)})\neq \varnothing$, then there exists some pair of arrows $(\alpha,\beta)\in k(Q_1//Q_1)$, such that $\alpha\neq\beta$. That means in this Brauer graph, there exists two edges $e_1,e_2$, such that they connect to the same vertices $v_1,v_2$.
			
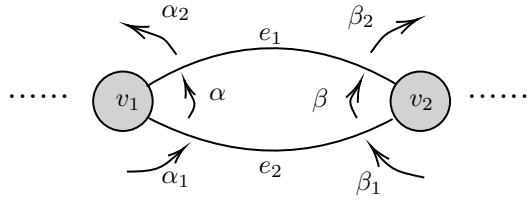
\begin{figure}[H]
			\begin{center}

				\tikzset{every picture/.style={line width=0.75pt}} 
				
				\begin{tikzpicture}[x=0.75pt,y=0.75pt,yscale=-1,xscale=1]
					
					\draw  [fill={rgb, 255:red, 210; green, 210; blue, 210 }  ,fill opacity=1 ] (80,95) .. controls (80,86.72) and (86.72,80) .. (95,80) .. controls (103.28,80) and (110,86.72) .. (110,95) .. controls (110,103.28) and (103.28,110) .. (95,110) .. controls (86.72,110) and (80,103.28) .. (80,95) -- cycle ;
					\draw  [fill={rgb, 255:red, 210; green, 210; blue, 210 }  ,fill opacity=1 ] (230,95) .. controls (230,86.72) and (236.72,80) .. (245,80) .. controls (253.28,80) and (260,86.72) .. (260,95) .. controls (260,103.28) and (253.28,110) .. (245,110) .. controls (236.72,110) and (230,103.28) .. (230,95) -- cycle ;
					\draw    (107,87.5) .. controls (146,63) and (190,61) .. (233,86.5) ;
					\draw    (108,103.5) .. controls (148,124) and (189,126.5) .. (231,104) ;
					\draw    (128,103.5) .. controls (133.64,95.04) and (130.44,95.42) .. (126.72,86.36) ;
					\draw [shift={(126,84.5)}, rotate = 70.02] [color={rgb, 255:red, 0; green, 0; blue, 0 }  ][line width=0.75]    (8.74,-2.63) .. controls (5.56,-1.12) and (2.65,-0.24) .. (0,0) .. controls (2.65,0.24) and (5.56,1.12) .. (8.74,2.63)   ;
					\draw    (97,130.5) .. controls (113.24,130.5) and (114.88,128.68) .. (124.57,119.8) ;
					\draw [shift={(126,118.5)}, rotate = 137.73] [color={rgb, 255:red, 0; green, 0; blue, 0 }  ][line width=0.75]    (8.74,-2.63) .. controls (5.56,-1.12) and (2.65,-0.24) .. (0,0) .. controls (2.65,0.24) and (5.56,1.12) .. (8.74,2.63)   ;
					\draw    (122,71.5) .. controls (114.24,60.83) and (117.77,64.27) .. (98.82,57.19) ;
					\draw [shift={(97,56.5)}, rotate = 20.85] [color={rgb, 255:red, 0; green, 0; blue, 0 }  ][line width=0.75]    (10.93,-3.29) .. controls (6.95,-1.4) and (3.31,-0.3) .. (0,0) .. controls (3.31,0.3) and (6.95,1.4) .. (10.93,3.29)   ;
					\draw    (213,103.5) .. controls (207.33,93.1) and (209.7,95.21) .. (214.19,86.19) ;
					\draw [shift={(215,84.5)}, rotate = 114.44] [color={rgb, 255:red, 0; green, 0; blue, 0 }  ][line width=0.75]    (10.93,-3.29) .. controls (6.95,-1.4) and (3.31,-0.3) .. (0,0) .. controls (3.31,0.3) and (6.95,1.4) .. (10.93,3.29)   ;
					\draw    (247,133.5) .. controls (233.56,130.62) and (230.26,130.5) .. (221.17,119) ;
					\draw [shift={(220,117.5)}, rotate = 52.43] [color={rgb, 255:red, 0; green, 0; blue, 0 }  ][line width=0.75]    (10.93,-3.29) .. controls (6.95,-1.4) and (3.31,-0.3) .. (0,0) .. controls (3.31,0.3) and (6.95,1.4) .. (10.93,3.29)   ;
					\draw    (220,70.5) .. controls (227.72,59.88) and (226.13,61.38) .. (242.18,54.3) ;
					\draw [shift={(244,53.5)}, rotate = 156.04] [color={rgb, 255:red, 0; green, 0; blue, 0 }  ][line width=0.75]    (10.93,-3.29) .. controls (6.95,-1.4) and (3.31,-0.3) .. (0,0) .. controls (3.31,0.3) and (6.95,1.4) .. (10.93,3.29)   ;
					
					\draw (90,90) node [anchor=north west][inner sep=0.75pt]   [align=left] {$v_1$};
					\draw (238,90) node [anchor=north west][inner sep=0.75pt]   [align=left] {$v_2$};
					\draw (35,90) node [anchor=north west][inner sep=0.75pt]  [font=\Large] [align=left] {......};
					\draw (267,90) node [anchor=north west][inner sep=0.75pt]  [font=\Large] [align=left] {......};
					\draw (113,129) node [anchor=north west][inner sep=0.75pt]   [align=left] {$\alpha_{1}$};
					\draw (137,87) node [anchor=north west][inner sep=0.75pt]   [align=left] {$\alpha$};
					\draw (113,45) node [anchor=north west][inner sep=0.75pt]   [align=left] {$\alpha_{2}$};
					\draw (211,129) node [anchor=north west][inner sep=0.75pt]   [align=left] {$\beta_{1}$};
					\draw (189,87) node [anchor=north west][inner sep=0.75pt]   [align=left] {$\beta$};
					\draw (207,45) node [anchor=north west][inner sep=0.75pt]   [align=left] {$\beta_{2}$};
					\draw (162,57) node [anchor=north west][inner sep=0.75pt]   [align=left] {$e_1$};
					\draw (162,124) node [anchor=north west][inner sep=0.75pt]   [align=left] {$e_2$};

				\end{tikzpicture}
			\end{center}
			\caption{Parallel arrows in a Brauer graph}
		\end{figure}
			
			If $e_1$,$e_2$ are different edges, then there is a multiple edge in Brauer graph.		
			Since the relations of type \uppercase\expandafter{\romannumeral3} are contained in the reduced Gr{\"o}bner basis $\mathcal{G}$, then $\beta\alpha_{1},\alpha\beta_{1},\beta_{2}\alpha,\alpha_{2}\beta\in\mathcal{G}$. Without loss of generality, let $grd(v_1)\neq 2$. If $grd(v_2)\neq 2$, then
			$$(\alpha\beta_{1},\pi(\alpha\beta_{1}^{(\alpha,\beta)}))=(\alpha\beta_{1},\pi(\beta\beta_{1}))=(\alpha\beta_{1},\beta\beta_{1})$$
			$$(\beta_{2}\alpha,\pi(\beta_{2}\alpha^{(\alpha,\beta)}))=(\beta_{2}\alpha,\pi(\beta_2\beta))=(\beta_{2}\alpha,\beta_2\beta)$$
			$$(\alpha_2\beta,\pi(\alpha_2\beta^{(\beta,\alpha)}))=(\alpha_2\beta,\pi(\alpha_2\alpha))=(\alpha_2\beta,\alpha_2\alpha)$$
			$$(\beta\alpha_1,\pi(\beta\alpha_1^{(\beta,\alpha)}))=(\beta\alpha_1,\pi(\alpha\alpha_{1}))=(\beta\alpha_1,\alpha\alpha_{1})$$
			That means $(\alpha,\beta),(\beta,\alpha)$ will not appear in  $k(Q_1//Q_1)\cap \mathrm{Ker}\psi_1^{gr(A)}$ since the parallel path pairs above will not appear in any $\psi_1^{gr(A)}((\alpha',\beta'))$ with $(\alpha',\beta')\in k(Q_1//Q_1)$.
			However, when $grd(v_2)= 2$, since $\pi(\beta\beta_{1})=\pi(\beta_2\beta)=0$, we have
			$$(\alpha\beta_{1},\pi(\alpha\beta_{1}^{(\alpha,\beta)}))=(\alpha\beta_{1},\pi(\beta\beta_{1}))=0,$$
			$$(\beta_{2}\alpha,\pi(\beta_{2}\alpha^{(\alpha,\beta)}))=(\beta_{2}\alpha,\pi(\beta_2\beta))=0.$$
			That means $(\alpha,\beta)\in k(Q_1//Q_1)\cap \mathrm{Ker}\psi_1^{gr(A)}$.
			
			If $e_1$ and $e_2$ are the same edges, then $G=
			(\begin{tiny}	
				\begin{tikzcd}
					\bullet & \bullet
					\arrow[no head, from=1-1, to=1-2]
				\end{tikzcd}
			\end{tiny})$ or
			$(\begin{tiny}		
				\begin{tikzcd}
					\bullet \arrow[loop, distance=2em, in=35,no head, out=325]
			\end{tikzcd}\end{tiny})$. By the hypothesis of the proposition, if the vertices $v_1,v_2$ in the first case have the property that $m(v_1)>2$ and $m(v_2)>2$, or if the unique vertex $v'$ in the second case has $m(v')>1$, it can be checked that $(\alpha,\beta)$ can not appear in $k(Q_1//Q_1)\cap \mathrm{Ker}\psi^{gr(A)}$ by the same reason of the previous cases. If the vertex $v'$ in the second case has $m(v')=1$, then the corresponding BGA is given by $A=gr(A)=k\langle x,y\rangle/\langle xy-yx,x^2,y^2\rangle$, whose first Hochschild cohomology group is given by the following vector space $$\{(x,x),\;(y,y),\;(x,yx),\;(y,yx)\},$$ and therefore $L_0^{gr(A)}=L_{00}^{gr(A)}=\{(x,x),(y,y)\}$.
		\end{proof}
		
		\begin{remark}\label{gr-L0}
			Since the associated graded algebra of a BGA is defined by homogeneous ideal, it is natural to consider the graded structure on $\mathrm{HH}^1(gr(A))$. Similar to the discussion in \cite[Page 258]{Strametz}, we have  $$\mathrm{HH}^1(gr(A))\bigg/rad(\mathrm{HH}^1(gr(A)))\cong L_0^{gr(A)}\bigg/ rad(L_0)^{gr(A)}.$$
		\end{remark}
		
		\begin{theorem}\label{grA-solvable}
			Let $A$ be a BGA over a field of characteristic zero such that the corresponding Brauer graph $G$ is different from	
			$(\begin{tiny}	
				\begin{tikzcd}
					\bullet & \bullet
					\arrow[shift left=1, no head, from=1-1, to=1-2]
					\arrow[shift right=1, no head, from=1-1, to=1-2]
				\end{tikzcd}
			\end{tiny})$ (in this case, both vertices have multiplicity $1$) and let $gr(A)$ be the associated graded algebra of $A$. Then $\mathrm{HH}^1(gr(A))$ is solvable.
		\end{theorem}
		
		\begin{proof}
			By Remark \ref{gr-L0} and Lemma \ref{grL0=L00}, we can only consider the solvability on $L_0^{gr(A)}$.
			\begin{itemize}
				\item If $G$ does not contain a subgraph $	
				(\begin{tiny}	
					\begin{tikzcd}
						\bullet & \bullet
						\arrow[shift left=1, no head, from=1-1, to=1-2]
						\arrow[shift right=1, no head, from=1-1, to=1-2]
					\end{tikzcd}
				\end{tiny})$ where one vertex, say $v$, satisfies  $grd(v)=2$, or
				$G\neq (\begin{tiny}		
					\begin{tikzcd}
						{\bullet[2]} & {\bullet[m]}
						\arrow[no head, from=1-1, to=1-2]
				\end{tikzcd}\end{tiny})$ with $m\geq 2$, then $L_0^{gr(A)}=L_{00}^{gr(A)}$, and $(L_0^{gr(A)})^{(1)}=(L_{00}^{gr(A)})^{(1)}=0$.
				
				\item Let $G= (\begin{tiny}		
					\begin{tikzcd}
						{v_1[2]} & {v_2[m]}
						\arrow[no head, from=1-1, to=1-2]
				\end{tikzcd}\end{tiny})$ with $m\geq 2$. If $m=2$, the analysis reduces to the case of its corresponding BGA in Corollary \ref{A-solvable}. So let $m>2$ and denote the arrow around $v_1$ by $\alpha$, the arrow around $v_2$ by $\beta$, then we have  $$L_0^{gr(A)}=L_{00}^{gr(A)}\cup\{(\beta,\alpha)\}.$$ Thus $(L_0^{gr(A)})^{(1)}=\{(\beta,\alpha)\}$, $(L_0^{gr(A)})^{(2)}=0$.
				
				\item Let $G$ contain a subgraph $	
				(\begin{tiny}	
					\begin{tikzcd}
						v_i & v_i'
						\arrow[shift left=1, no head, from=1-1, to=1-2]
						\arrow[shift right=1, no head, from=1-1, to=1-2]
					\end{tikzcd}
				\end{tiny})$ with $grd(v_i)=2$, $i=1,\cdots,m$. Assume $grd(v_i')>2$ and denote the parallel arrows in the multiple edges around $v_i,v_i'$ by $\alpha_i,\beta_i$ respectively. Then in this case, $L_0^{gr(A)}=L_{00}^{gr(A)}\cup\{(\beta_i,\alpha_i)|i=1,\cdots,m\}$. Thus  $(L_0^{gr(A)})^{(1)}=\{(\beta_i,\alpha_i)|i=1,\cdots,m\}$, $(L_0^{gr(A)})^{(2)}=0$.
			\end{itemize}
			
			To sum up, $\mathrm{HH}^1(gr(A))$ is solvable.
		\end{proof}
		
		\begin{remark}
			Suppose that the field $k$ satisfies condition $(*)$ for $A$. The conclusions in Lemma \ref{grL0=L00} and Theorem \ref{grA-solvable} are also true. In fact, we only need to assume $\mathrm{char}(k)=0$ when we need a specific description about the $k$-basis of $L_{00}$ (respectively, $L_{00}^{gr}$).
		\end{remark}

		\subsection{An injection from $\mathrm{HH}^1(A)$ to $\mathrm{HH}^1(gr(A))$}
		
		Let $A$ be a BGA. In this subsection, we will construct a Lie algebra monomorphism from $\mathrm{HH}^1(A)$ to $\mathrm{HH}^1(gr(A))$, which will give us a specific comparison between them. We begin by constructing a $k$-basis for the first Hochschild cohomology group $\mathrm{HH}^1(gr(A))$.
		
		\begin{lemma}\label{gen-set of grA}
			There is a $k$-basis $\mathcal{B}^{gr(A)}_{L,Ker}$ of $\mathrm{Ker}\psi_1^{gr(A)}$ which consists of the following five subsets of $\mathrm{Ker}\psi_1^{gr(A)}$:
			\begin{itemize}
				\item $S_1^{gr}$: the basis of $\langle (\alpha,\alpha)|\alpha\in Q_1\rangle\cap \mathrm{Ker}\psi_1^{gr(A)}$ described in Lemma \ref{dimL00 of grA};
				
				\item $S_2^{gr}$: consider the following subgraph of the Brauer graph:
				\begin{figure}[H]
				\begin{center}

					\tikzset{every picture/.style={line width=0.75pt}} 
					
					\begin{tikzpicture}[x=0.75pt,y=0.75pt,yscale=-1,xscale=1]
						
						\draw  [fill={rgb, 255:red, 210; green, 210; blue, 210 }  ,fill opacity=1 ] (80,95) .. controls (80,86.72) and (86.72,80) .. (95,80) .. controls (103.28,80) and (110,86.72) .. (110,95) .. controls (110,103.28) and (103.28,110) .. (95,110) .. controls (86.72,110) and (80,103.28) .. (80,95) -- cycle ;
						\draw  [fill={rgb, 255:red, 210; green, 210; blue, 210 }  ,fill opacity=1 ] (230,95) .. controls (230,86.72) and (236.72,80) .. (245,80) .. controls (253.28,80) and (260,86.72) .. (260,95) .. controls (260,103.28) and (253.28,110) .. (245,110) .. controls (236.72,110) and (230,103.28) .. (230,95) -- cycle ;
						\draw    (107,87.5) .. controls (146,63) and (190,61) .. (233,86.5) ;
						\draw    (108,103.5) .. controls (148,124) and (189,126.5) .. (231,104) ;
						\draw    (128,103.5) .. controls (133.64,95.04) and (130.44,95.42) .. (126.72,86.36) ;
						\draw [shift={(126,84.5)}, rotate = 70.02] [color={rgb, 255:red, 0; green, 0; blue, 0 }  ][line width=0.75]    (8.74,-2.63) .. controls (5.56,-1.12) and (2.65,-0.24) .. (0,0) .. controls (2.65,0.24) and (5.56,1.12) .. (8.74,2.63)   ;
						\draw    (97,130.5) .. controls (113.24,130.5) and (114.88,128.68) .. (124.57,119.8) ;
						\draw [shift={(126,118.5)}, rotate = 137.73] [color={rgb, 255:red, 0; green, 0; blue, 0 }  ][line width=0.75]    (8.74,-2.63) .. controls (5.56,-1.12) and (2.65,-0.24) .. (0,0) .. controls (2.65,0.24) and (5.56,1.12) .. (8.74,2.63)   ;
						\draw    (122,71.5) .. controls (114.24,60.83) and (117.77,64.27) .. (98.82,57.19) ;
						\draw [shift={(97,56.5)}, rotate = 20.85] [color={rgb, 255:red, 0; green, 0; blue, 0 }  ][line width=0.75]    (10.93,-3.29) .. controls (6.95,-1.4) and (3.31,-0.3) .. (0,0) .. controls (3.31,0.3) and (6.95,1.4) .. (10.93,3.29)   ;
						\draw    (209,82.5) .. controls (205.2,94.85) and (205.91,92.76) .. (211.99,102.81) ;
						\draw [shift={(213,104.5)}, rotate = 239.74] [color={rgb, 255:red, 0; green, 0; blue, 0 }  ][line width=0.75]    (10.93,-3.29) .. controls (6.95,-1.4) and (3.31,-0.3) .. (0,0) .. controls (3.31,0.3) and (6.95,1.4) .. (10.93,3.29)   ;
						\draw    (218,115.5) .. controls (223.79,127.08) and (229.58,128.42) .. (244.36,135.68) ;
						\draw [shift={(246,136.5)}, rotate = 206.57] [color={rgb, 255:red, 0; green, 0; blue, 0 }  ][line width=0.75]    (10.93,-3.29) .. controls (6.95,-1.4) and (3.31,-0.3) .. (0,0) .. controls (3.31,0.3) and (6.95,1.4) .. (10.93,3.29)   ;
						\draw    (245,50.5) .. controls (227.54,52.44) and (227.96,53.44) .. (215.22,66.27) ;
						\draw [shift={(214,67.5)}, rotate = 315] [color={rgb, 255:red, 0; green, 0; blue, 0 }  ][line width=0.75]    (10.93,-3.29) .. controls (6.95,-1.4) and (3.31,-0.3) .. (0,0) .. controls (3.31,0.3) and (6.95,1.4) .. (10.93,3.29)   ;
						
						\draw (90,90) node [anchor=north west][inner sep=0.75pt]   [align=left] {v};
						\draw (238,90) node [anchor=north west][inner sep=0.75pt]   [align=left] {w};
						\draw (35,90) node [anchor=north west][inner sep=0.75pt]  [font=\Large] [align=left] {......};
						\draw (267,90) node [anchor=north west][inner sep=0.75pt]  [font=\Large] [align=left] {......};
						\draw (113,129) node [anchor=north west][inner sep=0.75pt]   [align=left] {$\alpha_{k-1}$};
						\draw (137,87) node [anchor=north west][inner sep=0.75pt]   [align=left] {$\alpha_k$};
						\draw (113,47) node [anchor=north west][inner sep=0.75pt]   [align=left] {$\alpha_{0}$};
						\draw (207,129) node [anchor=north west][inner sep=0.75pt]   [align=left] {$\beta_{1}$};
						\draw (189,87) node [anchor=north west][inner sep=0.75pt]   [align=left] {$\beta_0$};
						\draw (203,45) node [anchor=north west][inner sep=0.75pt]   [align=left] {$\beta_m$};
						\draw (160,57) node [anchor=north west][inner sep=0.75pt]   [align=left] {$e_1$};
						\draw (160,123) node [anchor=north west][inner sep=0.75pt]   [align=left] {$e_2$};

					\end{tikzpicture}
					
				\end{center}
				\caption{The subgraph corresponding to some element in $S_2^{gr}$}
					\end{figure}
				with $\alpha_{k}\neq\beta_0$, $\alpha_{k}\beta_0, \beta_0\alpha_{k}\in R_3$. 
				
				Let $\hat{\alpha_k}:=\alpha_{k-1}\cdots\alpha_{0}\cdot C_v(\alpha_{0})^{m(v)-1}$, $\hat{\beta_0}:=\beta_m\cdots\beta_{1}\cdot C_w(\beta_{1})^{m(w)-1}$ 
				\begin{itemize}
					\item	If $grd(v)=grd(w)$, then the corresponding element in $S_2^{gr}$ is $(\beta_0,\hat{\alpha_k})-(\alpha_{k},\hat{\beta_0})$.
					\item  If $grd(v)>grd(w)$, then the corresponding element in $S_2^{gr}$ is $(\alpha_{k},\hat{\beta_0})$.
					\item If $grd(v)<grd(w)$, then the corresponding element in $S_2^{gr}$ is $(\beta_0,\hat{\alpha_k})$.
				\end{itemize}

				\item $S_3^{gr}$: elements of the form $(\beta,C_v(\alpha)^{m(v)})$, where $\alpha,\beta$ are corresponding to arrows in the following Brauer graph, with $l(C_v(\alpha)^{m(v)})>m(w)$, or $l(C_v(\alpha)^{m(v)})=m(w)$ and $\beta^{m(w)}>C_v(\alpha)^{m(v)}$.
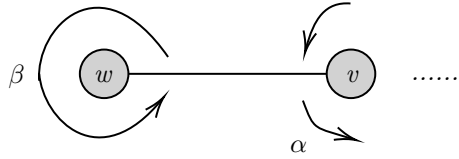
				\begin{figure}[H]
				\begin{center}
					\tikzset{every picture/.style={line width=0.75pt}} 
					\begin{tikzpicture}[x=0.75pt,y=0.75pt,yscale=-1,xscale=1]
						
						\draw    (100,108) -- (200,108) ;
						\draw  [fill={rgb, 255:red, 210; green, 210; blue, 210 }  ,fill opacity=1 ] (75.5,108) .. controls (75.5,101.23) and (80.98,95.75) .. (87.75,95.75) .. controls (94.52,95.75) and (100,101.23) .. (100,108) .. controls (100,114.77) and (94.52,120.25) .. (87.75,120.25) .. controls (80.98,120.25) and (75.5,114.77) .. (75.5,108) -- cycle ;
						\draw  [fill={rgb, 255:red, 210; green, 210; blue, 210 }  ,fill opacity=1 ] (200,108) .. controls (200,101.23) and (205.48,95.75) .. (212.25,95.75) .. controls (219.02,95.75) and (224.5,101.23) .. (224.5,108) .. controls (224.5,114.77) and (219.02,120.25) .. (212.25,120.25) .. controls (205.48,120.25) and (200,114.77) .. (200,108) -- cycle ;
						\draw    (120,99.5) .. controls (86,50.5) and (53,84.5) .. (55,111.5)(55,108) .. controls (56.98,131.27) and (85.42,159.92) .. (117.04,120.71) ;
						\draw [shift={(118,119.5)}, rotate = 127.97] [color={rgb, 255:red, 0; green, 0; blue, 0 }  ][line width=0.75]    (10.93,-3.29) .. controls (6.95,-1.4) and (3.31,-0.3) .. (0,0) .. controls (3.31,0.3) and (6.95,1.4) .. (10.93,3.29)   ;
						\draw    (212,72.5) .. controls (199.46,71.54) and (193.43,79.88) .. (189.43,95.74) ;
						\draw [shift={(189,97.5)}, rotate = 283.24] [color={rgb, 255:red, 0; green, 0; blue, 0 }  ][line width=0.75]    (10.93,-3.29) .. controls (6.95,-1.4) and (3.31,-0.3) .. (0,0) .. controls (3.31,0.3) and (6.95,1.4) .. (10.93,3.29)   ;
						\draw    (188,121.5) .. controls (193.85,137.1) and (193.05,134.64) .. (213.39,140.99) ;
						\draw [shift={(215,141.5)}, rotate = 197.65] [color={rgb, 255:red, 0; green, 0; blue, 0 }  ][line width=0.75]    (10.93,-3.29) .. controls (6.95,-1.4) and (3.31,-0.3) .. (0,0) .. controls (3.31,0.3) and (6.95,1.4) .. (10.93,3.29)   ;
						
						\draw (241,108) node [anchor=north west][inner sep=0.75pt]   [align=left] {......};
						\draw (81,105) node [anchor=north west][inner sep=0.75pt]   [align=left] {w};
						\draw (208,105) node [anchor=north west][inner sep=0.75pt]   [align=left] {v};
						\draw (38,102) node [anchor=north west][inner sep=0.75pt]   [align=left] {$\beta$};
						\draw (180,140) node [anchor=north west][inner sep=0.75pt]   [align=left] {$\alpha$};
					\end{tikzpicture}
				\end{center}
					\caption{The subgraph corresponding to some element in $S_3^{gr}$}
					\end{figure}
				
				\item $S_4^{gr}$: elements of the form $(\alpha,p)$ satisfying
				\begin{itemize}
					\item $l(p)>1$;
					\item $\psi_1^{gr(A)}(\alpha,p)=0$;
					\item there exists a special cycle $C_v$, such that $\alpha\mid C_v$, $p\mid C_v^{m(v)}$.
				\end{itemize}
				
				\item $S_5^{gr}$: the basis of the subspace $Im\psi_0^{gr(A)}$ generated by all the elements $(\alpha_1,\alpha_1 p)-(\alpha_0,p\alpha_0)$ where $p$ is a cycle in $Q$ and $\psi_1^{gr(A)}((\alpha_1,\alpha_1 p)) =\psi_1^{gr(A)}((\alpha_0,p\alpha_0))\neq 0$.
			\end{itemize}
			Furthermore, $\mathrm{Im}\psi_0^{gr(A)}$ is contained in $\langle S_1^{gr}\cup S_4^{gr}\cup S_5^{gr}\rangle$. The $k$-basis of $\mathrm{HH}^1(gr(A))$ induced from $\mathcal{B}_{L,Ker}^{gr(A)}$ will be denoted by $\mathcal{B}_L^{gr(A)}$.
		\end{lemma}
		
		\begin{proof}
			By definitions, it is straightforward to check that the elements in the set $S_1^{gr}\cup S_2^{gr}\cup S_3^{gr}\cup S_4^{gr}\cup S_5^{gr}$ are linearly independent and contained in $\mathrm{Ker}\psi_1^{gr(A)}$ (c.f. the proof of Theorem \ref{gen-set of A}). Denote the Gr{\"o}bner basis of $gr(A)$ by $\mathcal{G}^{gr(A)}$, the $i$-th type of relations in $kQ_G$ by $R_i$, $i=1,2,3$. We also denote the relations $\rho'$ in $I'$ induced by the first type by $R_1'$ (c.f. Lemma \ref{relations of grA}). Let $R_1\cap R_1'=R_1^0$, $R_1'\backslash R_1=R_1^1$, then $R_1^0\cup R_1^1\cup R_3\subseteq \mathcal{G}^{gr(A)}\subseteq R_1^0\cup R_1^1\cup R_2\cup R_3$.
			
			Consider $t=\sum_{i\in I}k_i(\alpha_{i},p_i)\in \mathrm{Ker}\psi_1^{gr(A)}$, and let $(\alpha,p)$ be a summand of $t$. By the graded structure discussed in Section 3 and Proposition \ref{l_{-1}}, we can assume $l(p)\geq 1$. 
			
			Firstly, we suppose that $\psi_1(\alpha,p)\neq 0$. That means there exists some $g\in\mathcal{G}^{gr(A)}$, such that $$\sum_{p'\in \mathrm{Supp}(g)}c_g(p')\cdot(\mathrm{Tip}(g),\pi(p'^{(\alpha,p)}))\neq 0.$$
			
			\begin{itemize}
				\item If $g\in\mathcal{G}^{gr(A)}\cap R_2$, obviously, $(g,\pi(g^{(\alpha,p)}))=0$.
				
				\item If $g\in R_1^1\subseteq\mathcal{G}^{gr(A)}$, then there exists $ v\in V$, such that $l(g)=grd(v)$, thus $(g,\pi(g^{(\alpha,p)}))=0$
				
				\item 	If there exists $ g\in R_1^0\subseteq\mathcal{G}^{gr(A)}$, such that there exists $ p_0\in \mathrm{Supp}(g)$, then $(\mathrm{Tip}(g),\pi(p_0^{(\alpha,p)}))\neq 0$. This implies $l(p)=1$, and consequently, it requires $\alpha=p$. By the form of $\psi_1^{gr(A)}((\alpha,\alpha))$, we know that $(\alpha,\alpha)$ must appear in a summand $t'$ of $t$ and $t'\in kS_1^{gr}$.
				
				\item 	If  there exists $ g\in R_3\subseteq\mathcal{G}^{gr(A)}$ such that  $(g,\pi(g^{(\alpha,p)}))\neq 0$,  then we analyze two possible cases based on whether there exists a vertex $v$ satisfying $C_v^{m(v)}\;|\;g^{(\alpha,p)}$ or not. 
				\begin{itemize}
					\item If $(\alpha, p)$ is a summand of some element in $S_2^{gr}$, there will be an element $t''$ in $kS_2^{gr}$, such that $t''$ is a summand of $t$.
					\item If $(\alpha, p)$ is not a summand of any element in $S_2^{gr}$, then we can just assume $g=\beta\alpha$, $(\beta\alpha,\pi(g^{(\alpha,p)})=(\beta\alpha,\beta p)$. Since $t\in \mathrm{Ker}\psi_1^{gr(A)}$, there exist some $k_1,k_2\in k$ and $(\beta,q)$, such that $k_1(\beta,q)-k_2(\alpha,p)$ is summand of $t$.  However, that means $k_2\beta p=k_1q\alpha$, and we will get $k_1=k_2,p=p_1\alpha,q=\beta p_1$. Since $gr(A)$ is also a special biserial algebra, $(\beta,q)-(\alpha,p)=\psi_0^{gr(A)}(e,p_1)$ for some $e\in Q_0$. That means $(\beta,q)-(\alpha,p)\in S_5^{gr}$.
				\end{itemize}   
			\end{itemize}
			
			Next we assume $\psi_1(\alpha,p)=0$. If $p\mid C_v^{m(v)}$ with $C_v$ being the special cycle which contains $\alpha$, then for the same reason as in Theorem \ref{gen-set of A}, $l(p)$ must bigger than $1$. This implies that $(\alpha,p)$ is contained in $S_4^{gr}$. We consider now the case in which $p\nmid C_v^{m(v)}$ with $p\mid C_w^{m(w)}$. Since there exists $\alpha\beta\in R_3$ with $\beta\mid C_w$, then $(\alpha\beta,\pi(\alpha\beta)^{(\alpha,p)})=0$. Hence there exists $ g\in R_2$ or $g\in R_1^1$, such that $g\mid p\beta$.
			\begin{itemize}
				\item If there exists $ g\in R_2$, such that $g\mid p\beta$, then $l(p\beta)\geq grd(w)+1$, which means $p$ is a cycle and $\alpha$ is a loop. This element is contained in $S_3^{gr}$.
				
				\item If there exists $ g\in R_1^1$, such that $g\mid p\beta$, then $grd(w)-1\leq l(p)\leq grd(w)$. This case is same as the discussion in the first case when $l(p)=grd(w)$. When $l(p)=grd(w)-1$, then $p$ is of the form of $\beta_m\cdots\beta_{1}\cdot C_w(\beta_{1})^{m(w)-1}$. Since $\alpha//p$, then $(\alpha,p)$ is contained in $S_2^{gr}$.
			\end{itemize}
			Therefore, each element $t\in Ker\psi_1^{gr(A)}$ can be represented by a linear combination of the elements in $\mathcal{S}^{gr}$.
		\end{proof}
		
		We also need a lemma to give some connection between $A$ and $gr(A)$.
		
		\begin{lemma}\label{1-1 map A-grA}
			Let $A$ be a BGA associated with a Brauer graph $G=(V,E)$ and $gr(A)$ the associated graded algebra of $A$. Then there is a canonical isomorphism of vector spaces from $A$ to $gr(A)$, which also induces an isomorphism from $\mathrm{Im}\psi_0^A$ to $\mathrm{Im}\psi_0^{gr(A)}$.
		\end{lemma}
		
		\begin{proof}
			For $A=kQ/I$ and $gr(A)=kQ/I'$, by the property of Gr\"{o}bner basis, under the same length-lexicographic order in $Q_{\geq 0}$, we can find a $k$-basis of $A$ (respectively, of $gr(A)$) in $Q_{\geq 0}$ if we fix the natural Gr{\"o}bner basis in $I$ (respectively, in $I'$). By Lemma \ref{relations of grA}, define a map $\phi$ from the $k$-basis of $A$ to the $k$-basis of $gr(A)$ by the following rules:
			\begin{itemize}
				\item If $v_L\stackrel{i}{-}v_S$ in $G$ is an unbalanced edge, then $C_{v_L}(\alpha)^{m(v_L)}-C_{v_S}(\beta)^{m(v_S)}\in\mathcal{G}$. Define $$\phi(C_{v_S}(\beta)^{m(v_S)})=C_{v_L}(\alpha)^{m(v_L)};$$
				
				\item Otherwise, $\phi$ is the identity morphism on the elements in the basis of $A$ which are different from above case.
			\end{itemize}
			Thus $\phi$ gives an isomorphism from $A$ to $gr(A)$.
			
			Now define the map $\hat{\phi}:Q_1//\mathcal{B}^A\rightarrow Q_1//\mathcal{B}^{gr(A)}$, $(\alpha,p)\mapsto(\alpha,\phi(p))$. Obviously, $\hat{\phi}$ is also an isomorphism. Moreover,  $\hat{\phi}|_{\mathrm{Im}\psi_0^A}$ and $\hat{\phi}^{-1}|_{\mathrm{Im}\psi_0^{gr(A)}}$ induce an isomorphism between $\mathrm{Im}\psi_0^A$ and $\mathrm{Im}\psi_0^{gr(A)}$.	
		\end{proof}
		
		Now we prove the main result of this subsection.
		
		\begin{theorem}\label{inj-map}
			Let $A$ be a BGA associated with a Brauer graph $G$, and $gr(A)$ the associated graded algebra of $A$. If $G\neq (v_S{-}v_L)$ with $m(v_L)>m(v_S)\geq 2$,  then there is a monomorphism $i$ from $\mathrm{HH}^1(A)$ to $\mathrm{HH}^1(gr(A))$ as Lie algebras.
		\end{theorem}
		
		\begin{proof}
			By Theorem \ref{gen-set of A} and Lemma \ref{gen-set of grA}, denote the $k$-bases of $\mathrm{Ker}\psi_1$, $\mathrm{Ker}\psi_1^{gr(A)}$ by the sets $\mathcal{B}_{L,Ker}=S_1\cup S_2\cup S_3\cup S_4\cup S_5$, $\mathcal{B}_{L,Ker}^{gr(A)}=S_1^{gr}\cup S_2^{gr}\cup S_3^{gr}\cup S_4^{gr}\cup S_5^{gr}$, respectively. By the correspondence in Lemma \ref{1-1 map A-grA}, we can choose $S_5$ and $S_5^{gr}$ such that $\hat{\phi}$ sends $S_5$ to $S_5^{gr}$.   Denote the morphism $i$ by:
			\begin{itemize}
				\item $i_1:S_1\rightarrow S_1^{gr}$, the natural embedding morphism, due to the basis of $\langle (\alpha,\alpha)|\alpha\in Q_1\rangle\cap \mathrm{Ker}\psi_1$ is the linear combinations of the basis of $\langle (\alpha,\alpha)|\alpha\in Q_1\rangle\cap \mathrm{Ker}\psi_1^{gr(A)}$.
				
				\item $i_2:S_2\rightarrow S_2^{gr}$, let $e=(\beta_0,\alpha_{k-1}\cdots\alpha_{0}\cdot C_v(\alpha_{0})^{m(v)-1})-(\alpha_{k},\beta_m\cdots\beta_{1}\cdot C_w(\beta_{1})^{m(w)-1})\in S_2$, then
				$$e\mapsto \left\{
				\begin{array}{*{3}{lll}}
					e& , & grd(v)=grd(w),\\
					-(\alpha_{k},\beta_m\cdots\beta_{1}\cdot C_w(\beta_{1})^{m(w)-1})& , & grd(v)>grd(w),\\
					(\beta_0,\alpha_{k-1}\cdots\alpha_{0}\cdot C_v(\alpha_{0})^{m(v)-1})&,&grd(v)<grd(w).
				\end{array}
				\right.$$
				
				\item $i_3:S_3\cup S_4\rightarrow S_3^{gr}\cup S_4^{gr}$, $(\alpha,p)\mapsto(\alpha,\phi(p))$, which $\phi$ is the one-to-one morphism between $A$ and $gr(A)$ in Lemma \ref{1-1 map A-grA}.
				
				\item $i_4:=\hat{\phi}|_{S_5}$.
			\end{itemize}
			Then $i=i_1\cup i_2\cup i_3\cup i_4$ is an injection from $\mathcal{B}_{L,Ker}$ to $\mathcal{B}_{L,Ker}^{gr(A)}$ by the definition of $S_i$ and $S_i^{gr}$, $i=1,2,3,4,5$. Moreover, $i_2,i_3,i_4$ are bijections. By Lemma \ref{1-1 map A-grA}, $i(\mathrm{Im}\psi_0^A)=\mathrm{Im}\psi_0^{gr(A)}$. Therefore, $i$ induces an injection from $\mathrm{HH}^1(A)$ to $\mathrm{HH}^1(gr(A))$. We will prove that $i$ is a monomorphism of Lie algebras.
			
			First of all, consider $\mathrm{ad}(r):=[r,-]$ with $r\in S_1$, then by the definition of $i$, $i(r)\in S_1^{gr}$. Since the restriction $i|_{S_1}$ is the natural embedding morphism, we have $r=i(r)$ in the vector space $k(Q_1//Q_1)$. Since $L_{00},L_{00}^{gr(A)}$ are solvable Lie ideals of $A,gr(A)$, respectively, it is easy to check that $\mathrm{ad}(r)|_{S_1}$ and $\mathrm{ad}(i(r))|_{S_1^{gr}}$ are zero morphisms. Thus $[i(r_1),i(r_2)]=i([r_1,r_2])=0$, $r_1,r_2\in S_1$.
			For $r'\in S_2$, let $$r'=(\beta_0,\alpha_{k-1}\cdots\alpha_{0}\cdot C_v(\alpha_{0})^{m(v)-1})-(\alpha_{k},\beta_m\cdots\beta_{1}\cdot C_w(\beta_{1})^{m(w)-1})$$ where $C_v$ and $C_w$ are different special cycles. For simplicity, let $\hat{\beta_0}=(\beta_0,\alpha_{k-1}\cdots\alpha_{0}\cdot C_v(\alpha_{0})^{m(v)-1})$, $\hat{\alpha_k}=(\alpha_{k},\beta_m\cdots\beta_{1}\cdot C_w(\beta_{1})^{m(w)-1})$, then $r'=\hat{\beta_0}-\hat{\alpha_k}$. Since $r\in S_1$, $r=(\alpha_{0},\alpha_{0})-(\alpha_{i},\alpha_{i})$ or $r=\sum_{v\in V,\alpha_v\mid C_v}k_v(\alpha_v,\alpha_v)$ with $k_v=\prod_{v'\in V,v'\neq v}m(v')$ by Lemma \ref{dimL00 of A}. Moreover,
			
			\begin{itemize}
				\item If $r=(\alpha_{0},\alpha_{0})-(\alpha_{i},\alpha_{i})$, then
				$$[r,r']= \left\{
				\begin{array}{*{3}{lll}}
					(m(v)-1)(\hat{\beta_0}-\hat{\beta_0})=0& , &i\neq k \\
					\{m(v)-(m(v)-1)\}\hat{\beta_0}-\hat{\alpha_k}=r'& , & i=k
				\end{array}
				\right.$$
				When $r'=i(r')$ in $k(Q_1//Q_{\geq 0})$, $i([r,r'])=[i(r),i(r')]$ obviously. When $r'\neq i(r')$, without loss of generality, let $i(r')=\hat{\beta_0}$. Then
				$$[i(r),i(r')]= \left\{
				\begin{array}{*{3}{lll}}
					(m(v)-1)(\hat{\beta_0}-\hat{\beta_0})=0& , &i\neq k \\
					\{m(v)-(m(v)-1)\}\hat{\beta_0}=i(r')& , & i=k
				\end{array}
				\right.$$
				Thus $i([r,r'])=[i(r),i(r')]$.
				
				\item If $r=(\beta_0,\beta_0)-(\beta_j,\beta_j)$, then
				$$[r,r']= \left\{
				\begin{array}{*{3}{lll}}
					0& , &j\neq m \\
					r'& , & j=m
				\end{array}
				\right.$$
				and it can be checked that $i([r,r'])=[i(r),i(r')]$.
				
				\item If $r=\sum_{v\in V,\alpha_v\mid C_v}k_v(\alpha_v,\alpha_v)$ with $k_v=\prod_{v'\in V,v'\neq v}m(v')$, when $\beta_0,\alpha_k$ are not loops, we have that
				$$[r,r']=\prod_{v\in V}m(v)\hat{\beta_0}-(\prod_{v\in V}m(v)-k_w)\hat{\alpha_k}-k_w\hat{\beta_0}=(\prod_{v\in V}m(v)-k_w)r'.$$
				If $\alpha_k$ is a loop, by the definition of the elements in $S_2$, then $\beta_0$ is also a loop, and
				$$[r,r']=(m(v)-1)k_v\hat{\beta_0}-(m(w)-1)k_w\hat{\alpha_k}-k_w\hat{\beta_0}+k_v\hat{\alpha_k}=(\prod_{v\in V}m(v)-k_v-k_w)r'.$$
				Thus $i([r,r'])=[i(r),i(r')]=c\cdot i(r'), c\in k$.
				
				\item Otherwise, if $r\in S_1$ and $r$ is not of the forms as above, then $[r,r']=[i(r),i(r')]=0$.
			\end{itemize}
			Therefore, $i([r,r'])=[i(r),i(r')]$ when $r\in S_1, r'\in S_2$.
			
			If $r'\in S_3\cup S_4$, then $r'=(\alpha,p)$ with $p\mid C_w,\alpha\mid C_v$, $v,w\in V$.  Since the one-to-one map $\phi$ only changes the cycles around the unbalanced edges, we only need to consider $r'=(\beta,C_v(\alpha)^{m(v)})$ with $i(r')=(\beta,C_w(\beta)^{m(w)})$. However, because of $\beta//C_v^{m(v)}$, $\beta$ is a loop.
			
			\begin{itemize}
				\item If $r=(\alpha_{0},\alpha_{0})-(\alpha_{i},\alpha_{i})$, then $[r,r']=m(v)(r'-r')=0$, $[i(r),i(r')]=0$.
				\item If $r=\sum_{v\in V,\alpha_v\mid C_v}k_v(\alpha_v,\alpha_v)$ with $k_v=\prod_{v'\in V,v'\neq v}m(v')$, then $[r,r']=(\prod_{v\in V}m(v)-k_w)r'$, $[i(r),i(r')]=(\prod_{v\in V}m(v)-k_w)i(r')$.
			\end{itemize}
			Therefore, we have proved that $i([r,r'])=[i(r),i(r')]$ when $r\in S_1, r'\in S$.
			
			Secondly, consider $ad(r)$ with $r\in S_2$, then by the definition of $i$, we have $i(r)\in S_2^{gr}$. Let $$r=(\beta_0,\alpha_{k-1}\cdots\alpha_{0}\cdot C_v(\alpha_{0})^{m(v)-1})-(\alpha_{k},\beta_m\cdots\beta_{1}\cdot C_w(\beta_{1})^{m(w)-1})\in S_2,$$ $$r'=(\beta_t,\gamma_{n-1}\cdots\gamma_0\cdot C_u(\gamma_0)^{m(u)-1})-(\gamma_n,\beta_{t-1}\cdots\beta_{t+1}\cdot C_w(\beta_{t+1})^{m(w)-1})\in S_2.$$ That means
			$$[r,r']=-(\alpha_k,\pi((\beta_m\cdots\beta_{1}\cdot C_w(\beta_{1})^{m(w)-1})^{(\beta_t,\gamma_{n-1}\cdots\gamma_0\cdot C_u(\gamma_0)^{m(u)-1})}))$$$$+(\gamma_n,\pi((\beta_{t-1}\cdots\beta_{t+1}\cdot C_w(\beta_{t+1})^{m(w)-1})^{(\beta_0,\alpha_{k-1}\cdots\alpha_{0}\cdot C_v(\alpha_{0})^{m(v)-1})}))$$
			Since $\gamma_n,\beta_t$ and $\beta_0,\alpha_k$ are arrows induced by subgraphs as in Figure \ref{S2-subgraph}, this implies that $grd(w)>2$ and $[r,r']=0$. It can be checked $[i(r),i(r')]=0$ in the same way.
			
			Let $r'\in S_3\cup S_4$. It is enough to consider the following two cases:
			\begin{itemize}
				\item Let $r'=(\beta,C_v(\alpha)^{m(v)})$ where $\beta$ is a loop and $\phi(C_v(\alpha)^{m(v)})=\beta^{m(v)}$. Since $G\neq (v_S{-}v_L)$, then $r=(\beta_0,\alpha_{k-1}\cdots\alpha_{0}\cdot C_v(\alpha_{0})^{m(v)-1})-(\alpha_{k},\beta_m\cdots\beta_{1}\cdot C_w(\beta_{1})^{m(w)-1})$ where $C_v$ and $C_w$ are not loops, thus $\beta\nmid C_v$ and $\beta\nmid C_w$. Therefore, $[r,r']=0$ and $[i(r),i(r')]=0$.
				\item Let $r'=(\beta,\beta^{m(v)})$ where $\beta$ is a loop and $\phi(\beta^{m(v)})=C_v(\alpha)^{m(v)}$. Then, as in the first case, we have $[r,r']=0$ and $[i(r),i(r')]=0$.
			\end{itemize}
			Therefore, $i([r,r'])=[i(r),i(r')]$ when $r\in S_2, r'\in S$.
			
			Now let us consider $r,r'\in S_3\cup S_4$. Then $r,r'$ are of the form of $(\alpha,p)$ with $l(p)>1$. Without loss of generality, it is enough to consider $r$ in the following forms: $r=(\beta,C_v(\alpha)^{m(v)})$ or $r=(\beta,\beta^{m(v)})$. Since $r'=(\alpha,p)$ with $l(p)>1$, we have $[r,r']=[i(r),i(r')]=0$.
			
			To sum it up, for all $r,r'\in \mathcal{S}$, we have $i([r,r'])=[i(r),i(r')]$. Hence $i$ is a monomorphism of Lie algebras from $\mathrm{HH}^1(A)$ to $\mathrm{HH}^1(gr(A))$.
		\end{proof}
		
		\begin{remark}
			Let $A$ be a BGA with its corresponding Brauer graph $G\neq (v_S{-}v_L)$ with $m(v_L)>m(v_S)\geq 2$. By Theorem \ref{inj-map}, if $\mathrm{HH}^1(gr(A))$ is solvable, so is $\mathrm{HH}^1(A)$.
		\end{remark}
		
		Although the injection above is not always a monomorphism of Lie algebras, it is enough to help us compute the difference between the dimensions of $\mathrm{HH}^1(A)$ and $\mathrm{HH}^1(gr(A))$.
		
		\begin{corollary}\label{dim(A-grA)}
			Let $A$ be a BGA associated with a Brauer graph $G=(V,E)$ and $gr(A)$ the associated graded algebra of $A$. Then $\mathrm{dim}_k\mathrm{HH}^1(gr(A))-\mathrm{dim}_k\mathrm{HH}^1(A)=|\Gamma_G|-1$.
		\end{corollary}
		
		\begin{proof}
			By the proof of the theorem above, $i$ is an injection from the basis of $\mathrm{Ker}\psi_1$ to the basis of $\mathrm{Ker}\psi_1^{gr(A)}$. Moreover, $i$ is a bijection between $S_2\cup S_3\cup S_4\cup S_5$ and $S_2^{gr}\cup S_3^{gr}\cup S_4^{gr}\cup S_5^{gr}$. By Lemma \ref{1-1 map A-grA}, we have $i(\mathrm{Im}\psi_0^A)=\mathrm{Im}\psi_0^{gr(A)}$. Therefore, by Lemma \ref{dimL00 of A} and Lemma \ref{dimL00 of grA}, we have
			$$
			\begin{array}{*{3}{lll}}
				\mathrm{dim}_k\mathrm{HH}^1(gr(A))-\mathrm{dim}_k\mathrm{HH}^1(A) & =&\mathrm{dim}_k\mathrm{Ker}\psi_1^{gr(A)}-\mathrm{dim}_k\mathrm{Im}\psi_0^{gr(A)}-\mathrm{dim}_k\mathrm{Ker}\psi_1^A+\mathrm{dim}_k\mathrm{Im}\psi_0^A\\
				& =&|S_1^{gr}|-|S_1|\\
				&=& \mathrm{dim}_kL_{00}^{gr(A)}-\mathrm{dim}_kL_{00}\\
				&=&|\Gamma_G|-1.
			\end{array}
			$$
		\end{proof}
		
		Now let us check some examples to verify the results we obtained above.
		
		\begin{example}Consider the Brauer graph $G$ in following form.
			$$
			\begin{tikzcd}
				G: & \bullet \arrow[r, no head] & \bullet \arrow[r, no head] & {\bullet[3]} &  & Q_G: & \bullet \arrow[r, "\alpha_1", shift left] & \bullet \arrow[l, "\alpha_2", shift left] \arrow["\beta"', loop, distance=2em, in=35, out=325]
			\end{tikzcd}$$
			The corresponding BGA of Brauer graph $G$ is $$A=kQ_G/\langle \beta^3-\alpha_{1}\alpha_{2},\alpha_{1}\alpha_{2}\alpha_{1},\alpha_{2}\alpha_{1}\alpha_{2},\beta\alpha_{1},\alpha_{2}\beta\rangle,$$ and the associated graded algebra of $A$ is $$gr(A)=kQ_G/\langle\beta^4,\alpha_{1}\alpha_{2},\beta\alpha_{1},\alpha_{2}\beta\rangle.$$ By the parallel paths method,
			$$\mathrm{HH}^1(A)=k\{3(\alpha_{1},\alpha_{1})+(\beta,\beta),\;(\beta,\beta^2),\;(\beta,\alpha_{1}\alpha_{2})\}$$
			$$\mathrm{HH}^1(gr(A))=k\{(\alpha_{1},\alpha_{1}),\;(\beta,\beta),\;(\beta,\beta^2),\;(\beta,\beta^3)\}$$
			Then $\mathrm{dim}_k\mathrm{HH}^1(gr(A))-\mathrm{dim}_k\mathrm{HH}^1(A)=1$. And the monomorphism $i$ is given by:
			$$	\begin{array}{*{5}{lllll}}
				i& :&3(\alpha_{1},\alpha_{1})+(\beta,\beta)&\mapsto&3(\alpha_{1},\alpha_{1})+(\beta,\beta),\\
				&&(\beta,\beta^2)&\mapsto&(\beta,\beta^2),\\
				& &(\beta,\alpha_{1}\alpha_{2})&\mapsto&(\beta,\beta^3).
			\end{array}$$
		\end{example}

		\begin{example}\label{not injection} This is an example of the case which is excluded in Theorem \ref{inj-map}: 
			$$
			\begin{tikzcd}
				G: & \bullet[2] \arrow[r, no head] & \bullet[3] &  & Q_G:& & \bullet \arrow["y"', loop, distance=2em, in=35, out=325] \arrow["x"', loop, distance=2em, in=215, out=145]
			\end{tikzcd}$$
			The corresponding BGA of Brauer graph $G$ is $$A=k\langle x,y\rangle /\langle y^3-x^2,x^3,xy,yx\rangle,$$ and the associated graded algebra of $A$ is $$gr(A)=k\langle x,y\rangle/\langle y^4,x^2,xy,yx\rangle.$$ By the parallel paths method,
			$$\mathrm{HH}^1(A)=k\{3(x,x)+2(y,y),\;(x,x^2),\;(y,y^2),\;(y,x^2),\;(x,y^2)-(y,x)\}$$
			$$\mathrm{HH}^1(gr(A))=k\{(x,x),\;(y,y),\;(x,y^3),\;(y,y^2),\;(y,y^3),\;(y,x)\}$$
			Then $\mathrm{dim}_k\mathrm{HH}^1(gr(A))-\mathrm{dim}_k\mathrm{HH}^1(A)=1$, but since $i((y,y^2))=(y,y^2)$, $i((x,y^2)-(y,x))=-(y,x)$, $i((x,x^2))=(x,y^3)$, we have
			$$i([(y,y^2),(x,y^2)-(y,x)])=i(2(x,x^2))=2(x,y^3),$$
			$$[i((y,y^2)),i((x,y^2)-(y,x))]=[(y,y^2),-(y,x)]=0,$$
			$i$ is not a morphism of Lie algebras. If there is a monomorphism from $\mathrm{HH}^1(A)$ to $\mathrm{HH}^1(gr(A))$, then $\mathrm{HH}^1(A)$ can be regarded as a Lie subalgebra of $\mathrm{HH}^1(gr(A))$. That makes the derived Lie subalgebra $\mathrm{HH}^1(A)^{(2)}$ a subspace of $\mathrm{HH}^1(gr(A))^{(2)}$. However,
			$$\mathrm{HH}^1(A)^{(2)}=k\{(y,x^2),(x,x^2)\},$$
			$$\mathrm{HH}^1(gr(A))^{(2)}=k\{(y,y^3)\}.$$
			Thus $\mathrm{dim}_k(\mathrm{HH}^1(A)^{(2)})=2>1=\mathrm{dim}_k(\mathrm{HH}^1(gr(A))^{(2)})$, a contradiction.
		\end{example}
		
		\subsection{A discussion related to $\mathrm{Out}(A)^{\circ}$}
		
		In this subsection, we assume that $k$ is an algebraically closed field of characteristic $0$. 
		
		Let $A$ be a BGA with its corresponding Brauer graph $G=(V,E)$. Then for the vector space $L_{00}$ of $A$ defined in Section 3.3, by Lemma \ref{dimL00 of A}, $$\mathrm{dim}_kL_{00}=|E|-|V|+2.$$ It is interesting to note that Antipov and Zvonareva have obtained the same number in \cite{AZ}. Denote by $T(A)$ the maximal torus of the identity component $\mathrm{Out}(A)^{\circ}$ of the group of the outer automorphisms for an algebra $A$.
		By \cite[Theorem 1.1]{AZ}, if the Brauer graph $G$ has at least two edges and is not a caterpillar (c.f. \cite[Section 3]{AZ}), then the rank of $T(A)$ is $|E|-|V|+2$.
		
		Now we use a construction by Briggs and Rubio y Degrassi in \cite{BR} to discuss the connection between $T(A)$ and $L_{00}$. Recall some definitions.
		
		\begin{definition}$(\cite[~Definition~2.2]{BR})$
			Let $A$ be a finite dimensional algebra over an algebraically closed field $k$. An element $f\in \mathrm{HH}^1(A)$ is called diagonalizable if it can be represented by a derivation $d\in \mathrm{Der}(A)$ which acts diagonalizably on $A$, with respect to some $k$-linear basis of $A$. More generally, we say that a subspace $S\subseteq \mathrm{HH}^1(A)$ is diagonalizable if its elements can be represented by derivations which are simultaneously diagonalizable on $A$. Note that $S$ is automatically a Lie subalgebra of $\mathrm{HH}^1(A)$ since $[S,S]=0$. The maximal diagonalizable subalgebras are by definition diagonalizable subalgebras which are maximal with respect to the inclusion. 
		\end{definition}
		
		By \cite[Proposition~3.1]{Strametz}, the Lie algebra of the identity component of the outer automorphism group of $A$ is isomorphic to $\mathrm{HH}^1(A)$. Thus the rank of the maximal torus of $\mathrm{Out}(A)^{\circ}$ is the maximal toral rank of $\mathrm{HH}^1(A)$, and by \cite[Proposition 2.3]{BR}, it is equal to the dimension of the maximal diagonalizable subalgebra of $\mathrm{HH}^1(A)$. Obviously, $L_{00}$ is a diagonalizable subalgebra of $\mathrm{HH}^1(A)$, we will prove that it is maximal for any BGA and its associated graded algebra.
		
		\begin{proposition}\label{diag}
			Let $A$ be a BGA and $gr(A)$ the associated graded algebra of $A$. Then $L_{00}$ (respectively, $L_{00}^{gr(A)}$) is a maximal diagonalizable subalgebra of $\mathrm{HH}^1(A)$ (respectively, $\mathrm{HH}^1(gr(A))$).
		\end{proposition}
		
		\begin{proof}
			By the basis of $\mathrm{Ker}\psi_1$ given in Theorem \ref{gen-set of A} (respectively, $\mathrm{Ker}\psi_1^{gr(A)}$ in Lemma \ref{gen-set of grA}),  every element in $\mathrm{HH}^1(A)$ (respectively, $\mathrm{HH}^1(gr(A))$) can be presented by an element in $\langle S_1\cup S_2\cup S_3\cup S_4\rangle$ (respectively, $\langle S_1^{gr}\cup S_2^{gr}\cup S_3^{gr}\cup S_4^{gr}\rangle$). Note that the elements in $S_1$ (respectively, $S_1^{gr}$) can be regarded as a generating set of $L_{00}$ (respectively, $L_{00}^{gr(A)}$). For every element $x\in S_i$ with $i=2,3,4$ and $r\in kS_1$, we have $[x,r]=cx$, $c\in k$. Therefore, it is sufficient to show that for every element $x\in S_i$, $i=2,3,4$, there exists an element $r\in kS_1$, such that $[x,r]\neq 0$. We only prove the statement for the case of BGA. The case of the associated graded algebra $gr(A)$ can be checked in the same way.
			
			Let $x=(\beta_0,\alpha_{k-1}\cdots\alpha_{0}\cdot C_v(\alpha_{0})^{m(v)-1})-(\alpha_{k},\beta_m\cdots\beta_{1}\cdot C_w(\beta_{1})^{m(w)-1})\in S_2$. If $C_v$ and $C_w$ are different special cycles, we can apply the same reasoning as in the proof of Theorem \ref{inj-map}. If $C_v$ and $C_w$ are the same special cycles, by the definition of $S_2$, we have $l(C_v)\geq 4$. Therefore, we can choose a nonzero element $r=(\beta_{0},\beta_{0})-(\beta,\beta)\in kS_1$ with $\beta\neq\alpha_{k}$, 
			$[x,r]=x\neq 0$.
			
			For $x\in S_3$, the statement follows from the proof of Theorem \ref{inj-map}.
			
			Let $x=(\alpha,p)\in S_4$ with $\alpha\mid C_v$, we have the following two cases to discuss.
			
			\textit{Case 1}. If $\alpha\nmid p$, choose 
			$r=\sum_{v\in V,\alpha_v\mid C_v}k_v(\alpha_v,\alpha_v)$ with $k_v=\prod_{v'\in V,v'\neq v}m(v')$ and $\alpha_{v}=\alpha$, then $r\in kS_1$ and $[x,r]=k_vr\neq 0$.
			
			\textit{Case 2}. Assume $\alpha\mid p$. 
			\begin{itemize}
				\item If $p$ contains $\alpha$ more than one time, assume the number of times of $\alpha$ appear in $p$ is $n$ with $n\in\mathbb{N}$ and $n\geq 2$. Choose 
				$r=\sum_{v\in V,\alpha_v\mid C_v}k_v(\alpha_v,\alpha_v)$ with $k_v=\prod_{v'\in V,v'\neq v}m(v')$ and $\alpha_{v}=\alpha$, then $r\in kS_1$ and $[x,r]=(1-n)k_vr\neq 0$.
				
				\item If $p$ contains $\alpha$ just once, there exist another $\alpha'\mid C_v$ and $p$ contains $\alpha'$ $n$ times. Choose 
				$r=\sum_{v\in V,\alpha_v\mid C_v}k_v(\alpha_v,\alpha_v)$ with $k_v=\prod_{v'\in V,v'\neq v}m(v')$ and $\alpha_{v}=\alpha'$, then $r\in kS_1$ and in this case, $[x,r]=-nk_vr\neq 0$.
			\end{itemize}
			To sum up, for every element $x\in S_i$, $i=2,3,4$, there exists an element $r\in kS_1$, such that $[x,r]\neq 0$. Thus $L_{00}$ is the maximal diagonalizable subalgebra of $\mathrm{HH}^1(A)$.
		\end{proof}
		
		Since the dimension of a maximal torus of an algebraic group $G$ is called the rank of $G$ (see for example \cite[~Section~7.2.1]{Spr}), we can rewrite Corollary \ref{dim(A-grA)} as follows.
		
		\begin{corollary}\label{diff2}
			The difference between the dimension of $\mathrm{HH}^1(A)$ and of $\mathrm{HH}^1(gr(A))$ is equal to the difference between the rank of $\mathrm{Out}(A)^{\circ}$ and of $\mathrm{Out}(gr(A))^{\circ}$. In particular, it is also equal to the difference between the dimensions of the corresponding maximal dual fundamental groups.
		\end{corollary}
		
		\begin{proof}
			By Corollary \ref{dim(A-grA)}, Proposition \ref{diag} and \cite[Proposition~3.1]{Strametz}, the difference between the dimension of $\mathrm{HH}^1(A)$ and of $\mathrm{HH}^1(gr(A))$ is equal to the difference between the rank of $\mathrm{Out}(A)^{\circ}$ and of $\mathrm{Out}(gr(A))^{\circ}$.
			
		    By \cite[~Corollary~4.3]{BR}, for any finite dimensional algebra $\Lambda$, the maximal torus rank of $\mathrm{HH}^1(\Lambda)$ is equal to the maximal dimension of the corresponding dual fundamental group of $\Lambda$. Therefore, the difference between the dimension of $\mathrm{HH}^1(A)$ and of $\mathrm{HH}^1(gr(A))$ is equal to the difference between their maximal dimensions of the corresponding dual fundamental groups.
		\end{proof}

		\section*{Acknowledgments}
		We are very grateful to Lleonard Rubio y Degrassi for sending us the valuable comments and inspiring suggestions on the first version of this paper on arXiv when we communicated with him by emails. In particular, Section 5.3 is totally based on his comments. We would also like to thank Andrea Solotar for helpful comments on our paper. Finally, we would like to thank the referee for careful reading and many helpful suggestions to improve our presentations.


	\end{document}